\documentclass[12pt,reqno]{amsart}
\usepackage[headings]{fullpage}
\usepackage{amsmath,amssymb,amsthm,bbm}
\usepackage[bookmarks=true,%
colorlinks=true,%
linkcolor=blue,%
citecolor=blue,%
filecolor=blue,%
menucolor=blue,%
urlcolor=blue,%
breaklinks=true]{hyperref}
\usepackage{graphicx,subcaption,url}
\usepackage{fullpage,comment}
\usepackage[shortlabels]{enumitem}
\setlist{font=\normalfont}
\usepackage{mathrsfs,mathtools}
\usepackage{accents}
\usepackage{tikz,tikz-cd}
\usetikzlibrary{knots,hobby,math,decorations.markings}
\usepackage{ifthen}
\allowdisplaybreaks

\makeatletter
\renewcommand*{\fps@figure}{htpb!}
\makeatother

\newtheorem{theorem}{Theorem}[section]
\theoremstyle{definition}
\newtheorem{proposition}[theorem]{Proposition}
\newtheorem{lemma}[theorem]{Lemma}

\newtheorem{remark}[theorem]{Remark}
\newtheorem{corollary}[theorem]{Corollary}

\DeclareMathOperator{\tr}{tr}
\DeclareMathOperator{\qtr}{\widehat{tr}}
\DeclareMathOperator{\ttr}{\widetilde{tr}}

\newcommand{\term}[1]{\textit{#1}}

\newcommand{\cx}{\mathbb{C}}
\newcommand{\nats}{\mathbb{N}}
\newcommand{\ints}{\mathbb{Z}}

\newcommand{\Runiv}{R_{\mathrm{univ}}}
\newcommand{\iunit}{\mathbf{i}}

\newcommand{\SL}{\mathrm{SL}}
\newcommand{\PSL}{\mathrm{PSL}}
\newcommand{\PGL}{\mathrm{PGL}}
\newcommand{\SO}{\mathrm{SO}}
\newcommand{\Oq}{\mathcal{O}_{q^2}(\SL_2)}

\newcommand{\frameM}{\mathrm{Fr}M}
\newcommand{\frameX}{X^{\mathrm{Fr}}_M}

\newcommand{\skein}{\mathcal{S}}
\newcommand{\skeinrd}{\skein^{\mathrm{rd}}}
\newcommand{\skeincr}{\skein^{\mathrm{cr}}}

\newcommand{\surface}{\Sigma}
\newcommand{\surclose}{\overline{\Sigma}}
\newcommand{\marked}{\mathcal{P}}
\newcommand{\heeg}{\mathcal{H}}
\newcommand{\thicken}[1]{\widetilde{#1}}
\newcommand{\ori}{\mathfrak{o}}
\newcommand{\cut}{\Theta}

\newcommand{\bigon}{\mathbb{B}}
\newcommand{\annulus}{\mathbb{A}}
\newcommand{\lantern}{\mathbb{L}}

\newcommand{\qtorus}{\mathbb{T}}
\newcommand{\qglue}{\widehat{\mathcal{G}}}
\newcommand{\qvar}[1]{\hat{#1}}
\newcommand{\qz}{\qvar{z}}
\newcommand{\qy}{\qvar{y}}

\newcommand{\vexp}[1]{\mathbf{#1}}

\newcommand{\ideal}[1]{\langle#1\rangle}
\newcommand{\lideal}[1]{\prescript{}{L\!}{\ideal{#1}}}
\newcommand{\quotbyL}[1]{/\!\!\lideal{#1}}
\newcommand{\rideal}[1]{\ideal{#1}^{}_{\!R}}

\newcommand{\triang}{\mathcal{T}}
\newcommand{\dtrunc}{\mathfrak{T}}

\newcommand{\embed}{\hookrightarrow}
\newcommand{\onto}{\twoheadrightarrow}

\tikzset{knot diagram/every knot diagram/.style={
    background color=gray!20,clip width=5,end tolerance=5pt,clip radius=0.1cm}}
\newcommand{\stsize}{\scriptsize}
\newenvironment{linkdiag}[1][0.5]{\mathop{}\!
  \begin{tikzpicture}[scale=0.8,baseline=(ref.base)]
    \node (ref) at (0.5,{#1}){\phantom{$-$}};}{\end{tikzpicture}\!\mathop{}}
\tikzset{-o-/.code 2 args={
    \pgfkeysalso{decoration={markings,mark=at position #1 with {\arrow{#2}}},
      postaction={decorate}}
}}
\newcommand{\stnode}[1]{node[inner sep=1pt,right]{\stsize$#1$}}
\newcommand{\stnodel}[1]{node[inner sep=1pt,left]{\stsize$#1$}}

\newenvironment{anndiag}[3][]{
\begin{linkdiag}[#3/2]
\fill[gray!20] (0,0) rectangle (#2,#3);
\draw[dashed] (0,0) -- (#2,0) (0,#3) -- (#2,#3);
\draw (#2,0) -- (#2,#3);
\ifthenelse{\equal{#1}{u}}{\foreach \y in {0,4/11,8/11,1}}{\foreach \y in {0,1/3,2/3,1}}
\draw[fill=white] (#2,{\y*#3})circle(0.07);
}{\end{linkdiag}}

\newcommand{\relempty}[1][0]{
\begin{linkdiag}
\fill[gray!20] (0,-#1)rectangle(1,{1+#1});
\draw (1,-#1)--(1,{1+#1});
\ifthenelse{\equal{#1}{0}}{}{\draw[fill=white] (1,0.5)circle(0.07);}
\end{linkdiag}\,
}
\newcommand{\relcross}[2]{
\begin{linkdiag}
\fill[gray!20] (0,0)rectangle(1,1);
\draw[-stealth] (1,0)--(1,1);
\begin{knot}
\strand[thick] (0,0.3)..controls +(0.5,0) and +(-0.5,0)..(1,0.7);
\strand[thick] (0,0.7)..controls +(0.5,0) and +(-0.5,0)..(1,0.3);
\end{knot}
\draw (1,0.7)\stnode{#1} (1,0.3)\stnode{#2};
\end{linkdiag}
}
\newcommand{\relarc}[3][->]{
\begin{linkdiag}
\fill[gray!20] (0,0)rectangle(1,1);\draw[#1] (1,0)--(1,1);
\draw[thick] (1,0.7)..controls(0.2,0.7) and (0.2,0.3)..(1,0.3);
\draw (1,0.7)\stnode{#2} (1,0.3)\stnode{#3};
\end{linkdiag}
}
\newcommand{\relarcl}[3][->]{
\begin{linkdiag}
\fill[gray!20] (0,0)rectangle(1,1);\draw[#1] (0,0)--(0,1);
\draw[thick] (0,0.7)..controls(0.8,0.7) and (0.8,0.3)..(0,0.3);
\draw (0,0.7)\stnodel{#2} (0,0.3)\stnodel{#3};
\end{linkdiag}
}
\newcommand{\relacross}[3][->]{
\begin{linkdiag}
\fill[gray!20] (0,0)rectangle(1,1);\draw[#1] (1,0)--(1,1);
\draw[thick] (0,0.7)--(1,0.7) (0,0.3)--(1,0.3);
\draw (1,0.7)\stnode{#2} (1,0.3)\stnode{#3};
\end{linkdiag}
}
\newcommand{\relcorner}[3][]{
\begin{linkdiag}
\fill[gray!20] (0,-0.1)rectangle(1,1.1);\draw[#1] (1,-0.1)--(1,1.1);
\draw[thick] (1,0.8)..controls(0.2,0.8) and (0.2,0.2)..(1,0.2);
\draw[fill=white] (1,0.5)circle(0.07);
\draw (1,0.8)\stnode{#2} (1,0.2)\stnode{#3};
\end{linkdiag}
}
\newcommand{\relbottom}[1]{
\begin{linkdiag}
\fill[gray!20] (0,0)rectangle(1,1);\draw (1,0)--(1,1);
\draw[thick] (0,0.2)--(1,0.2);
\draw[fill=white] (1,0.5)circle(0.07);
\draw (1,0.2)\stnode{#1};
\end{linkdiag}
}
\newcommand{\reltwup}[1]{
\begin{linkdiag}
\fill[gray!20] (0,0)rectangle(1,1);\draw (1,0)--(1,1);
\draw[thick] (0,0.2)..controls (0.5,0.2) and (0.5,0.8)..(1,0.8);
\draw[fill=white] (1,0.5)circle(0.07);
\draw (1,0.8)\stnode{#1};
\end{linkdiag}
}
\newcommand{\relbottoml}[1]{
\begin{linkdiag}
\fill[gray!20] (0,0)rectangle(1,1);\draw (0,0)--(0,1);
\draw[thick] (0,0.2)--(1,0.2);
\draw[fill=white] (0,0.5)circle(0.07);
\draw (0,0.2)\stnodel{ #1};
\end{linkdiag}
}
\newcommand{\reltwupl}[1]{
\begin{linkdiag}
\fill[gray!20] (0,0)rectangle(1,1);\draw (0,0)--(0,1);
\draw[thick] (1,0.2)..controls (0.5,0.2) and (0.5,0.8)..(0,0.8);
\draw[fill=white] (0,0.5)circle(0.07);
\draw (0,0.8)\stnodel{ #1};
\end{linkdiag}
}

\def\BZ{\mathbbm Z}

\def\BC{\mathbbm C}

\def\calT{\mathcal T}

\def\calS{\mathcal S}

\def\s{\sigma}

\def\longto{\longrightarrow}

\def\a{\alpha}
\def\b{\beta}

\def\d{\delta}

\def\be{\begin{equation}}
\def\ee{\end{equation}}

\def\SO{\mathrm{SO}}

\newcommand{\cutred}{\cut_{\mathrm{red}}}

\makeatletter

\renewcommand\thepart{\@Roman\c@part}%
\renewcommand\part{%
\if@noskipsec \leavevmode \fi
\par
\addvspace{6.7ex}%
\@afterindentfalse
\secdef\@part\@spart}
\def\@part[#1]#2{%
\ifnum \c@secnumdepth >\m@ne
\refstepcounter{part}%
\addcontentsline{toc}{part}{Part~\thepart.\ #1}%
\else
\addcontentsline{toc}{part}{#1}%
\fi
{\parindent \z@ \raggedright
\interlinepenalty \@M
\normalfont
\ifnum \c@secnumdepth >\m@ne
\centering\large\scshape \partname~\thepart.%
\hspace{1ex}%
\fi%
\large\scshape #2%
\markboth{}{}\par}%
\nobreak
\vskip 4.7ex
\@afterheading}
\def\@spart#1{
\refstepcounter{part}%
\addcontentsline{toc}{part}{#1}%
{\parindent \z@ \raggedright
\interlinepenalty \@M
\normalfont
\centering\large\scshape #1\par}%
\nobreak
\vskip 4.7ex
\@afterheading}
\renewcommand*\l@part[2]{%
\ifnum \c@tocdepth >-2\relax
\addpenalty\@secpenalty
\addvspace{0.75em \@plus\p@}%
\begingroup
\parindent \z@ \rightskip \@pnumwidth
\parfillskip -\@pnumwidth
{\leavevmode
\normalsize \bfseries #1\hfil \hb@xt@\@pnumwidth{\hss #2}}\par
\nobreak
\if@compatibility
\global\@nobreaktrue
\everypar{\global\@nobreakfalse\everypar{}}%
\fi
\endgroup
\fi}

\def\l@subsection{\@tocline{2}{0pt}{2pc}{6pc}{}}
\makeatother

\begin{document}
\title{A Quantum trace map for 3-manifolds}

\author{Stavros Garoufalidis}
\address{
International Center for Mathematics, Department of Mathematics \\
Southern University of Science and Technology \\
Shenzhen, China \newline
{\tt \url{http://people.mpim-bonn.mpg.de/stavros}}}
\email{stavros@mpim-bonn.mpg.de}

\author{Tao Yu}
\address{Shenzhen International Center for Mathematics \\
Southern University of Science and Technology \\
1088 Xueyuan Avenue, Shenzhen, Guangdong, China}
\email{yut6@sustech.edu.cn}

\thanks{
  {\em Keywords and phrases}: Kauffman bracket skein module, 3-manifolds, knots,
  character varieties, quantum trace map, ideal triangulations, gluing equations,
  quantum gluing module, Frobenius-Chebyshev homomorphism. 
}

\date{19 March 2024}

\begin{abstract}
We define a quantum trace map from the skein module of a 3-manifold with torus
boundary components to a module (left and right quotient of a quantum torus)
constructed from an ideal triangulation. Our map is a 3-dimensional version of the
well-known quantum trace map on surfaces introduced by Bonahon and Wong and further
developed by L\^e. 
\end{abstract}

\maketitle

{\footnotesize
\tableofcontents
}


\section{Introduction}
\label{sec.intro}

\subsection{The quantum trace map of a surface}

The quantum trace map, introduced by Bonahon--Wong~\cite{BW:qtrace} connects
the skein algebra of a punctured surface (a quantum object), with an algebra
of $q$-commuting variables related to hyperbolic geometry. The quantum trace map
was originally introduced as a replacement of the topologically defined 
skein module (generated by framed links) and the algebro-geometric quotient of
the character variety of a surface group by a more manageable object, namely a
quantum torus, i.e., a Laurent polynomial ring of $q$-commuting variables.
The original discovery of Bonahon--Wong used miraculous cancellations which have
seen been explained by subsequent work of L\^e~\cite{Le:cancel}. 
This map has recently attracted a lot of attention from researchers in topology,
representation theory, character varieties, cluster algebras and their quantization. 

Recall that the skein module of a closed oriented surface $S(\surface)$ (with
coefficients in a ring $R$ that contains an invertible element $q$) is the
$R$-module generated by the set of isotopy classes of framed unoriented links
in $\surface \times (-1,1)$, modulo the relations~\eqref{eq-skein} and~\eqref{eq-loop}. 

\begin{align}
\begin{linkdiag}
\fill[gray!20] (-0.1,0)rectangle(1.1,1);
\begin{knot}
\strand[thick] (1,1)--(0,0);
\strand[thick] (0,1)--(1,0);
\end{knot}
\end{linkdiag}
&=q
\begin{linkdiag}
\fill[gray!20] (-0.1,0)rectangle(1.1,1);
\draw[thick] (0,0)..controls (0.5,0.5)..(0,1);
\draw[thick] (1,0)..controls (0.5,0.5)..(1,1);
\end{linkdiag}
+q^{-1}
\begin{linkdiag}
\fill[gray!20] (-0.1,0)rectangle(1.1,1);
\draw[thick] (0,0)..controls (0.5,0.5)..(1,0);
\draw[thick] (0,1)..controls (0.5,0.5)..(1,1);
\end{linkdiag},\label{eq-skein}\\
\begin{linkdiag}
\fill[gray!20] (0,0)rectangle(1,1);
\draw[thick] (0.5,0.5)circle(0.3);
\end{linkdiag}
&=(-q^2-q^{-2})
\begin{linkdiag}
\fill[gray!20] (0,0)rectangle(1,1);
\end{linkdiag}.\label{eq-loop}
\end{align}

The skein module was introduced in the early days of quantum topology by
Przytycki~\cite{Przytycki} and Turaev~\cite{Tu:conway}. 

There is a well-known connection between the skein module of a surface (or an
arbitrary 3-manifold) and the $\SL_2(\BC)$-character variety. Namely, when
$q=1$, the skein module $\calS_1(\surface)$ is a commutative algebra whose quotient
by its nil radical coincides with the coordinate ring of the $\SL_2(\BC)$-character
variety of $\surface$. The quantum trace map sends the skein module
$\calS_1(\surface)$ to a quantum torus that depends on an ideal triangulation of
$\surface$. Two key features of this map are
\begin{itemize}
\item[(a)]
  The skein module $\calS(\surface)$ of $\surface$ is an associative (and in general
  non-commutative) algebra, and so is the target quantum torus.
\item[(b)]
  The quantum trace map depends on an ideal triangulation of $\surface$, but it
  is invariant under 2--2 Pachner moves that connect any two such triangulations.
  The reason behind this is the fact that the $\SL_2(\BC)$-character variety of
  irreducible representations of a surface is irreducible and has canonical
  coordinates induced by an ideal triangulation of the surface.
\end{itemize}

Our aim is to define a 3-dimensional analogue of the quantum trace map for
a 3-manifold $M$ equipped with an ideal triangulation $\calT$. Unfortunately the
two key features above fail for 3-manifolds. Indeed, the domain
of such a map, namely the skein module $\calS(M)$, is no longer an
algebra but only a module over a universal coefficient ring (and in general only
a module over the skein algebra $\calS(\partial M)$ of the boundary of $M$).
Likewise, as in the case of a surface, an ideal triangulation $\calT$ gives coordinates
for some components of the $\SL_2(\BC)$-character variety of irreducible representations
of $M$. But now, this variety contains several components, some of which are
detected by an ideal triangulation $\calT$, but these detected components are no
longer invariant under 2--3 Pachner moves that connect every two ideal triangulations. 

With these subtleties in mind, we define a 3-dimensional
quantum trace map that allows one to do computations using the standard methods of
3-manifold triangulations developed in \texttt{SnapPy}~\cite{snappy}.

\subsection{Preliminaries}

To define the 3-dimensional quantum trace map we will fix an oriented 3-manifold $M$
and an ideal triangulation $\calT$ of it. The boundary of the manifold can be arbitrary,
however the most important case for us will be the case where the boundary is a finite
union of tori, for instance when $M$ is the complement of a knot in 3-space. There
are three ingredients that go into the definition. 

\begin{itemize}
\item
  The gluing equations variety $G_\calT$.
\item
  The coordinate ring $\BC[\frameX]$ of $\SL_2(\BC)$-character variety of $\frameX$. 
\item
  The quantum torus $\qtorus(\triang)$ and its left/right quotient $\qglue(\triang)$. 
\end{itemize}

We will begin by giving a brief description of what is needed here, and refer to
the later sections for a detailed discussion. We first recall the gluing equations
variety of an ideal triangulation introduced by Thurston~\cite{Thurston} and further
studied by Neumann--Zagier~\cite{NZ}.  
Fix a connected, oriented 3-manifold $M$ with torus boundary components and
an ideal triangulation $\calT$ of $M$ that consists of $N$ tetrahedra $T_1,\dots,T_N$.
Note that the number of edges of $\calT$ is also $N$. We can assign shape parameters
$Z_j$, $Z'_j=1/(1-Z_j)$ and $Z''_j=1-1/Z_j$ to each pair of opposite edges of the ideal
tetrahedron $T_j$, where the triple $(Z_j,Z'_j,Z''_j)$ satisfies the equations
\begin{equation}
\label{3Z}  
  Z_jZ'_jZ''_j=-1, \qquad Z_j + Z'^{-1}_j = 1, \qquad Z'_j + Z''^{-1}_j = 1,
  \qquad Z''_j + Z^{-1}_j = 1 \,.
\end{equation}
Each edge $e$ of $\calT$ gives rise to a gluing equation
\begin{equation}
\label{eT}
\prod_{T : \, e \in T} Z_T^\Box=1
\end{equation}
given as the product of the shapes of the tetrahedra
that go around the edge. 

The gluing equations variety $G^P_\calT$ is defined as the solutions in
$(\BC^\times)^{3N}$ of all shapes $Z=(Z_1,Z_1',Z_1'',\dots,Z_N,Z_N',Z_N'')$
that satisfy the Lagrangian equations~\eqref{3Z} for $j=1,\dots,N$ and all the
edge gluing equations~\eqref{eT}. A point in the gluing equations variety $G^P_\calT$
gives a $\PSL_2(\BC)$-representation of $\pi_1(M)$, well-defined up to conjugation
(thus the superscript $P$ in $G^P_\calT$). Said differently, an ideal triangulation gives
a chart for the $\PSL_2(\BC)$-character variety of the manifold.

However, we need a lift of the theory to $\SL_2(\cx)$. To achieve this, we
assign square root shape parameters $z_j,z'_j,z''_j$ to pairs of opposite edges of
each tetrahedron $T_j$. These parameters satisfy the following equations
\begin{equation}
\label{3z}
z_j z'_j z''_j=\iunit, \qquad z^2_j+(z'_j)^{-2}=1, \qquad (z'_j)^2+(z''_j)^{-2}=1,
\qquad (z''_j)^2+z_j^{-2}=1 
\end{equation}
and the edge equations which take the form
\begin{equation}
\label{et}
\prod_{T : \, e \in T} z_T^\Box= -1 \,.
\end{equation}
As before, the gluing equations variety $G_\calT$ is defined as 
the solutions in $(\BC^\times)^{3N}$ of all shapes
$z=(z_1,z_1',z_1'',\dots,z_N,z_N',z_N'')$ that satisfy the Lagrangian
equations~\eqref{3z} for $j=1,\dots,N$ and all the edge gluing equations~\eqref{et}.
With this twist, a point in the gluing equations variety $G_\calT$
gives an $\SL_2(\BC)$-representation of $\pi_1(\frameX)$ of the oriented framed bundle
$\frameX$ of $M$; see Proposition~\ref{prop-classical-rep}. At the same time, the
specialization to $q=1$ of the skein module $\skein(M)$, divided by the nilradical,
coincides with the coordinate ring of $\frameX$; see Lemma~\ref{lem-classical} below.

The last ingredient that we now discuss is the quantum torus and its two-sided
quotient, the gluing equations module. As in the skein module of $M$, we fix a
ring $R$ that contains an invertible element $q^{1/2}$ and a primitive 4-th root
of unity $\iunit$. Associated to $\calT$ is a quantum torus
\begin{equation}
\label{qtorusT}  
\qtorus(\triang) = \bigotimes_{j=1}^N \qtorus \langle \qz_j, \qz''_j \rangle, \qquad
\qtorus \langle \qz_j, \qz''_j \rangle =
R \langle \qz_j, \qz''_j \rangle/\ideal{\qz''_j \qz_j - q \qz_j \qz''_j}  
\end{equation}
where the variables $\qz_j$ and $\qz''_\ell$ for $j,\ell=1,\dots,N$ commute except in
the following instance 
$\qz_j \qz_j'' = q \qz_j'' \qz_j$. A more symmetric definition of
$\qtorus \langle \qz_j, \qz''_j \rangle$ is given by the quotient of
$R \langle \qz_j, \qz'_j, \qz''_j \rangle$ where the three variables satisfy the
$q$-commutation relations
\begin{equation}\label{qzcommute}
\qz_j \qz_j' = q \qz_j' \qz_j, \qquad
\qz_j' \qz''_j = q \qz''_j \qz_j', \qquad
\qz_j'' \qz_j = q \qz_j \qz_j'', \qquad
\qz_j\qz'_j\qz''_j=\iunit q^{3/2} \,.
\end{equation}

The quantum torus $\qtorus(\triang)$ is an associative algebra and has a left and
a right ideal generated, respectively, by the Lagrangian equations
\begin{equation}\label{3qz}
\qz_j^{-2}+(\qz''_j)^2=1
\end{equation}
for $j=1,\dots, N$ and by the edge gluing equations
\begin{equation}\label{eqt}
\prod_{T : \, e \in T} \qz_T^\Box= -q^2
\end{equation}
for all edges, where the product of these $q$-commuting variables is given by
their canonical Weyl-ordering (see Section~\ref{sub.qtorus} below). 

The \term{quantum gluing equations module} $\qglue(\triang)$ is the quotient of
$\qtorus(\triang)$ from the left by the edge-equations and from the right by the
Lagrangian equations
\begin{equation}
\label{ghatdef}
\qglue(\triang)=
\rideal{\mathrm{edge}}\backslash\qtorus(\triang)\quotbyL{\mathrm{Lagrangian}}
\,.
\end{equation}

The quantum gluing equations module is implicit in the work of Dimofte who studied
the quantization of the character variety of an ideally-triangulated
3-manifold~\cite{Dimofte:quantum}. Dimofte used the symplectic properties of the
gluing equations (coming from the symplectic properties of the Neumann--Zagier
matrices), as well as standard methods of non-commutative symplectic reduction to
arrive at a module of $q$-commuting operators. A similar module appears in
~\cite[Eqn.(10)]{AGLR}.

Our definition of $\qglue(\triang)$ comes from a presentation of the skein module
of $M$ as a quotient by a left and by a right ideal (see Proposition~\ref{prop.ST}
below), which itself comes from the fact that the 3-manifold $M$ is obtained by
a thickened surface by attaching 2-handles on either side.

\subsection{Our results}
\label{sub.results}

We now have all the ingredients to phrase our main result.
Fix an ideal triangulation $\calT$ of a 3-manifold $M$ as above.

\begin{theorem}
\label{thm.1}  
There exists a map
\begin{equation}
\label{defqtr}
\qtr : \skein(M) \longto \qglue(\triang)
\end{equation}
that fits in a commutative diagram
\begin{equation}
\label{eqn.diagram}
\begin{tikzcd}
\skein(M) \arrow[d] \arrow[r] &
\qglue(\triang)
\arrow[d] \\
\BC[\frameX] \arrow[r] & \BC[G_\calT]
\end{tikzcd} \,.
\end{equation}
\end{theorem}

The left vertical map was already discussed.
The quantum trace map~\eqref{defqtr} is given by the composition 
\begin{equation}
\label{qrt2steps}
\skein(M) \stackrel{\cong}{\longleftarrow}
\rideal{B}\backslash \skein(\surface_\triang) \quotbyL{A}
\longrightarrow \qglue(\triang).
\end{equation}
Here, $\surface_\triang$ is the boundary of a small neighborhood of the dual 1-skeleton
of $\triang$, and it is decorated with two sets of curves $\{A_f\},\{B_e\}$ that bounds
disks in $M$ but not $\surface_\triang$. The kernel of the natural map
$\skein(\surface_\triang)\to\skein(M)$ includes handle slides along these curves,
discussed in detail in Section~\ref{sec-kbsm} below. Let $\lideal{A}$
and $\rideal{B}$ denote the submodule generated by handle slides along the
corresponding types of curves.

\begin{proposition}
\label{prop.ST}
$\lideal{A}$ and $\rideal{B}$ are left and right ideals in $\skein(\surface_\triang)$
respectively, and we have an isomorphism of $R$-modules
\begin{equation}
\label{STmap}
\skein(M)\cong\rideal{B}\backslash\skein(\surface_\triang)\quotbyL{A} \,.
\end{equation}
\end{proposition}

The reason behind the isomorphism~\eqref{STmap} is topological, namely the manifold
$M$ is obtained from the thickened surface $\surface_\calT \times [-1,1]$
by attaching $A$-handles on one side and $B$-handles on the
other.

The identification~\eqref{STmap} is a convenient way to encode elements of
the 3d-skein module $\skein(M)$ in coordinates. Aside from its use in the
3d-quantum trace map, the above coordinate presentation of the skein module of a
3-manifold is useful computationally and also theoretically. Indeed, several quantum
invariants, such as the Witten--Reshetikhin--Turaev invariant and its lift to
the Habiro ring, the state integrals of Andersen--Kashaev and the 3D-index of 
Dimofte--Gaiotto--Gukov can be extended to invariants of the skein module of
$M$ and factor through the quantum trace map. We will discuss this topic in
a subsequent publication.

We next discuss the quantum trace map when $q=\zeta$ is
a root of unity. More precisely, we fix a root of unity $\zeta$ such that $\zeta^4$ is
a primitive $N$-th root of unity, and let $\varepsilon=\zeta^{N^2}$, a 4-th root of
unity. In this case, we denote the skein module and the quantum gluing equations module
by $\skein_\zeta(M)$ and $\qglue_\zeta(\triang)$, and their classical versions by
$\skein_\varepsilon(M)$ and $\qglue_\varepsilon(\triang)$ to indicate the dependence on
the chosen root of unity.

In this case, both the skein and the quantum gluing equations modules have a
Chebyshev-Frobenius homomorphism
\begin{equation}\label{2frob}
\Phi_\zeta:\skein_\varepsilon(M)\to\skein_\zeta(M), \qquad
\varphi_\zeta:\qglue_\varepsilon(\triang)\to\qglue_\zeta(\triang)
\end{equation}
where the former is defined geometrically by threading a framed link by a linear
combination of parallels given by the $N$-th Chebyshev polynomial, and the latter
is defined algebraically on a quantum torus by raising its generators to their
$N$-th powers. In both cases, the quantum binomial theorem and the vanishing of
the quantum binomial at roots of unity imply that both maps~\eqref{2frob} are
well-defined. This is discussed in detail in Section~\ref{sub.2frob} below.

The two maps~\eqref{2frob} are related by the quantum trace, which is analogous to
the surface case \cite[Theorem~5.2]{BL}.

\begin{theorem}
\label{thm.frob}
The following diagram commutes.
\begin{equation}
\label{eq-Cheb-Frob}
\begin{tikzcd}
\skein_\varepsilon(M) \arrow[r,"\qtr^\varepsilon_\triang"] \arrow[d,"\Phi_\zeta"] &
\qglue_\varepsilon(\triang) \arrow[d,"\varphi_\zeta"] \\
\skein_\zeta(M) \arrow[r,"\qtr^\zeta_\triang"] &
\qglue_\zeta(\triang)
\end{tikzcd} \,.
\end{equation} 
\end{theorem}

We finally discuss the issue of effective computation. An important feature of the
quantum trace map~\eqref{defqtr} is that it is effectively computable. In fact, it is
computable by the same methods of \texttt{SnapPy} pioneered by Thurston and
developed by Weeks (and more recently, by Culler, Dunfield and Goerner~\cite{snappy})
to study hyperbolic structures of 3-manifolds and their deformations. For a detailed
discussion, see Section~\ref{sec.compute}.

\subsection{Further directions}
\label{sub.further}

We end this section with some brief comments about further directions. As mentioned
already, ideal triangulations are related by a sequence of 2--3 Pachner moves. However,
even classically, the gluing equations variety $G_\calT$ ``sees'' some components of
representations of the framed manifold, and these components can change under 2--3
Pachner moves. Hence, the codomain $\qglue(\triang)$ of the quantum trace map can
change under 2--3 Pachner moves, and what is worse, it can become trivial when $\calT$
is a degenerate ideal triangulation (e.g., has a univalent vertex). Hence, the
map~\eqref{eqn.diagram} is \term{not} invariant under 2--3 Pachner moves. On the other
hand, the map is invariant under 3--2 Pachner moves, and under certain conditions,
also invariant under 2--3 moves. We postpone this discussion to a subsequent
publication.

Finally, extensions to the $\SL_n$-version of the quantum-trace map, building on
the results of~\cite{LeYu:SLn} are possible and will also be discussed subsequently.


\section{Triangulations}
\label{sec.triang}

\subsection{Oriented triangulations and their dual surfaces}
\label{sec-dualS}

In this section we recall oriented ideal triangulations, and define their dual
surfaces. Let $T$ be an oriented tetrahedron. A labeling of the vertices of $T$ by
$0,1,2,3$ is compatible with the orientation if vertices $1,2,3$ are counterclockwise
when viewed from vertex $0$. We represent the tetrahedron using a top view like
Figure~\ref{fig-label-tetra}.

\begin{figure}[htpb!]
\centering
\begin{tikzpicture}[baseline=(ref.base)]
\tikzmath{\s=2;}
\path (0,0) coordinate (0) node[below]{2} 
({\s/2},{\s}) coordinate (1) node[above]{0} 
(\s,0) coordinate (2) node[below]{3} 
({1.3*\s},{0.7*\s}) coordinate (3) node[right]{1}; 
\draw[dashed] (0)--(3);
\draw (1) -- (2) (2) -- (0) -- (1) -- (3) -- cycle;
\node (ref) at (1,1){\phantom{$-$}};
\end{tikzpicture}
\quad$\Rightarrow$\quad
\begin{tikzpicture}[baseline=(ref.base)]
\draw (0,0) circle(1) node[anchor=-150]{0};
\draw (0,0) -- (-150:1)node[anchor=30]{2} (0,0) -- (-30:1)
node[anchor=150]{3} (0,0) -- (90:1)node[above]{1};
\node (ref) at (0,0){\phantom{$-$}};
\end{tikzpicture}
\caption{Labeling a tetrahedron.}\label{fig-label-tetra}
\end{figure}
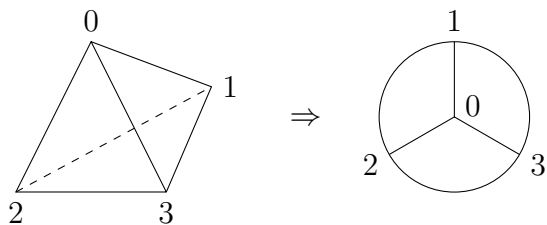

Suppose $M$ is a compact oriented 3-manifold with nonempty boundary. An oriented
triangulation $\triang$ of $M$ is a collection of tetrahedra $T_1,\dotsc,T_N$ whose faces
are paired using orientation reversing homeomorphisms such that gluing the tetrahedra
minus the vertices gives the interior of $M$.

As in the introduction, consider the dual 1-skeleton of $\triang$, which is a graph
embedded in $M$. A small neighborhood of the dual 1-skeleton is a handlebody. Then
$\surface_\triang$ is defined as the boundary of this handlebody and oriented using the
outward normal. The intersection of $\surface_\triang$ with each face $f$ of $\triang$
is a circle denoted $A_f$. There is another set of pairwise disjoint curves $B_e$ on
$\surface_\triang$, one for each edge $e$ of $\triang$, such that $B_e$ bounds a disk
dual to $e$ in the complement of the handlebody. Here, dual to $e$ means that the disk
intersects $e$ at a point, and it does not intersect other edges. We define $B_e$ more
carefully in the following.

In a tetrahedron $T$, $\surface_\triang\cap T$ is a sphere with 4 boundary components.
Borrowing from the theory of mapping class groups of surfaces, we call it the
\term{lantern}. The boundary of the
lantern consists of the $A$-curves on the faces of the tetrahedron. Each pair of boundary
curves can be connected by an arc that goes around the edges of the tetrahedron. We call
these arcs the \term{standard arcs} on the lantern, and the lantern decorated with the
standard arcs is called the \term{standard lantern}, denoted $\lantern$. This is shown
in Figure~\ref{fig-smoothL}, where the blue arcs are the standard arcs.

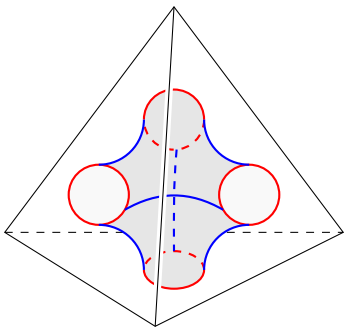
\begin{figure}[htpb!]
\centering
\begin{tikzpicture}
\path (0,2.5)coordinate(A) (-2.25,-0.5)coordinate(B) (-0.25,-1.75)
coordinate(C) (2.25,-0.5)coordinate(D);
\draw[dashed] (B) -- (D);
\fill[gray!20] (0.4,1) arc[radius=0.4,start angle=0,end angle=180]
arc[radius=0.6,start angle=0,end angle=-90]
arc[radius=0.4,start angle=90,end angle=-90]
arc[radius=0.6,start angle=90,end angle=0]
arc[x radius=0.4,y radius=0.25,start angle=-180,end angle=0]
arc[radius=0.6,start angle=180,end angle=90]
arc[radius=0.4,start angle=-90,delta angle=-180]
arc[radius=0.6,start angle=-90,end angle=-180];
\begin{scope}[red,thick]
\draw[fill=gray!5] (-1,0)circle(0.4) (1,0)circle(0.4);
\draw (0.4,1)arc[radius=0.4,start angle=0,end angle=180] (0.4,-1)
arc[x radius=0.4,y radius=0.25,start angle=0,end angle=-180];
\draw[dashed] (0.4,1)arc[radius=0.4,start angle=0,end angle=-180] (0.4,-1)
arc[x radius=0.4,y radius=0.25,start angle=0,end angle=180];
\end{scope}
\begin{scope}[blue,thick]
\draw (-1,0.4)arc[radius=0.6,start angle=-90,end angle=0]
(1,0.4)arc[radius=0.6,start angle=-90,end angle=-180]
(-1,-0.4)arc[radius=0.6,start angle=90,end angle=0]
(1,-0.4)arc[radius=0.6,start angle=90,end angle=180];
\path (-1,0)++(-30:0.4)coordinate(E) (1,0)++(-150:0.4)coordinate(F);
\draw (E) to[out=30,in=150] (F);
\draw[dashed] (0,1) ++(-85:0.4) to[out=-95,in=90] (0,-0.75);
\end{scope}
\draw[knot,knot gap=7,background color=white] (A) -- (C);
\draw (A) -- (B) -- (C) -- (D) -- cycle;
\end{tikzpicture}
\caption{Lantern surface in a tetrahedron. $A$-curves in red,
  $B$-arcs in blue.}\label{fig-smoothL}
\end{figure}

In the triangulation $\triang$, there is an embedded copy of $\lantern$ in each
tetrahedron. When a pair of faces is glued, so does the corresponding boundary components
of embedded $\lantern$. The face pairing also includes how the edges of the faces are
matched. Then we can require that the standard arcs dual to the matched edges connect to
each other. After all faces are glued, the standard arcs form the $B$-circles.

The data $\heeg_\triang:=(\surface_\triang,\{A_f\},\{B_e\})$ is called the \term{dual
surface} to the triangulation $\triang$. We just argued that $\heeg_\triang$ can be
constructed from the face pairings of $\triang$. Conversely, The intersection pattern of
the curves $\{A_f\}$ and $\{B_e\}$ determines the face pairings of the triangulation. In
this sense, $\heeg_\triang$ is equivalent to the triangulation $\triang$.
 
\subsection{Diagrams for the lantern}

The orientation of $\surface_\triang$ is important, for example in the definition of the
skein algebra. Thus, we want to describe a triangulation $\triang$ using combinatorial
data that respects orientations.

Now consider the standard lantern inside the tetrahedron. Notice that the lantern is a
smooth version of the truncated tetrahedron, which is embedded in the tetrahedra of
$\triang$ in a dual position. See Figure~\ref{fig-dual-tetra}, where the top view is also
included. Note the dual vertex 3 is in the back, which affects the orientation of the
figure. We either stretch the back vertex to infinity and truncate, or rotate the dual
tetrahedron to make some vertex the top one. See Figure~\ref{fig-flatL}. Note the
labeling is opposite of the orientation of the dual tetrahedron.

\begin{figure}[htpb!]
\centering
\begin{tikzpicture}[line join=bevel,baseline=(ref.base)]
\tikzmath{\s=3;}
\path (0,0) coordinate (0) node[below]{2} 
({\s/2},{\s}) coordinate (1) node[above]{0} 
(\s,0) coordinate (2) node[below]{3} 
({1.3*\s},{0.7*\s}) coordinate (3) node[right]{1}; 
\draw[dashed] (0) -- (3);
\foreach \a/\b/\c in {0/1/2,1/3/2,0/2/3,0/3/1}
\path (barycentric cs:\a=0.2,\b=0.4,\c=0.4)coordinate(v\a\b\c)
(barycentric cs:\b=0.2,\c=0.4,\a=0.4)coordinate(v\b\c\a)
(barycentric cs:\c=0.2,\a=0.4,\b=0.4)coordinate(v\c\a\b);
\fill[gray!20] (v310) -- (v031) -- (v213) -- (v321) -- (v132) -- (v023) -- (v302)
-- (v120) -- (v012) -- (v201) -- cycle; 
\begin{scope}[red,thick]
\draw[fill=gray!5] (v012) -- (v120) -- (v201) -- cycle
(v132) -- (v321) -- (v213) -- cycle;
\draw (v023) -- (v302) (v031) -- (v310);
\draw[densely dashed] (v023) -- (v230) -- (v302) (v310) -- (v103) -- (v031);
\end{scope}
\begin{scope}[blue,thick]
\draw (v012) -- (v321) (v120) -- (v302) (v201) -- (v310) (v132) -- (v023) (v031)
-- (v213);
\draw[densely dashed] (v230) -- (v103);
\end{scope}
\draw[knot,knot gap=7,background color=white] (1) -- (2);
\draw (2) -- (0) -- (1) -- (3) -- cycle;
\node (ref) at (1,1.5){\phantom{$-$}};
\end{tikzpicture}
\quad$\Rightarrow$\quad
\begin{tikzpicture}[baseline=(ref.base)]
\draw[thick,blue,fill=gray!20] (0,0)circle(1);
\fill[white] (0,0)circle(0.12) foreach \t in {-90,30,150} {(\t:1)circle(0.12)};
\draw[thick,blue,dashed] foreach \t in {30,150,-90} {(\t:0.12) -- (\t:1.12)};
\draw[thick,red,radius=0.12] (30:1)circle node[anchor=-150]{\small2} 
(150:1)circle node[anchor=-30]{\small3} 
(-90:1)circle node[below,outer sep=2pt]{\small1} 
(0,0)circle node[anchor=-60,outer sep=1pt]{\small0};
\draw[knot,knot gap=7,background color=white]
(0,0) -- (-150:1.5)node[anchor=30]{3} 
(0,0) -- (-30:1.5)node[anchor=150]{2} 
(0,0) -- (90:1.5)node[above]{1}; 
\draw (0,0) circle(1.5) node[anchor=120,outer sep=1pt]{0};
\node (ref) at (0,0){\phantom{$-$}};
\end{tikzpicture}
\caption{Truncated dual tetrahedron.}\label{fig-dual-tetra}
\end{figure}
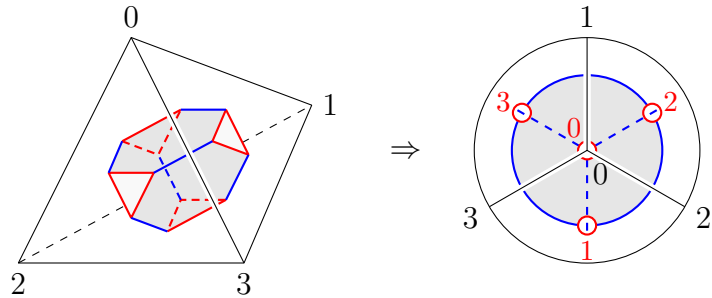

\begin{figure}[htpb!]
\centering
\begin{tikzpicture}[baseline=(ref.base)]
\draw[thick,blue,fill=gray!20] (0,0)circle(1);
\fill[white] (0,0)circle(0.12) foreach \t in {-90,30,150} {(\t:1)circle(0.12)};
\draw[thick,blue,dashed] foreach \t in {30,150,-90} {(\t:0.12) -- (\t:1.12)};
\draw[thick,red,radius=0.12] (30:1)circle node[anchor=-150]{\small2} 
(150:1)circle node[anchor=-30]{\small3} 
(-90:1)circle node[below,outer sep=2pt]{\small1} 
(0,0)circle node[above,outer sep=2pt]{\small0};
\node (ref) at (0,0){\phantom{$-$}};
\end{tikzpicture}
\quad$\Rightarrow$\quad
\begin{tikzpicture}[thick,baseline=(ref.base)]
\tikzmath{\r1=0.4;\r2=0.2;}
\draw[red,fill=gray!20] (0,0)circle(1) (45:1)node[above right,inner sep=2pt]{0};
\draw[blue] (-90:\r1) -- (30:\r1) -- (150:\r1) -- cycle;
\draw[red,fill=white] (-90:\r1)circle(\r2)node{1} (30:\r1)circle(\r2)
node{2} (150:\r1)circle(\r2)node{3};
\draw[blue] foreach \t in {-90,30,150} {(\t:{\r1+\r2}) -- (\t:1)};
\node (ref) at (0,0){\phantom{$-$}};
\end{tikzpicture}
\quad or \quad
\begin{tikzpicture}[baseline=(ref.base)]
\draw[thick,blue,fill=gray!20] (0,0)circle(1);
\draw[thick,blue] foreach \t in {-30,-150,90} {(0,0) -- (\t:1)};
\fill[white] (0,0)circle(0.12) foreach \t in {-30,-150,90} {(\t:1)circle(0.12)};
\draw[thick,red,radius=0.12] (90:1)
circle node[anchor=-90,outer sep=2pt]{1} (-30:1)
circle node[anchor=150]{2} (-150:1)circle node[anchor=30]{3} (0,0)
circle node[anchor=90,outer sep=2pt]{0};
\node (ref) at (0,0){\phantom{$-$}};
\end{tikzpicture}
\caption{Different diagrams of the lantern.}\label{fig-flatL}
\end{figure}
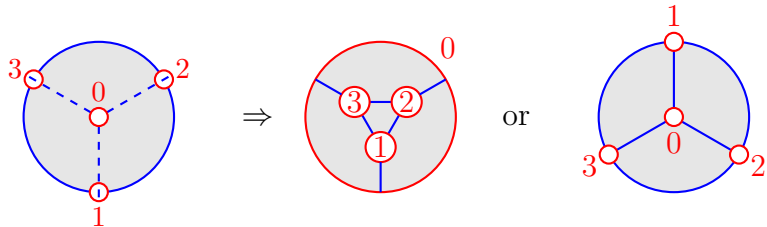


\section{Classical case}
\label{sec.classical}

By ``classical'' one usually refers to the specialization $q=1$ of the skein
module, which is then related to the $\PSL_2(\BC)$-character variety of the
ambient manifold. 
In this section we recall several well-known facts about triangulations and
$\SL_2(\BC)$-character varieties of 3-manifolds. Our constructions will involve
$\SL_2(\BC)$-representations rather than $\PSL_2(\BC)$ ones that come naturally
from developing maps in hyperbolic geometry (due to the fact that the orientation
preserving isometries of 3-dimensional hyperbolic space is $\PSL_2(\BC)$). 

This subtle distinction between $\PSL_2$ versus $\SL_2$ will require minor twists
and modifications of well-known results.

\subsection{Twisted character variety}
\label{sub.twisted}

The first twist in this story is that we need to consider the frame bundle
$p:\frameM\to M$ of oriented orthonormal frames in the tangent bundle of $M$ (defined
with respect to some fixed Riemannian metric on $M$) as opposed to $M$ itself.

This is a principal $\SO_3$-bundle, and in dimension 3, it is well-known that it
is trivial. This implies that
\begin{equation}
\label{eq-frame-iso}
\pi_1(\frameM)\cong\pi_1(M)\times\ints/2.
\end{equation}
Here $\ints/2\cong\pi_1(\SO_3)$ is canonically included in the center of $\pi_1(\frameM)$
using the inclusion of the fiber $i:\SO_3\embed\frameM$, whereas the inclusion of
$\pi_1(M)$ depends on a spin structure of $M$.

Let $\frameX$ be the subset of the $\SL_2(\cx)$-character variety of $\frameM$ represented
by homomorphisms $\pi_1(\frameM)\to\SL_2(\cx)$ sending a nontrivial loop in the fiber to
$-I$. It is non-canonically isomorphic to the usual $\SL_2(\BC)$-character variety,
denoted $X_M$, using \eqref{eq-frame-iso}. We call $\frameX$ the twisted character
variety of $M$. It is more natural than $X_M$ in our setup. See Section~\ref{sec-kbsm}.

The skein module of $M$ is generated by framed links in $M$. Conveniently, framed
curves in $M$ represent elements of $\pi_1(\frameM)$. Indeed, let $\alpha$ be
a smooth path with nonvanishing tangent. A normal framing of $\alpha$ is a unit vector
field along $\alpha$ which is everywhere orthogonal to the tangent vector of $\alpha$.
The curve $\alpha$ with a normal framing determines a section of $\frameM$ along
$\alpha$ using the tangent vector, the normal vector, and their cross product.

Every element of $\pi_1(\frameM)$ can be represented by such a path in $M$ with a normal
framing. To do so, first ignore the framing and find a smooth path with the specified
tangent vectors at the endpoints. If we assign a random normal framing, it is either
homotopic to the given element in $\pi_1(\frameM)$ or differ by a full rotation in the
fiber. The latter can be inserted in the normal framing.

\subsection{Cell decomposition}
\label{sec-cell}

Given a triangulation $\triang$, we define a cell decomposition (a CW complex) on $M$
that is convenient to use in the frame bundle $\frameM$.

Take the dual surface $\heeg_\triang:=(\surface_\triang,\{A_f\},\{B_e\})$. For technical
reasons involving orientations and smoothness, we replace each $A_f$ and $B_e$ with two
parallel copies. If we cut along all parallels of $A_f$, we get annuli between the
parallels and lanterns in each tetrahedron, and each lantern is decorated with two
parallels for every standard arc. These parallels define a cell structure on
$\surface_\triang$, which is extended to a cell structure on $M$ by attaching 2-cells
along all parallels of $A$-curves and $B$-curves, and then filling in with 3-cells
inside the cylinders bounded by parallel curves and inside the lanterns. We can
simplify this cell structure by contracting the transverse edges between parallel
$A$-curves. Let $\dtrunc$ denote the cell structure of $M$ after the contraction.

If we also collapse the cylinders bounded by parallel $A$-curves in the transverse
direction, or equivalently, if we do not double the $A$-curves, we obtain the cell
structure in \cite{GGZ:glue} using doubly truncated tetrahedra. Here, a doubly
truncated tetrahedron is obtained from an ideal tetrahedron by truncating the vertices
and then the edges. See Figure~\ref{fig-doubly}. Then $M$ decomposes as the union of
the double truncation of all tetrahedra and the prism neighborhoods of the edges. See
the same figure for an example of an edge with valence 4.

\begin{figure}[htpb!]
\centering
\begin{tikzpicture}[baseline=0cm,line join=bevel]
\tikzmath{\s=3;}
\path (0,0) coordinate (0) (60:\s) coordinate (1) (\s,0) coordinate (2)
({1.3*\s},{1.05*\s}) coordinate (3);
\draw[dashed] (0) -- (3);
\foreach \a/\b/\c/\d in {0/1/2/3,1/3/2/0,2/0/1/3,3/2/1/0}
\draw[thick,red] (barycentric cs:\a=0,\b=7,\c=3,\d=1)
-- (barycentric cs:\a=0,\b=3,\c=7,\d=1)
-- (barycentric cs:\a=0,\b=1,\c=7,\d=3)
-- (barycentric cs:\a=0,\b=1,\c=3,\d=7)
-- (barycentric cs:\a=0,\b=3,\c=1,\d=7)
-- (barycentric cs:\a=0,\b=7,\c=1,\d=3) -- cycle;
\draw[thick,red,knot,knot gap=3,background color=white]
(barycentric cs:0=0,1=6.8,2=3.2,3=1) -- (barycentric cs:0=0,1=3.5,2=6.5,3=1)
(barycentric cs:0=1,1=6.8,2=3.2,3=0) -- (barycentric cs:0=1,1=3.5,2=6.5,3=0)
(barycentric cs:0=0,1=2.5,2=7,3=1.5) -- (barycentric cs:0=0,1=1.1,2=7,3=2.9)
(barycentric cs:0=1.5,1=2.5,2=7,3=0) -- (barycentric cs:0=2.95,1=1.05,2=7,3=0)
(barycentric cs:0=6.8,1=3.2,2=1,3=0) -- (barycentric cs:0=3.5,1=6.5,2=1,3=0)
(barycentric cs:0=0,1=6.5,2=1,3=3.5) -- (barycentric cs:0=0,1=3.2,2=1,3=6.8);
\foreach \a/\b/\c/\d in {0/1/2/3,1/3/2/0,2/0/1/3,3/2/1/0}
\draw[thick,blue,line cap=round] (barycentric cs:\a=7,\b=3,\c=0,\d=1)
-- (barycentric cs:\a=7,\b=3,\c=1,\d=0)
(barycentric cs:\a=7,\b=1,\c=3,\d=0)
-- (barycentric cs:\a=7,\b=0,\c=3,\d=1)
(barycentric cs:\a=7,\b=0,\c=1,\d=3)
-- (barycentric cs:\a=7,\b=1,\c=0,\d=3);
\draw[knot,knot gap=7,background color=white] (1) -- (2);
\draw (0) -- (1) -- (3) -- (2) -- cycle;
\end{tikzpicture}
\begin{tikzpicture}[baseline=0cm,line join=bevel]
\draw[blue,thick] (0,0)coordinate(A) -- (-30:0.5)coordinate(B) 
-- ++(30:0.5)coordinate(C) -- ++(90:0.5)coordinate(D);
\draw[blue,thick,dashed] (A) -- (D);
\draw[dashed] (0.5,0) -- ++(-1,1.5);
\begin{scope}[xshift=-1cm,yshift=1.5cm]
\draw[blue,thick] (0,0) -- (-30:0.5) -- ++(30:0.5) -- ++(90:0.5) -- cycle;
\end{scope}
\draw[red,thick,line cap=round] foreach \p in {A,B,C,D} {(\p) -- ++(-1,1.5)};
\fill (0.5,0)circle(0.05) ++(-1,1.5)circle(0.05);
\end{tikzpicture}
\caption{Doubly truncated tetrahedron and prism neighborhood of an edge.}
\label{fig-doubly}
\end{figure}
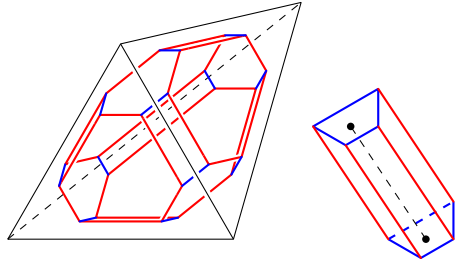

Following \cite{GGZ:glue}, the 1-skeleton of the doubly truncated cell structure
consists of three types of edges, namely short, medium and long, denoted by
$\gamma$, $\beta$ and $\alpha$, respectively.
The short edges are arcs on the parallels of the $B$-curves. The edges on the
$A$-curves are divided into two types. The long ones are parallel to the edges of
the tetrahedra and the rest are the medium ones. In \cite{GGZ:glue} there is a
systematic way to label different edges in each tetrahedron, but we will not need it
here.

To connect with the frame bundle $\frameM$, we need a smooth version of the 1-skeleton.
By construction, the 1-skeleton can be drawn on $\surface_\triang$. The $\gamma$ edges
are already joined smoothly since they are segments of $B$-curves. For the $\alpha$ and
$\beta$ edges, we homotope them according to Figure~\ref{fig-smooth-skel}, where
orientations are also chosen for the edges. The key here is that at each point of the
0-skeleton, the tangent vectors of the edges all agree.

\begin{figure}[htpb!]
\centering
\begin{tikzpicture}[scale=1.25,every node/.style={inner sep=1pt}]
  \draw[blue,dashed] (0.8,-1) -- +(0,2) node[above]{$B_e$} (3,0.4)
  arc[radius=0.4,start angle=90,delta angle=180];
\begin{scope}[blue,thick]
\draw (0.4,-1) -- +(0,2) (1.2,-1) -- +(0,2);
\path[->,tips] (0.4,1) -- +(0,-0.4);
\path[->,tips] (0.4,0) -- +(0,-0.7);
\path[->,tips] (1.2,-1) -- +(0,0.4);
\path[->,tips] (1.2,0) -- +(0,0.7);
\draw (3,0.6) arc[radius=0.6,start angle=90,delta angle=180];
\path[->,tips] (3,0.6) arc[radius=0.6,start angle=90,delta angle=150];
\path[->,tips] (3,0.6) arc[radius=0.6,start angle=90,delta angle=40];
\path (0.4,0.8)node[left]{$\gamma$} (0.4,-0.8)node[left]{$\gamma$} (1.2,0.8)
node[right]{$\gamma$} (1.2,-0.8)node[right]{$\gamma$};
\end{scope}
\draw[red,dashed] (0,0)node[left]{$A_f$} -- (3,0);
\begin{scope}[red,thick]
\draw (0.8,0)circle(0.4) (1.8,0)circle(0.6) (-0.2,0.6)
arc[radius=0.6,start angle=90,delta angle=-180] (2.7,0)++(30:0.3)
arc[radius=0.3,start angle=30,delta angle=300];
\path[->,tips] (0.8,0.4)arc[radius=0.4,start angle=90,delta angle=45];
\path[->,tips] (0.8,-0.4)arc[radius=0.4,start angle=-90,delta angle=45];
\path[->,tips] (1.8,0.6)arc[radius=0.6,start angle=90,delta angle=-45];
\path[->,tips] (1.8,-0.6)arc[radius=0.6,start angle=-90,delta angle=-45];
\path[->,tips] (-0.2,0.6)arc[radius=0.6,start angle=90,delta angle=-30];
\path[->,tips] (-0.2,0.6)arc[radius=0.6,start angle=90,delta angle=-160];
\path[->,tips] (2.7,0)++(30:0.3)arc[radius=0.3,start angle=30,delta angle=30];
\path[->,tips] (2.7,0)++(30:0.3)arc[radius=0.3,start angle=30,delta angle=280];
\path (0.8,0.4)node[above right]{$\alpha$} (0.8,-0.4)node[below left]{$\alpha$}
(1.8,0.6)node[above]{$\beta$} (1.8,-0.6)node[below]{$\beta$};
\end{scope}
\fill foreach \x in {0.4,1.2,2.4} {(\x,0)circle(0.04)};
\end{tikzpicture}
\caption{A smooth 1-skeleton of $\dtrunc$ on $\surface_\triang$.}\label{fig-smooth-skel}
\end{figure}

Now we also describe the 2-cells of $\dtrunc$. From Figure~\ref{fig-smooth-skel}, we can
see the disks bounded by $\alpha^2$ and $\beta^2$, which come from the doubling of
$A$-curves. The rest of the 2-cells can be seen in Figure~\ref{fig-doubly}. Two types are
on the surface $\surface_\triang$ and partially visible in Figure~\ref{fig-smooth-skel}:
the hexagons $(\gamma^{-1}\beta)^3$ near the vertices of the tetrahedra and the rectangles
$(\alpha^{-1}\gamma)^2$ near the edges of the tetrahedra. Finally, two more types are not
on $\surface_\triang$: the hexagons $(\alpha^{-1}\beta)^3$ near the faces and the polygons
$\gamma^n$ at the bases of the edge prisms. Here, the boundaries of the 2-cells are
expressed as a concatenation of edges. We did not distinguish between different edges of
the same type, but the notation makes sense since at each point of the 0-skeleton, the
labels $\alpha,\beta,\gamma$ and their inverses uniquely determine the next edge.

By the construction of the last section, each edge of the 1-skeleton $\dtrunc^{(1)}$
determines a path in $\frameM$. The condition on the tangent vectors at the 0-skeleton
implies that these paths glue together to define a section $s:\dtrunc^{(1)}\to\frameM$.

The section $s$ cannot be extended to the entire 2-skeleton by the following lemma, but
it is a natural construction, and it allows us to define representations in
Proposition~\ref{prop-classical-rep} without making additional choices.

\begin{lemma}
The section $s$ extends over the 2-cells bounded by the hexagons $(\alpha^{-1}\beta)^3$
and $(\gamma^{-1}\beta)^3$. $s$ does not extend over 2-cells bounded by the circles
$\alpha^2,\beta^2$, the rectangles $(\alpha^{-1}\gamma)^2$, or the polygons around the
edges $\gamma^n$.
\end{lemma}

\begin{proof}
First we consider the 2-cells on the surface $\surface_\triang$, which are
$\alpha^2,\beta^2$,
$(\alpha^{-1}\gamma)^2$, and $(\gamma^{-1}\beta)^3$. To determine if the boundary of such
a 2-cell is trivial in $\pi_1(\frameM)$, we count how many times the tangent vector turns
on the surface. $\alpha^2$ and $\beta^2$ are the easiest: they do a full turn. The others
are slightly tricky since some edges are traversed in reverse direction. Note the
orientations of the edges are used to define the section $s$, so even if we go backwards,
the tangent vector should be the same. With this in mind, it is easy to see that
$(\alpha^{-1}\gamma)^2$ does a full turn, and $(\gamma^{-1}\beta)^3$ does two full turns,
so the former is nontrivial, while the latter is trivial in $\pi_1(\frameM)$.

Next, we look at the polygons around the edges $\gamma^n$. The normal framing can be
taken to lie on the 2-cell bounded by $\gamma^n$ pointing inward. Thus, the framing
does a full turn, which is nontrivial.

Finally, we look at $(\alpha^{-1}\beta)^3$. It is easier to consider $(\alpha\beta)^3$,
which is a smooth curve that also bounds a disk in $M$, and the normal framing of
$(\alpha\beta)^3$ is the outward normal of the boundary of the disk. This shows that
$(\alpha\beta)^3$ is homotopic to a full turn. Since $\alpha^2$ is a full turn and
$(\alpha^{-1}\beta)^3$ differs from $(\alpha\beta)^3$ by 3 insertions of $\alpha^2$,
$(\alpha^{-1}\beta)^3$ is trivial in $\pi_1(\frameM)$.
\end{proof}

\begin{corollary}
The section $s$ induces a surjective homomorphism of the fundamental groupoids
$s_\ast:\pi_1(\dtrunc^{(1)},\dtrunc^{(0)})
\to\pi_1(\frameM,s(\dtrunc^{(0)}))$. 
\end{corollary}

\begin{proof}
Given a path $a$ in $\frameM$ with endpoints on $s(\dtrunc^{(0)})$, we can
require the projection $p(a)$ on $M$ is on the 1-skeleton $\dtrunc^{(1)}$ after 
a homotopy. Then in $\pi_1(\frameM,s(\dtrunc^{(0)}))$, $a$ and $s_\ast(p(a))$
differ by some rotation in the fiber, which is in the image of $s_\ast$. Thus, every
element $a$ is in the image of $s_\ast$.
\end{proof}

\subsection{Representations}

In this section we discuss how to assign $\SL_2(\BC)$-representations to points of
the gluing equations variety $G_\calT$. Since the $\gamma$ edges goes around the
edges of the tetrahedra, each $\gamma$ edge is assigned a shape parameter $z_\gamma$.

\begin{proposition}
\label{prop-classical-rep}
Given a point $z \in G_\triang$, the assignment of matrices
\begin{equation}
\label{eq-abc-mat}
\alpha\mapsto\begin{pmatrix}0&1\\-1&0\end{pmatrix},\qquad
\beta\mapsto\begin{pmatrix}\iunit&0\\1&-\iunit\end{pmatrix},\qquad
\gamma\mapsto\begin{pmatrix}z_\gamma^{-1}&0\\0&z_\gamma\end{pmatrix}
\end{equation}
to the 1-skeleton $\dtrunc^{(1)}$ defines a homomorphism
\begin{equation}
\pi_1(\frameM,s(\dtrunc^{(0)}))\to\SL_2(\cx).
\end{equation}
This representation is a lift of the $\PGL_2(\cx)$-representation given
in~\cite{GGZ:glue}.
\end{proposition}

\begin{proof}
By comparing \eqref{eq-abc-mat} with Example~10.15 of \cite{GGZ:glue}, we see that they
are conjugate up to scalars. Note our $\gamma$ has the opposite orientation as their
$\gamma^{012}$. This implies the lifting property assuming the $\SL_2(\cx)$-representation
here is well-defined.

Let $\rho_s:\pi_1(\dtrunc^{(1)},\dtrunc^{(0)})\to\SL_2(\cx)$ denote the
assignment of the matrices. To show that the representation is well-defined, we first
look at the images of boundaries of the 2-cells under $\rho_s$. By a direct calculation,
we see that the boundary of a 2-cell maps to $I$ if $s$ extends over the 2-cell or to
$-I$ if $s$ does not extend.

Choose a spin structure of $M$, represented by a section
$t:\dtrunc^{(1)}\to\frameM$
that extends to a section $\bar{t}:M\to\frameM$. We can require that $t$ agrees with $s$
on $\dtrunc^{(0)}$. Then $t$ determines an embedding
$\pi_1(M,\dtrunc^{(0)})\to\pi_1(\frameM,s(\dtrunc^{(0)}))$, which is
part of an isomorphism
\begin{equation}
\pi_1(M,\dtrunc^{(0)})\times\ints/2
\xrightarrow[\cong]{\bar{t}_\ast\cdot i_\ast}
\pi_1(\frameM,s(\dtrunc^{(0)})).
\end{equation}

The difference between $s_\ast$ and $t_\ast$ is given by a map
$e:\pi_1(\dtrunc^{(1)},\dtrunc^{(0)})\to\ints/2$. Let
$\rho_t:\pi_1(\dtrunc^{(1)},\dtrunc^{(0)})\to\SL_2(\cx)$ be the
assignment of matrices modified from $\rho_s$ using the signs determined by $e$. By
construction, $\rho_t$ maps the boundaries of all 2-cells to $I$. Thus, $\rho_t$
determines a representation $\bar{\rho}_t:\pi_1(M,\dtrunc^{(0)})\to\SL_2(\cx)$,
which can be extended to
\begin{equation}
\rho:\pi_1(\frameM,s(\dtrunc^{(0)}))
\cong\pi_1(M,\dtrunc^{(0)})\times\ints/2
\xrightarrow{\bar{\rho}_t\cdot i_\ast}
\SL_2(\cx).
\end{equation}
Then $\rho$ is also the representation induced by $\rho_s$. To see this, note that
$\rho\circ t_\ast=\rho_t$ by construction. Since $t_\ast$ and $\rho_t$ are both related
to their $s$-counterparts by $e$, we get $\rho\circ s_\ast=\rho_s$.
\end{proof}

\begin{remark}
Recall $Z^\square_j=(z^\square_j)^2$ is the usual shape parameter. Although opposite edges
of a tetrahedron have the same shape parameters, they do not need to have the same square
roots. For the proof to work, $z_jz'_jz''_j=\iunit$ in \eqref{3z} can be
replaced by 4 equations of the same form for triples of edges sharing a vertex in each
tetrahedron. These equations imply that $(z^\square_j)^2$ is the same for opposite edges,
but in each tetrahedron, we can choose all 3 pairs of opposite edges to take the same or
the opposite square roots.

This is also reflected in the quantization. See Remark~\ref{rem-qglue-ext}.
\end{remark}


\section{Skein modules}
\label{sec.skein}

\subsection{Kauffman bracket skein modules}
\label{sec-kbsm}

In this section we recall the basics of the skein module of a 3-manifold in several
flavors. The \term{skein module} $\skein_q(M;R)$ of an oriented 3-manifold
$M$ is the $R$-module generated by the set of isotopy classes of framed unoriented
links in $M$, modulo the relations~\eqref{eq-skein} and~\eqref{eq-loop}. 

When $M=\surface\times(-1,1)$ is the thickening of an oriented surface $\surface$,
the skein module $\skein(M)$ gains an algebra structure by stacking. This means given
two links $\alpha,\beta$, isotoped such that $\alpha\subset\surface\times(0,1)$ and
$\beta\subset\surface\times(-1,0)$, we define $\alpha\beta=\alpha\cup\beta$. Then the
skein algebra $\skein(\surface)$ is the module $\skein(M)$ with this product structure.

In this definition, we fix a ring $R$ that contains an invertible
element $q^{1/2}$ and a primitive 4-th root of unity $\iunit$. 
The two main examples that we are interested in are the universal case
$R=\Runiv:=\ints[\iunit][q^{\pm1/2}]$ and the classical limit $R=\cx$ with
$q^{1/2}=1$.

Some statements about the skein module are independent of the choice of $R$ and $q$. We
will omit them if the choice is unimportant. When both are omitted, it is usually
sufficient to consider the universal case $R=\Runiv$. This follows from the universal
coefficient property \cite[Proposition~2.2(4)]{Przytycki:fund}
\begin{equation}
\label{eq-skein-ucoef}
\skein_q(M;R)\otimes_R R'\cong \skein_{q'}(M;R')
\end{equation}
induced by a ring homomorphism $r:R\to R'$ with $q'=r(q)$.

The skein module is a quantization of the character variety in the following sense. When
$q=\pm1$, \eqref{eq-skein} shows that crossing changes do not affect the corresponding
elements in $\skein_{\pm1}(M;\cx)$. Thus, only the (framed) homotopy classes of the links
matter. This also means that disjoint union defines a commutative algebra structure on
$\skein_{\pm1}(M;\cx)$.

In both cases, $\skein_{\pm1}(M;\cx)$ is related to the coordinate ring of the character
variety. Let $c$ be a closed curve in $M$. Define the \term{trace function}
$t_c:X_M\to\cx$ by $t_c(\rho)=\tr(\rho(c))$ for every $\rho:\pi_1(M)\to\SL_2(\cx)$. It
is easy to see that $t_c$ only depends on the conjugacy class of $\rho$ and homotopy
class of $c$, and it is independent of the orientation of $c$.

\begin{theorem}[\cite{Bul}]\label{thm-classical-m1}
The algebra homomorphism $\skein_{-1}(M;\cx)\to\cx[X_M]$ sending each knot $K$ to $-t_K$
is surjective, and the kernel is the nilradical of $\skein_{-1}(M;\cx)$.
\end{theorem}

This result is not sufficient for our purposes, since the specialization to
$q=-1$ is generally not a commutative limit for a quantum torus. It turns out that
the specialization to $q=1$ is related to the twisted character variety $\frameX$.
Trace functions can be defined on $\frameX$ similarly for smooth curves in $M$ with
normal framing. 

\begin{lemma}
\label{lem-classical}
The algebra homomorphism $\skein_1(M;\cx)\to\cx[\frameX]$ sending each framed knot $K$ to
$t_K$ is surjective, and the kernel is the nilradical of $\skein_1(M;\cx)$.
\end{lemma}

\begin{proof}
Fix a spin structure of $M$. Given a framed knot $K$, define $\sigma(K)=1$ if the
trivialization of $\frameM|_K$ determined by the normal framing agrees with the spin
structure up to homotopy, and define $\sigma(K)=-1$ otherwise. \cite{Bar} shows that
$K\mapsto -\sigma(K)K$ defines an algebra isomorphism between $\skein_{\pm1}(M;\cx)$ (in
either direction). Similarly, by unpacking the isomorphism $X_M\cong\frameX$ induced by
\eqref{eq-frame-iso}, the isomorphism $\cx[\frameX]\to\cx[X_M]$ of coordinate rings is
given by $t_K\mapsto\sigma(K)t_K$ for any framed knot $K$. Then the corollary follows
from Theorem~\ref{thm-classical-m1} combined with these isomorphisms.
\end{proof}

The skein module of a 3-manifold does not change when 3-handles are attached to the
manifold, and changes in a predictable way when 2-handles are attached. Indeed, 
consider the 3-manifold $N$ obtained from $M$ with a 2-handle attached along a smooth
curve $c$ on $\partial M$. $c$
comes with two isotopic normal framings given by the outward and inward normal vectors of
$\partial M$. Let $L\subset M$ be a framed link with a segment close to $c$ where the
framing of $L$ and $c$ are opposite. Then there is a well-defined connected sum $L\# c$
which is isotopic to $L$ in $N$. The operation $L\to L\# c$ in $M$ is called a
\term{handle slide}.

\begin{proposition}[{\cite[Proposition~2.2]{Przytycki:fund}}]
\label{prop-handle}
The inclusion $i:M\embed N$ induces a surjective map $i_\ast:\skein(M)\onto\skein(N)$.
The kernel is the submodule generated by handle slides.
\end{proposition}

\begin{remark}
The proof of the proposition does not depend on the details of the defining relations.
It only requires the fact that the defining relations of $\skein(N)$ can always be
performed in $M$. For the variations of skein modules defined below, we also have
handle slides for attaching 2-handles. The proof will be omitted.
\end{remark}

\begin{remark}
Although we will not use it in our paper, we remark that the assignement of every
framed link in $M$ to its homology class with coefficients in $\BZ/2\BZ$ induces
a $H_1(M,\BZ/2\BZ)$-grading on the skein module $\skein(M)$. The quantum trace map
preserves that grading. 
\end{remark}

\subsection{Stated skein modules}

Motivated by the ideas of TQFT, we want a local version of the skein module. This
means cutting a 3-manifold along an embedded surface, and doing so, a framed link
becomes a framed tangle in the cut manifold. In addition, we will use a choice of
$\pm$ at each boundary point of the framed tangle, thus arriving at the stated
framed tangles that we now discuss in detail. 

We recall the stated skein modules defined by L\^e \cite{Le:decomp}. A
\term{punctured bordered surface} is a surface of the form
$\surface=\surclose\setminus\marked$, where $\surclose$ is a compact oriented surface,
possibly with boundary
and $\marked$ is a finite set that intersects every component of $\partial\surclose$.
A component of $\partial\surface$ is called a \term{boundary edge}, which is
homeomorphic to an open interval. Although the orientation of the surface induces
one on its boundary, we will not use the orientation of the boundary at all. 

A tangle over $\surface$ is an unoriented embedded 1-dimensional submanifold
$\alpha\subset\surface\times(-1,1)$ such that
\begin{enumerate}
\item
  the projection of $\partial\alpha$ is in $\partial\surface$, and that
\item
  the heights of $\partial\alpha$ (i.e., the image of $\partial\alpha$ under the
  projection map $\surface \times (-1,1) \to (-1,1)$) are distinct over each boundary
  edge.
\end{enumerate}

A vector $\surface\times(-1,1)$ is \term{vertical} if it is tangent to
$\{p\}\times(-1,1)$ in the positive direction. A \term{framing} of a tangle $\alpha$
over $\surface$ is a transverse vector field along $\alpha$ which is vertical at
$\partial\alpha$. A \term{state} of a tangle $\alpha$ is a map
$\partial\alpha\to\{+,-\}$. By convention, $\pm$ are identified with $\pm1$ when
necessary.

Throughout the paper, all tangles will be framed and stated. 
An \term{isotopy} of tangles over $\surface$ is homotopy
within the class of tangles over $\surface$. In particular, the height order of
$\partial\alpha$ over each boundary edge of $\surface$ is preserved by isotopy.

As usual, tangles are represented by their diagrams. The framing of a diagram is
always vertical. At each endpoint, the state is labeled by a sign $\pm$. It is also
important to indicate height orders in the diagrams. In the figures drawn below, 
we use an arrow on the boundary of the surface to indicate that as one follows the
direction of the arrow, the heights of the endpoints are consecutive and increasing.
This arrow should not be confused with the orientation of the boundary inherited from
the orientation of the surface. 

Fix a ring $R$ and $q\in R$ as in the introduction. The \term{stated skein module}
$\skein_q(\surface;R)$ of $\surface$ (abbreviated by $\skein(\surface)$ when $R$ and $q$
are clear as before) is the $R$-module spanned by isotopy classes of (framed, stated)
tangles over $\surface$ modulo the skein relation \eqref{eq-skein}, the trivial loop
relation \eqref{eq-loop}, the trivial arc relations \eqref{eq-arcs}, and the
state exchange relation \eqref{eq-stex}.
\begin{align}
\label{eq-arcs}
\relarc{+}{-}&=q^{-1/2}\relempty,\qquad
\relarc{+}{+}=\relarc{-}{-}=0,\\
\label{eq-stex}
\begin{linkdiag}
\fill[gray!20] (0,0)rectangle(1,1);\draw (1,0)--(1,1);
\draw [thick] (0,0.7)..controls(0.8,0.7) and (0.8,0.3)..(0,0.3);
\end{linkdiag}
&=q^{1/2}\relacross{-}{+}-q^{5/2}\relacross{+}{-}.
\end{align}
In these diagrams, the shaded region is a part of the surface, and the thin line on
the side is part of the boundary of the surface.

\begin{remark}
\label{rem-rel-motive}
The trivial arc relations \eqref{eq-arcs} as well as \eqref{eq-arctriv} below appeared in
\cite{BW:qtrace}. They are quantizations of the classical matrix entries of the
$\alpha$-matrix as in Equation~\eqref{Mist}. The powers of $q$ and the
relation~\eqref{eq-stex}
are required so that splitting in Section~\ref{sec-splitting} works, as explained in
\cite{Le:decomp}.

We define additional new quotients of the skein algebra with similar philosophy, where
certain arcs are set to scalars that quantize the classical values and make splitting
work. 
\end{remark}

\begin{lemma}[{\cite[Lemmas~2.3 and 2.4]{Le:decomp}}]
Relations \eqref{eq-arcs} and \eqref{eq-stex} imply the following.
\begin{enumerate}
\item Trivial arc relation.
\begin{equation}
\label{eq-arctriv}    
\relarc{-}{+}=-q^{-5/2}\relempty.
\end{equation}
\item Height exchange relations. 
\begin{equation}
\label{eq-heightex}
\relcross{\nu}{+}=q^{-\nu}\relacross{+}{\nu},\qquad
\relcross{-}{\nu}=q^\nu\relacross{\nu}{-}
\end{equation}
\end{enumerate}
for $\nu = \pm$.
\end{lemma}

A tangle diagram is \term{simple} if it does not have crossings, trivial arcs, or
trivial loops. Let $\ori$ be an orientation of $\partial\surface$. A diagram is
\term{$\ori$-ordered} if the height over each boundary edge is increasing as one
follows $\ori$. All diagrams so far have been positively ordered, meaning $\ori$ is
the positive orientation induced by $\surface$. Finally, a diagram $D$ is
\term{increasingly stated} if there are no endpoints on the same boundary edge such
that the higher endpoint has $-$ state and the lower endpoint has $+$ state.

\begin{theorem}[{\cite[Theorem~2.11]{Le:decomp}}]\label{thm-basis}
$\skein(\surface)$ is a free $R$-module. For any orientation $\ori$, the set of
$\ori$-ordered, increasingly stated, simple diagrams is a basis of $\skein(\surface)$.
\end{theorem}

Just as in the case of closed surfaces, the stated skein module $\skein(\surface)$
has an associative (in general, non-commutative) product given by stacking.
Specifically, given tangles $\alpha,\beta\in\skein(\surface)$, the product
$\alpha\cup\beta$ is defined as stacking $\alpha$ above $\beta$. What's more, it
has a $\BZ$-grading associated to each boundary edge $e$ of $\surface$, defined as
follows. For a tangle $\alpha$ with state $s:\partial\alpha\to\{\pm\}$, define
\begin{equation}
d_e(\alpha)=\sum_{x\in\alpha\cap e}s(x) \in \BZ \,.
\end{equation}
It is easy to see that $d_e$ preserves the defining relations~\eqref{eq-skein}
and~\eqref{eq-loop}, hence it becomes a $\ints$-grading on $\skein(\surface)$ that is
also compatible with the product $\cup$. 

Now we consider the reduced skein algebra $\skeinrd(\surface)$ defined in \cite{CL}.
It is the quotient of $\skein(\surface)$ by the (two-sided) ideal generated by the
\term{bad arcs}.
\begin{equation}
\label{eq-bad}
\relcorner{-}{+}=0.
\end{equation}
This relation holds for all possible heights on the boundaries of bad arcs. Same as
Remark~\ref{rem-rel-motive}, the motivation for this relation is the vanishing entry
in the $\beta$ matrix given in Equation~\eqref{Mbeta}.

By \cite[Corollary~4.6]{LeYu:qtrace}, the left, right, and two-sided ideal generated
by the bad arcs agree. For consistency with the quotients later, we usually consider
the left ideal.

The following useful result is obtained in the proof of \cite[Theorem~7.1]{CL}. 

\begin{lemma}
\label{lemma-bad}
In $\skeinrd(\surface)$, for a positively ordered diagram of the following form, we have
\begin{equation}
\begin{linkdiag}[0.8]
\fill[gray!20] (0,-0.2) rectangle (1,1.8); \draw (1,-0.2) -- (1,1.8);
\draw[fill=white] (1,1.4)circle(0.07) (1,-0.2)circle(0.07);
\begin{knot}
\strand[thick] (0,0.8) -- +(1,0) (0,1.1) -- +(1,0);
\strand[thick] (1,0.55)\stnode{+}..controls +(-0.8,0) and +(-0.8,0)..(1,1.6)\stnode{-};
\end{knot}
\draw[thick] (0,0.3) -- +(1,0) (0,0) -- +(1,0);
\path[tips,->] (0.8,0.95)node{...} (0.8,0.15)node{...} (1,0) -- +(0,1.3);
\end{linkdiag}
=0.
\end{equation}
\end{lemma}

\begin{theorem}[{\cite[Theorem~7.1]{CL}}]
\label{thm-basis-rd}
$\skeinrd(\surface)$ is a free $R$-module. A basis is given by the subset of
elements without bad arc components in the basis of $\skein(\surface)$ from
Theorem~\ref{thm-basis} with $\ori$ positive.
\end{theorem}

\subsection{Surfaces with triangular boundary}

The closed surface that appears in Proposition~\ref{prop.ST} is obtained by gluing
together punctured bordered \term{surface with triangular boundary}. The latter
are punctured bordered surfaces $\surface$ with each boundary component of
$\surclose$ containing three boundary edges of $\surface$. The three boundary
edges on the same component of $\partial\surclose$ form a \term{boundary triangle}.
Note that punctures in the interior of $\surclose$ are allowed.

\begin{figure}[htpb!]
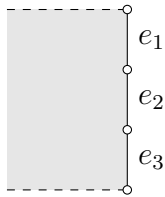

\centering
$\begin{anndiag}{2}{3}
\foreach \i in {1,2,3}
\path (2,{3.5-\i}) node[right]{$e_\i$};
\end{anndiag}$
\caption{Labeling boundary edges in a boundary triangle using the cyclic order
  induced by the orientation of the boundary triangle.}\label{fig-b-label}
\end{figure}

For a surface $\surface$ with triangular boundary, we need a different product
structure on $\skein(\surface)$, which we denote by $\cdot$. It is modified from $\cup$
by a power of $q$. The reason for this modification is to obtain the correct
quantization \eqref{eq-zzp}. See also Section~\ref{sub.qtorus}. Without stating
otherwise, we use $\cdot$ as the product for the rest of the paper.

Let $b:\ints^3\otimes\ints^3\to\ints$ be the unique skew-symmetric bilinear form
that is invariant under cyclic permutations of the components of $\ints^3$ such that
$b((1,0,0),(0,1,0))=1$. For each boundary triangle $E=\{e_1,e_2,e_3\}$, label the
edges as Figure~\ref{fig-b-label}. Then the gradings of a tangle $\alpha$ form a
vector $d_E(\alpha)=(d_{e_1}(\a),d_{e_2}(\a),d_{e_3}(\a)) \in\ints^3$. For tangles
$\alpha,\beta$, define a new product by
\begin{equation}
\label{eq-new-prod}
\alpha \cdot \beta=q^{-\frac{1}{2}\sum_{E}b(d_E(\alpha),d_E(\beta))}\alpha\cup\beta
\end{equation}
where the sum is over all boundary triangles $E$. It is easy to see that this extends
linearly to an associative product with the empty tangle as unit.

\subsection{Corner-reduced skein module}

In this section $\surface$ denotes a surface with triangular boundary. Let
$\skeincr(\surface)$
be the quotient of the reduced skein algebra $\skeinrd(\surface)$ by the left ideal
(under $\cdot$ product) generated by the following relations near marked points. 
\begin{equation}
\label{eq-crdef}
\relcorner{-}{-}=-\iunit q^{-1},\qquad
\relcorner{+}{+}=\iunit q,\qquad
\relcorner{+}{-}=1.
\end{equation}
As mentioned in Remark~\ref{rem-rel-motive}, the motivation for these relations are the
three nonzero entries of the $\beta$-matrix given in Equation~\eqref{Mbeta}. Keep in
mind that as in~\eqref{eq-bad} the heights in the above relations are arbitrary. Note
also that the boundary of the above arcs lies in different boundary edges since the
boundary of the above pictures is part of a triangle. By \cite[Proposition~7.4]{CL},
the first two are equivalent in $\skeinrd(\surface)$.

\begin{remark}
\label{remark-ichoice}
Because we do not specify which 4-th root of unity $\iunit$ is, technically
$\skeincr(\surface)$ depends on this choice as well, and it is independent of the
choice of $q$. In this paper, the only time this is relevant is in Section
~\ref{sub.2frob}, where we discuss roots of unity.
\end{remark}

\begin{lemma}
\label{lemma-prod-corner}
In $\skeincr(\surface)$, the lowest endpoint on a boundary edge can be modified in
the following way.
\begin{align}
\label{eq-prod-mm}
\begin{linkdiag}[0.9]
\fill[gray!20] (0,0) rectangle (1,1.8); \draw (1,0) -- (1,1.8);
\draw[fill=white] (1,0)circle(0.07) (1,1.3)circle(0.07);
\draw[thick] (0,0.5) -- +(1,0) (0,1) -- +(1,0);
\draw[thick] (0,0.2) -- +(1,0) \stnode{+};
\path[tips,->] (0.8,0.75)node[rotate=90]{...} (1,0) -- +(0,1.2);
\end{linkdiag}
&=\iunit q^{\frac{1}{2}(d-d'+3)}
\begin{linkdiag}[0.9]
\fill[gray!20] (0,0) rectangle (1,1.8); \draw (1,0) -- (1,1.8);
\draw[fill=white] (1,0)circle(0.07) (1,1.3)circle(0.07);
\begin{knot}
\strand[thick] (0,0.5) -- +(1,0) (0,1) -- +(1,0);
\strand[thick] (0,0.2)..controls+(0.5,0) and +(-0.5,0)..(1,1.6)\stnode{-};
\end{knot}
\path[tips,->] (0.8,0.75)node[rotate=90]{...} (1,0) -- +(0,1.2);
\end{linkdiag},\\
\label{eq-prod-mixed}
\begin{linkdiag}[0.9]
\fill[gray!20] (0,0) rectangle (1,1.8); \draw (1,0) -- (1,1.8);
\draw[fill=white] (1,0)circle(0.07) (1,1.3)circle(0.07);
\draw[thick] (0,0.5) -- +(1,0) (0,1) -- +(1,0);
\draw[thick] (0,1.6) -- +(1,0) \stnode{-};
\path[tips,->] (0.8,0.75)node[rotate=90]{...} (1,0) -- +(0,1.2);
\end{linkdiag}
&=-\iunit q^{d'-d''}
\begin{linkdiag}[0.9]
\fill[gray!20] (0,0) rectangle (1,1.8); \draw (1,0) -- (1,1.8);
\draw[fill=white] (1,0)circle(0.07) (1,1.3)circle(0.07);
\draw[thick] (0,0.5) -- +(1,0) (0,1) -- +(1,0);
\draw[thick] (0,1.6) -- +(1,0) \stnode{+};
\path[tips,->] (0.8,0.75)node[rotate=90]{...} (1,0) -- +(0,1.2);
\end{linkdiag}
+q^{\frac{1}{2}(d+d'-2d''-1)}
\begin{linkdiag}[0.9]
\fill[gray!20] (0,0) rectangle (1,1.8); \draw (1,0) -- (1,1.8);
\draw[fill=white] (1,0)circle(0.07) (1,1.3)circle(0.07);
\begin{knot}
\strand[thick] (0,0.5) -- +(1,0) (0,1) -- +(1,0);
\strand[thick] (0,1.6)..controls+(0.5,0) and +(-0.5,0)..(1,0.2)\stnode{-};
\end{knot}
\path[tips,->] (0.8,0.75)node[rotate=90]{...} (1,0) -- +(0,1.2);
\end{linkdiag}.
\end{align}
Here, $d$ and $d'$ are the gradings of the left-hand sides on the top and bottom
edges, respectively, and $d''$ is the grading on the third edge in the same boundary
triangle (not shown in the diagrams).
\end{lemma}

\begin{proof}
Let $\alpha$ denote the left-hand side of \eqref{eq-prod-mm}. First we use the
definition of the product to combine the diagrams. Then we resolve the lowest crossing.
\begin{align*}
\alpha&=\iunit q \alpha\cdot\relcorner{-}{-}
=(\iunit q)q^{-\frac{1}{2}(d'-d)}\alpha\cup\relcorner{-}{-}\\
&=\iunit q^{\frac{1}{2}(d-d')+1}
\begin{linkdiag}[0.75]
\fill[gray!20] (0,-0.3) rectangle (1,1.8); \draw (1,-0.3) -- (1,1.8);
\draw[fill=white] (1,-0.3)circle(0.07) (1,1.3)circle(0.07);
\begin{knot}
\strand[thick] (0,0.5) -- +(1,0) (0,1) -- +(1,0);
\strand[thick] (0,0.2) -- +(1,0)\stnode{+};
\strand[thick] (1,-0.1)\stnode{-}..controls +(-0.8,0) and +(-0.8,0)..(1,1.6)\stnode{-};
\end{knot}
\path[tips,->] (0.8,0.75)node[rotate=90]{...} (1,0) -- +(0,1.2);
\end{linkdiag}
=\iunit q^{\frac{1}{2}(d-d')+1}\left(
q
\begin{linkdiag}[0.75]
\fill[gray!20] (0,-0.3) rectangle (1,1.8); \draw (1,-0.3) -- (1,1.8);
\draw[fill=white] (1,-0.3)circle(0.07) (1,1.3)circle(0.07);
\begin{knot}
\strand[thick] (0,0.5) -- +(1,0) (0,1) -- +(1,0);
\strand[thick] (0,0.2)..controls +(0.5,0) and +(-0.5,0)..(1,1.6)\stnode{-};
\end{knot}
\draw[thick] (1,-0.1)\stnode{-}..controls +(-0.5,0) and +(-0.5,0)..(1,0.2)\stnode{+};
\path[tips,->] (0.8,0.75)node[rotate=90]{...} (1,0) -- +(0,1.2);
\end{linkdiag}
+q^{-1}
\begin{linkdiag}[0.75]
\fill[gray!20] (0,-0.3) rectangle (1,1.8); \draw (1,-0.3) -- (1,1.8);
\draw[fill=white] (1,-0.3)circle(0.07) (1,1.3)circle(0.07);
\begin{knot}
\strand[thick] (0,0.5) -- +(1,0) (0,1) -- +(1,0);
\strand[thick] (1,0.2)\stnode{+}..controls +(-0.6,0) and +(-0.6,0)..(1,1.6)\stnode{-};
\end{knot}
\draw[thick] (0,0.2)..controls +(0.5,0) and +(-0.5,0)..(1,-0.1)\stnode{-};
\path[tips,->] (0.8,0.75)node[rotate=90]{...} (1,0) -- +(0,1.2);
\end{linkdiag}
\right).
\end{align*}
The second term is zero by Lemma~\ref{lemma-bad}. The first term simplifies
to \eqref{eq-prod-mm}.

For \eqref{eq-prod-mixed}, still let $\alpha$ denote the left-hand side. Then
\begin{equation*}
\alpha=\alpha\cdot\relcorner{+}{-}
=q^{\frac{1}{2}(d+d'-2d'')}
\begin{linkdiag}[1.05]
\fill[gray!20] (0,0) rectangle (1,2.1); \draw (1,0) -- (1,2.1);
\draw[fill=white] (1,0)circle(0.07) (1,1.3)circle(0.07);
\begin{knot}
\strand[thick] (0,0.5) -- +(1,0) (0,1) -- +(1,0);
\strand[thick] (0,1.9) -- +(1,0)\stnode{-};
\strand[thick] (1,0.2)\stnode{-}..controls +(-0.8,0) and +(-0.8,0)..(1,1.6)\stnode{+};
\end{knot}
\path[tips,->] (0.8,0.75)node[rotate=90]{...} (1,0) -- +(0,1.2);
\end{linkdiag}
=q^{\frac{1}{2}(d+d'-2d'')}\left(
q^2
\begin{linkdiag}[1.05]
\fill[gray!20] (0,0) rectangle (1,2.1); \draw (1,0) -- (1,2.1);
\draw[fill=white] (1,0)circle(0.07) (1,1.3)circle(0.07);
\begin{knot}
\strand[thick] (0,0.5) -- +(1,0) (0,1) -- +(1,0);
\strand[thick] (0,1.9) -- +(1,0)\stnode{+};
\strand[thick] (1,0.2)\stnode{-}..controls +(-0.8,0) and +(-0.8,0)..(1,1.6)\stnode{-};
\end{knot}
\path[tips,->] (0.8,0.75)node[rotate=90]{...} (1,0) -- +(0,1.2);
\end{linkdiag}
+q^{-1/2}
\begin{linkdiag}[1.05]
\fill[gray!20] (0,0) rectangle (1,2.1); \draw (1,0) -- (1,2.1);
\draw[fill=white] (1,0)circle(0.07) (1,1.3)circle(0.07);
\begin{knot}
\strand[thick] (0,0.5) -- +(1,0) (0,1) -- +(1,0);
\strand[thick] (0,1.9)..controls +(0.8,0) and +(-0.8,0)..(1,0.2)\stnode{-};
\end{knot}
\path[tips,->] (0.8,0.75)node[rotate=90]{...} (1,0) -- +(0,1.2);
\end{linkdiag}
\right).
\end{equation*}
The second term is in the desired form. The first term is a $\cup$-product. By
rewriting in terms of the new product, we obtain \eqref{eq-prod-mixed}.
\end{proof}

In Lemma~\ref{lemma-prod-corner}, there are strands unchanged by the relations. They
are always higher than the strands changed by the relations, which is the result of
the quotient by a left ideal. We will say such relations in $\skeincr(\surface)$ holds
at the bottom and omit the unchanged strands. By rewriting the relations in the lemma,
we obtain the following.

\begin{corollary}
\label{cor-tw}
In $\skeincr(\surface)$, the following relations hold at the bottom.
\begin{align}
\label{eq-twup}
\relbottom{+}&=\iunit q^{\frac{1}{2}(d-d'+3)}\reltwup{-},&
\relbottom{-}&=\iunit q^{\frac{1}{2}(d'-d+3)}\reltwup{+}+q^{\frac{1}{2}(2d''-d-d'+1)}
\reltwup{-}.\\
\label{eq-twdown}
\reltwup{-}&=-\iunit q^{\frac{1}{2}(d'-d-3)}\relbottom{+},&
\reltwup{+}&=q^{\frac{1}{2}(2d''-d-d'-1)}\relbottom{+}-\iunit q^{\frac{1}{2}(d-d'-3)}
\relbottom{-}.
\end{align}
Here, $d,d',d''$ have the same meaning as in Lemma~\ref{lemma-prod-corner}.
\end{corollary}

Note that these moves do not preserve the gradings $d$ and $d'$. Although \eqref{eq-twup}
and \eqref{eq-twdown} are inverses, a naive substitution does not show this, because the
gradings need to be adjusted.

\newcommand{\relaround}[1]{
\begin{linkdiag}
\fill[gray!20] (-0.2,-0.2)rectangle(1.2,1.2);
\draw[fill=white] (0.5,0.5)circle(0.25);
\foreach \t in {0,120,240}
\draw[fill=white] (0.5,0.5) +(\t:0.25)circle(0.07);
\draw[thick] (-0.2,0.5)..controls +(0.35,0) and +(-0.35,0)..(0.5,#1)
..controls +(0.35,0) and +(-0.35,0).. (1.2,0.5);
\end{linkdiag}
}

\begin{lemma}
\label{lemma-crhandle}
Suppose $\surface$ is a surface with triangular boundary.
In $\skeincr(\surface)$, the following handle slide moves hold at the bottom.
\begin{equation}
\begin{linkdiag}
\fill[gray!20] (-0.2,-0.2)rectangle(1.2,1.2);
\draw[fill=white] (0.5,0.5)circle(0.3);
\foreach \t in {0,120,240}
\draw[fill=white] (0.5,0.5) +(\t:0.3)circle(0.07);
\draw[thick] (-0.2,0.5) ..controls (0.1,0.5) and (-0.1,0)..(0.5,0)
arc[radius=0.5,start angle=-90,delta angle=180]
..controls (-0.1,1) and (-0.1,0.5).. (0.2,0.5)\stnode{\mu};
\end{linkdiag}
=(-q^3)
\begin{linkdiag}
\fill[gray!20] (-0.2,-0.2)rectangle(1.2,1.2);
\draw[fill=white] (0.5,0.5)circle(0.3);
\foreach \t in {0,120,240}
\draw[fill=white] (0.5,0.5) +(\t:0.3)circle(0.07);
\draw[thick] (-0.2,0.5) -- (0.2,0.5)\stnode{\mu};
\end{linkdiag},
\qquad
\relaround{1}=\relaround{0}.
\end{equation}
\end{lemma}

\begin{proof}
The first identity is obtained by applying \eqref{eq-twup} three times. For the
second one, we redraw it as
\begin{equation}
\begin{anndiag}{1}{2}
\draw[thick] (0,0.5)..controls +(0.5,0)..(0.5,0)
(0,1.5)..controls +(0.5,0)..(0.5,2);
\end{anndiag}
=
\begin{anndiag}{1}{2}
\draw[thick] (0,0.5)..controls +(0.5,0)..(0.5,1)..controls +(0,0.5)..(0,1.5);
\end{anndiag}.
\end{equation}
To prove this, we combine the first identity with \eqref{eq-stex}.
\begin{align*}
\begin{anndiag}[u]{1}{2.2}
\draw[thick] (0,0.6)..controls +(0.5,0)..(0.5,0)
(0,1.8)..controls +(0.5,0)..(0.5,2.2);
\end{anndiag}
&=
q^{1/2}
\begin{anndiag}[u]{1}{2.2}
\draw[thick] (0,0.55) -- +(1,0)\stnode{-}
(1,0.25)\stnode{+}..controls +(-0.5,0)..(0.5,0)
(0,1.8)..controls +(0.5,0)..(0.5,2.2);
\path[tips,->] (1,0) -- (1,0.75);
\end{anndiag}
-q^{5/2}
\begin{anndiag}[u]{1}{2.2}
\draw[thick] (0,0.55) -- +(1,0)\stnode{+}
(1,0.25)\stnode{-}..controls +(-0.5,0)..(0.5,0)
(0,1.8)..controls +(0.5,0)..(0.5,2.2);
\path[tips,->] (1,0) -- (1,0.75);
\end{anndiag}
=(-q^3)^{-1}
\left(
q^{1/2}
\begin{anndiag}[u]{1}{2.2}
\begin{scope}[thick]
\begin{knot}
\strand (0,0.55) -- +(1,0)\stnode{-};
\strand (1,0.25)\stnode{+}..controls +(-0.4,0)..(0.6,2.2);
\end{knot}
\draw (0.6,0)..controls +(0,0.4) and +(0,0.4)..(0.4,0)
(0,1.8)..controls +(0.4,0)..(0.4,2.2);
\end{scope}
\path[tips,->] (1,0) -- (1,0.75);
\end{anndiag}
-q^{5/2}
\begin{anndiag}[u]{1}{2.2}
\begin{scope}[thick]
\begin{knot}
\strand (0,0.55) -- +(1,0)\stnode{+};
\strand (1,0.25)\stnode{-}..controls +(-0.4,0)..(0.6,2.2);
\end{knot}
\draw (0.6,0)..controls +(0,0.4) and +(0,0.4)..(0.4,0)
(0,1.8)..controls +(0.4,0)..(0.4,2.2);
\end{scope}
\path[tips,->] (1,0) -- (1,0.75);
\end{anndiag}
\right)\\
&=(-q^3)^{-1}
\begin{anndiag}[u]{1}{2.2}
\begin{knot}
\strand[thick] (0,0.6) -- (0.7,0.6);
\strand[thick] (0,1.8)..controls +(0.2,0)..(0.2,2)
arc[radius=0.1,start angle=180,end angle=0] -- (0.4,0.4)
to[out angle=-90,curve through={(0.9,0.4)},in angle=0] (0.7,0.6);
\end{knot}
\end{anndiag}
=
\begin{anndiag}[u]{1}{2.2}
\draw[thick] (0,0.5)..controls +(0.5,0)..(0.5,1.2)..controls +(0,0.6)..(0,1.8);
\end{anndiag}.\qedhere
\end{align*}
\end{proof}

\subsection{Splitting}
\label{sec-splitting}

The stated skein algebra and its reduced version have splitting homomorphisms
connecting the punctured bordered surfaces before and after splitting, thus
reducing the surfaces to elementary pieces, namely a standard monogon, bigon, and
triangle shown in Figure~\ref{fig-elem-pb}. We briefly recall how this works.

\begin{figure}[htpb!]
\begin{tikzpicture}
\draw[fill=gray!20] (0,0)..controls (-1,0.6) and (-1,2)..(0,2)
..controls (1,2) and (1,0.6)..cycle
(3,0)..controls (2,0.6) and (2,1.4)..(3,2)
..controls (4,1.4) and (4,0.6)..cycle
(5.5,0) -- +({4/sqrt(3)},0) coordinate(A) -- +({2/sqrt(3)},2) coordinate(B)
        -- cycle;
\draw[fill=white] (0,0)circle(0.05) (3,0)circle(0.05) (3,2)circle(0.05)
(5.5,0)circle(0.05) (A)circle(0.05) (B)circle(0.05);
\end{tikzpicture}
\caption{Elementary surfaces: monogon, bigon, and triangle.}\label{fig-elem-pb}
\end{figure}
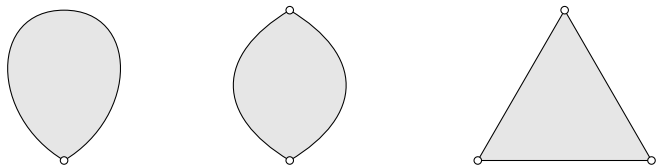

Let $\surface=\surclose\setminus\marked$ be a punctured bordered surface. An
\term{ideal arc} 
on $\surface$ is a simple arc $e:[0,1]\to\surclose$ such that $e(0),e(1)\in\marked$
and $e((0,1))$ is in the interior of $\surface$. Splitting $\surface$ along $e$
produces a new punctured bordered surface denoted $\surface_e$. There is a quotient
map $p:\surface_e\to\surface$ gluing the new boundary edges back. Let
$\thicken{p}:\surface_e\times(-1,1)\to\surface\times(-1,1)$ be the induced map.

Fix an ideal arc $e$ on a punctured bordered surface $\surface$, and a tangle
$\alpha$ over $\surface$. Isotope $\alpha$ such that it is transverse to
$\thicken{e}=e\times(-1,1)$ and the heights of $\alpha\cap\thicken{e}$ are distinct,
and define
\begin{equation}
\label{eq-le-split}
\cut_e(\alpha)=\sum_{s:\alpha\cap\thicken{e}\to\{\pm\}}(\alpha,s),
\end{equation}
where $(\alpha,s)$ is the tangle $\thicken{p}^{-1}(\alpha)$ with the state $s(x)$
assigned to both points in $\thicken{p}^{-1}(x)$ for all $x\in\alpha\cap\thicken{e}$.

\begin{theorem}[{\cite[Theorem~3.1]{Le:decomp}, \cite[Thm.~7.6]{CL}}]
\label{thm-le-split}
Given an ideal arc $e$ on a punctured bordered surface $\surface$, \eqref{eq-le-split}
is a well-defined algebra homomorphism (with respect to $\cup$)
\begin{equation}\label{cute}
\cut_e:\skein(\surface)\to\skein(\surface_e)
\end{equation}
that satisfies
\begin{equation}
\cut_e\circ\cut_f=\cut_f\circ\cut_e
\end{equation}
if $e$ and $f$ are disjoint ideal arcs in $\surface$, and descends to an algebra
homomorphism
\begin{equation}
\cut_e : \skeinrd(\surface)\to\skeinrd(\surface_e),
\end{equation}
denoted by $\cut_e$ by abuse of notation.
\end{theorem}

Now suppose $\surface=\surclose\setminus\marked$ has triangular boundary. When we
consider the corner reduced skein module $\skeincr(\surface)$, we can split along a
closed curve $c$ in the interior of $\surface$ with three distinguished points
$p_1,p_2,p_3\in c$. Let $\surface'_c=\surface\setminus\{p_1,p_2,p_3\}$.
Then $c\cap\surface'_c$ is the union of three ideal arcs in $\surface'_c$.
$\surface'_c$ split along these arcs is a surface $\surface_c$ with triangular
boundary. Let $\cut:\skein(\surface'_c)\to\skein(\surface_c)$ be the composition of
splits.

\begin{theorem}
\label{thm-split}
The composition of splits $\cut:\skein(\surface'_c)\to\skein(\surface_c)$ induces
an $R$-module homomorphism $\cut_c:\skeincr(\surface)\to\skeincr(\surface_c)$.
\begin{equation}\label{eq-splitcr}
\begin{tikzcd}[column sep=large,row sep=small]
\skein(\surface'_c) \arrow[r,"\cut"] \arrow[d,two heads] & \skein(\surface_c)
\arrow[dd,two heads]\\
\skein(\surface) \arrow[d,two heads]\\
\skeincr(\surface) \arrow[r,"\cut_c"] & \skeincr(\surface_c)
\end{tikzcd}
\end{equation}
\end{theorem}

\begin{proof}
First consider the descent to $\skein(\surface)\to\skeincr(\surface_c)$. The map
$\skein(\surface'_c)\onto\skein(\surface)$ is induced by the inclusion
$\surface'_c\embed\surface$. The kernel of this quotient is generated by isotopies
across the punctures $p_1,p_2,p_3$, or in terms of diagrams,
\begin{equation}
\begin{linkdiag}
\fill[gray!20] (0,0) rectangle (2,1);
\draw (1,0) -- (1,1);
\draw[thick] (0,0.2) -- (2,0.2);
\draw[fill=white] (1,0.5)circle(0.07) node[inner sep=2pt,above left]{\small$p_i$};
\end{linkdiag}
=
\begin{linkdiag}
\fill[gray!20] (0,0) rectangle (2,1);
\draw (1,0) -- (1,1);
\draw[thick] (0,0.2)..controls +(0.5,0) and +(-0.5,0)..(1,0.8)
..controls +(0.5,0) and +(-0.5,0)..(2,0.2);
\draw[fill=white] (1,0.5)circle(0.07);
\end{linkdiag}.
\end{equation}
We need to show that the splitting of both sides are equal in $\skeincr(\surface_c)$.
We can always isotope the tangle such that the intersection with
$\tilde{c}=c\times(-1,1)$ is the lowest. Thus, we only need to check the equality
at the bottom.
\begin{align*}
&\quad\cut_c\left(
\begin{linkdiag}
\fill[gray!20] (0,0) rectangle (2,1);
\draw (1,0) -- (1,1);
\draw[thick] (0,0.2)..controls +(0.5,0) and +(-0.5,0)..(1,0.8)
..controls +(0.5,0) and +(-0.5,0)..(2,0.2);
\draw[fill=white] (1,0.5)circle(0.07);
\end{linkdiag}\right)
=\reltwup{+}\reltwupl{+}+\reltwup{-}\reltwupl{-}\\
&=
\left(q^{\frac{1}{2}(2d_+''-d_+-d_+'-1)}\relbottom{+}-\iunit q^{\frac{1}{2}(d_+-d_+'-3)}
\relbottom{-}\right)\left(\iunit q^{\frac{1}{2}(d_+'-d_++3)}\relbottoml{-}\right)+\\
&\quad+\left(-\iunit q^{\frac{1}{2}(d_-'-d_--3)}\relbottom{+}\right)
\left(\iunit q^{\frac{1}{2}(d_--d_-'+3)}\relbottoml{+}+q^{\frac{1}{2}(d_-''-d_--d_-'+1)}
\relbottoml{-}\right)\\
&=\relbottom{+}\relbottoml{+}+\relbottom{-}\relbottoml{-}
= \cut_c\left(
\begin{linkdiag}
\fill[gray!20] (0,0) rectangle (2,1);
\draw (1,0) -- +(0,1);
\draw[thick] (0,0.2) -- (2,0.2);
\draw[fill=white] (1,0.5)circle(0.07);
\end{linkdiag}\right).
\end{align*}
Here, Corollary~\ref{cor-tw} is applied to every factor, where $d_\pm,d_\pm',d_\pm''$
are the gradings of $\reltwup{\pm}$. Then $d_+=d_-+2$ and $d_+'=d_-',d_+''=d_-''$. Note
for $\reltwupl{\pm}$, a rotation is necessary for Corollary~\ref{cor-tw} to apply,
which is why the primes do not match.
Note in the right half of each term, the direction of twist is opposite of the left
half because of the orientation. This shows $\skein(\surface)\to\skeincr(\surface_c)$
is well-defined.

The descent to $\cut_c:\skeincr(\surface)\to\skeincr(\surface_c)$ is trivial. The
splitting homomorphism only affects a neighborhood of $c$, whereas the quotient
$\skein(\surface)\onto\skeincr(\surface)$ happens near $\partial\surface$.
\end{proof}

If the split surface $\surface_c=\surface_1\sqcup\surface_2$ is disconnected, then
$\skeincr(\surface_c)$ is naturally isomorphic to
$\skeincr(\surface_1)\otimes\skeincr(\surface_2)$. In this case, the splitting
homomorphism has the form
\begin{equation}
\cut_c:\skeincr(\surface)\to\skeincr(\surface_1)\otimes\skeincr(\surface_2).
\end{equation}

\subsection{The skein of the bigon and action of $\Oq$}
\label{sec-Oq}

The splitting Theorem~\ref{thm-le-split} of the previous section reduces the study of
the stated skein module of a surface to that of an elementary surface. In this and
in the next sections we study three examples of elementary punctured bordered surfaces,
namely the standard bigon, annulus and the lantern.

We begin with the bigon $\bigon$ shown in Figure~\ref{fig-arc-bigon}. It contains
the tangle $a_{\mu\nu}$ with states $\mu,\nu$ at $a(0),a(1)$ respectively.

\begin{figure}[htpb!]
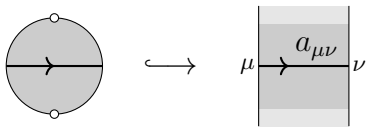

\[
\begin{linkdiag}
\draw[fill=gray!40] (0.5,0.5)circle(0.8);
\draw[fill=white] (0.5,-0.3)circle(0.07) (0.5,1.3)circle(0.07);
\draw[thick] (-0.3,0.5) -- (1.3,0.5);
\path[tips,thick,->] (-0.3,0.5) -- (0.5,0.5);
\end{linkdiag}
\quad\lhook\joinrel\longrightarrow\quad
\begin{linkdiag}
\fill[gray!20] (0,-0.5) rectangle (1.5,1.5);
\fill[gray!40] (0,-0.2) rectangle (1.5,1.2);
\draw (0,-0.5) -- (0,1.5) (1.5,-0.5) -- (1.5,1.5);
\draw[thick] (0,0.5)\stnodel{\mu} -- (1.5,0.5)\stnode{\nu};
\path[tips,thick,->] (0,0.5) -- (0.5,0.5) node[above right,inner sep=2pt]{$a_{\mu\nu}$};
\end{linkdiag}
\]
\caption{Bigon neighborhood of an arc.}\label{fig-arc-bigon}
\end{figure}

The skein module $\skein(\bigon)$ of the bigon can be identified with the quantum
matrix group $\Oq$ generated by $a$, $b$, $c$, $d$ with relations 
\begin{equation}
\label{eq-Oq-def}
\begin{aligned}
ca &=q^2 ac, \qquad db =q^2 bd, \qquad ba =q^2 ab, \qquad dc =q^2 cd \\
bc &= cb, \qquad ad-q^{-2}bc =1, \qquad da-q^2cb =1 \,.
\end{aligned}
\end{equation}
See \cite{Majid} and also \cite[Defn.~1]{CL}. Explicitly, 
in~\cite[Theorem~3.4]{CL}, it was shown that there is an isomorphism
\be
\label{bigon-iso}
\Oq \to \skein(\bigon)
\ee
given by mapping $a$, $b$, $c$, $d$ of $\Oq$ to
$a_{++}$, $a_{+-}$, $a_{-+}$ and $a_{--}$, and checking that the defining
relations~\eqref{eq-Oq-def} hold true.

The skein of a bigon acts on the skein of a punctured bordered surface $\surface$
as follows. Fix an oriented simple arc $a:[0,1]\to\surface$ whose endpoints are on
different boundary triangles of $\surface$. The inclusion of the bigon induces an
$R$-module embedding
\begin{equation}
\label{phiasurface}
\phi_a: \Oq\embed\skein(\surface) \,.
\end{equation}
By the assumption that $a$ ends on distinct boundary triangles, the factor in
\eqref{eq-new-prod} is trivial.
Therefore, the two product structures match.
Since $\skeincr(\surface)$ is a left $\skein(\surface)$-module, it induces a left
$\Oq$-module structure.

This module structure is compatible with splitting in the following sense. Suppose
$c$ is a simple closed curve in the interior of $\surface$ that intersects the arc
$a$ once transversely. Let $a',a''$ be the arcs on $\surface_c$ obtained by splitting
$a$ at $a\cap c$. Then the arcs $a',a''$ induces an $(\Oq\otimes\Oq)$-action on
$\skeincr(\surface_c)$.

\begin{lemma}
\label{lemma-oq-compatible}
The following diagram is commutative
\begin{equation}
\begin{tikzcd}
\Oq\otimes\skeincr(\surface) \arrow[r,"\Delta\otimes\cut_c"] \arrow[d] &
(\Oq\otimes\Oq)\otimes\skeincr(\surface_c) \arrow[d] \\
\skeincr(\surface) \arrow[r,"\cut_c"] & \skeincr(\surface_c)
\end{tikzcd}
\end{equation}
Here, the vertical arrows are the module actions, and $\Delta$ is the coproduct of
$\Oq$.
\end{lemma}

\begin{proof}
This is obvious using the identification of $\Delta$ with the unique splitting
homomorphism of the bigon, which is proved in \cite[Theorem~3.4]{CL}.
\end{proof}

\subsection{The skein of the annulus}

Fix 3 points $p_1,p_2,p_3\in S^1$. Then the annulus $S^1\times[0,1]$ becomes a
surface with triangular boundary $\annulus=(S^1\times[0,1])\setminus\marked$ where
$\marked=\{p_1,p_2,p_3\}\times\{0,1\}$. We call $\annulus$ the \term{standard annulus}.

If we split $\annulus$ along its core $S^1\times{1/2}$ and choose the new marked
points at $\{p_1,p_2,p_3\}\times{1/2}$, then the two components are canonically
identified with standard annulus. Then the splitting homomorphism is a comultiplication
\begin{equation}
\Delta:\skeincr(\annulus)\to\skeincr(\annulus)\otimes\skeincr(\annulus).
\end{equation}
It is coassociative by the commuting property of the splitting homomorphism.

Fix $p\in S^1\setminus\{p_1,p_2,p_3\}$, and let $a:[0,1]\to\annulus$, $a(t)=(p,t)$. In
the last section, we defined a left $\Oq$-action on $\skeincr(\annulus)$ along $a$. By
Lemma~\ref{lemma-oq-compatible}, the action is compatible with the comultiplication.
Thus, $\skeincr(\annulus)$ is a left $\Oq$-module-coalgebra.

\begin{figure}[htpb!]
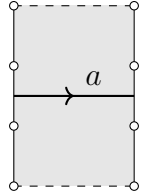

\centering
$\begin{anndiag}{2}{3}
\draw (0,0) -- (0,3);
\foreach \y in {0,1,2,3}
\draw[fill=white] (0,\y)circle(0.07);
\draw[thick] (0,1.5) -- (2,1.5);
\path[tips,thick,->] (0,1.5) -- (1,1.5) node[above right]{$a$};
\end{anndiag}$
\caption{Standard annulus.}\label{fig-std-ann}
\end{figure}

\begin{theorem}
\label{thm-stdA}
The action of $\Oq$ on the empty diagram defines an isomorphism of left
$\Oq$-module-coalgebras.
\begin{equation}
\phi_a:\Oq\to\skeincr(\annulus).
\end{equation}
\end{theorem}

The proof is given in Appendix~\ref{sec-stdA-iso}.

\begin{corollary}
\label{cor-comod}
The comultiplication $\Delta$ on $\skeincr(\annulus)$ has a counit such that
\begin{equation}\label{eq-counit}
\epsilon(a_{\mu\nu})=\delta_{\mu\nu},
\end{equation}
where $a_{\mu\nu}$ is defined in the last section.

Suppose $\surface$ is a surface with triangular boundary, and $c$ is a simple closed
curve parallel to a boundary triangle. Then the splitting homomorphism along $c$ is
a $\skeincr(\annulus)$-comodule structure for $\skeincr(\surface)$.
\end{corollary}

\begin{proof}
Since $\Oq$ has a counit, it can be transferred to $\skeincr(\annulus)$, which gives
\eqref{eq-counit}.

The second part is the same as \cite[Proposition~4.1(a)]{CL}.
\end{proof}

Note the isomorphism $\phi_a$ depends on the choice of the point $p$ that determines the
arc $a$. However, the corollary does not, since the counit $\epsilon$ is uniquely
determined by the comultiplication $\Delta$. The isomorphism is used to show that
\eqref{eq-counit} is sufficient to define the counit.

The corollary is very useful for calculation. We give one example which is used later.

\newcommand{\reldist}[3][]{
\begin{linkdiag}
\fill[gray!20] (0,-0.1)rectangle(1,1.1);\draw[#1] (1,-0.1)--(1,1.1);
\draw[thick] (0,0.8)--(1,0.8) (0,0.2)--(1,0.2);
\draw[fill=white] (1,0.5)circle(0.07);
\draw (1,0.8)\stnode{#2} (1,0.2)\stnode{#3};
\end{linkdiag}
}

\begin{lemma}
In $\skeincr(\surface)$, we have
\begin{equation}
\label{eq-cup-puncture}
\begin{linkdiag}
\fill[gray!20] (0,-0.1)rectangle(1,1.1);\draw (1,-0.1)--(1,1.1);
\draw[thick] (0,0.8)..controls +(0.8,0) and +(0.8,0)..(0,0.2);
\draw[fill=white] (1,0.5)circle(0.07);
\end{linkdiag}
=\iunit q\reldist{+}{+}+\reldist{-}{+}-\iunit q^{-1}\reldist{-}{-}.
\end{equation}
Here, the left-hand side has no endpoints on the boundary triangle.
\end{lemma}

\begin{proof}
Apply the comodule structure and then the counit.
\begin{equation}
\begin{linkdiag}
\fill[gray!20] (0,-0.1)rectangle(1,1.1);
\draw (1,-0.1)--(1,1.1); \draw[dashed] (0.3,-0.1) -- (0.3,1.1);
\draw[thick] (0,0.8)..controls +(0.8,0) and +(0.8,0)..(0,0.2);
\draw[fill=white] (0.3,0.5)circle(0.07) (1,0.5)circle(0.07);
\end{linkdiag}
=\sum_{\mu,\nu}\reldist{\mu}{\nu}\epsilon\left(
\begin{linkdiag}
\fill[gray!20] (0,-0.1)rectangle(1,1.1);
\draw (0,-0.1)--(0,1.1) (1,-0.1)--(1,1.1);
\draw[thick] (0,0.8)\stnodel{\mu}..controls +(0.8,0) and +(0.8,0)..(0,0.2)\stnodel{\nu};
\draw[fill=white] (0,0.5)circle(0.07) (1,0.5)circle(0.07);
\end{linkdiag}
\right).
\end{equation}
The tangle in the counit evaluates to scalars by the defining relations of $\skeincr$,
including the bad arc relation of $\skeinrd$. The lemma is obtained after substitution.
\end{proof}

\subsection{Interlude: basics on the quantum torus}
\label{sub.qtorus}

In this section we include a short discussion on the quantum torus, which is an
example of a skew Laurent polynomial ring. For more details, see e.g.~\cite{GW}.

Given a skew-symmetric $r\times r$ matrix $B$ with integer entries, the quantum torus
$\qtorus(B)$ is defined by
\begin{equation}
\qtorus(B)=R\langle x_1^{\pm1},\dotsc,x_r^{\pm1} \rangle/\ideal{x_ix_j-q^{B_{ij}}x_jx_i}.
\end{equation}
If we use notations other than $x_j$ for the generators and specify the $q$-commuting
relations elsewhere, we simply write the generators as in the introduction.

The quantum torus is an associative, and in general non-commutative algebra with
unit. Additively, there is an $R$-linear isomorphism from the Laurent polynomial ring
\begin{equation}
R[t_1^{\pm1},\dotsc,t_r^{\pm1}]\xrightarrow{\cong}\qtorus(B).
\end{equation}
The image of the monomial $t_1^{k_1}\dotsm t_r^{k_r}$ is the \term{Weyl-ordered monomial}
in $\qtorus(B)$.
\begin{equation}
x^{\vexp{k}}=q^{-\frac{1}{2}\sum_{i<j}B_{ij}}x_1^{k_1}\dotsm x_r^{k_r},
\qquad \vexp{k}=(k_1,\dotsc,k_r)\in\ints^r.
\end{equation}
These monomial $q$-commute according to the bilinear form $\langle\cdot,\cdot\rangle_B$
associated to $B$ on $\ints^r$.
\begin{equation}\label{eq-qtorus-prod}
x^{\vexp{k}}x^{\vexp{l}}=q^{\frac{1}{2}\langle\vexp{k},\vexp{l}\rangle_B}x^{\vexp{k}+\vexp{l}}
=q^{\langle\vexp{k},\vexp{l}\rangle_B}x^{\vexp{k}}x^{\vexp{l}}.
\end{equation}
This normalization can be formally understood as the Baker-Campbell-Hausdorff formula
for $e^{\sum_i k_i\ln(x_i)}$.

Note that Weyl-ordering can be defined for all products with $q$-commuting factors.
Weyl-ordering is also related to the product structures on the skein algebra
$\skein(\surface)$ of a surface $\surface$ with triangular boundary. Let $\alpha,\beta$
be two disjoint diagrams on $\surface$ such that no boundary edge contains endpoint of
both $\alpha$ and $\beta$. Then by definition, $\alpha\cup\beta$ is commuting, but
$\alpha\cdot\beta$ is $q$-commuting. Comparing Definition~\eqref{eq-new-prod} with
\eqref{eq-qtorus-prod}, we see that the $\cup$-product is the Weyl-ordering of
$\cdot$-product in this case. This observation is used in the next section.

\subsection{The skein of the lantern}
\label{sec-lantern}

Recall the standard lantern $\lantern$ defined in Section~\ref{sec-dualS}. On each
boundary component of $\lantern$, delete a point between each pair of endpoints of
standard arcs. This gives a surface with triangular boundary. See
Figure~\ref{fig-puncL}.
We use $\lantern$ to denote the compact surface except in skein modules where triangular
boundaries are required.

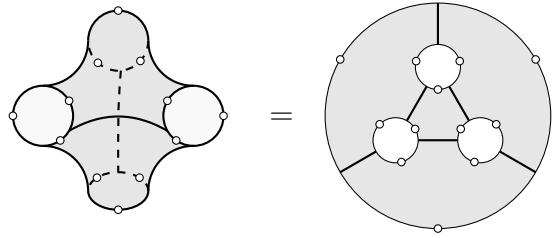
\begin{figure}[htpb!]
\centering
\begin{tikzpicture}[baseline=(ref.base)]
\fill[gray!20] (0.4,1) arc[radius=0.4,start angle=0,end angle=180]
arc[radius=0.6,start angle=0,end angle=-90]
arc[radius=0.4,start angle=90,end angle=-90]
arc[radius=0.6,start angle=90,end angle=0]
arc[x radius=0.4,y radius=0.25,start angle=-180,end angle=0]
arc[radius=0.6,start angle=180,end angle=90]
arc[radius=0.4,start angle=-90,delta angle=-180]
arc[radius=0.6,start angle=-90,end angle=-180];
\begin{scope}[thick]
\draw[fill=gray!5] (-1,0)circle(0.4) (1,0)circle(0.4);
\draw (0.4,1)arc[radius=0.4,start angle=0,end angle=180] (0.4,-1)
arc[x radius=0.4,y radius=0.25,start angle=0,end angle=-180];
\draw[dashed] (0.4,1)arc[radius=0.4,start angle=0,end angle=-180] (0.4,-1)
arc[x radius=0.4,y radius=0.25,start angle=0,end angle=180];
\draw (-1,0.4)arc[radius=0.6,start angle=-90,end angle=0]
(1,0.4)arc[radius=0.6,start angle=-90,end angle=-180]
(-1,-0.4)arc[radius=0.6,start angle=90,end angle=0]
(1,-0.4)arc[radius=0.6,start angle=90,end angle=180];
\path (-1,0)++(-30:0.4)coordinate(E) (1,0)++(-150:0.4)coordinate(F);
\draw (E) to[out=30,in=150] (F);
\draw[dashed] (0,1) ++(-85:0.4) to[out=-95,in=90] (0,-0.75);
\end{scope}
\path (0.4,-1)
arc[x radius=0.4,y radius=0.25,start angle=0,end angle=45] coordinate (P) (0.4,-1)
arc[x radius=0.4,y radius=0.25,start angle=0,end angle=135] coordinate (Q);
\draw[fill=white] (-1,0) foreach \t in {30,-55,180} {+(\t:0.4)circle(0.05)}
(1,0) foreach \t in {150,-125,0} {+(\t:0.4)circle(0.05)}
(0,1) foreach \t in {-42.5,-132.5,90} {+(\t:0.4)circle(0.05)}
(0,-1.25)circle(0.05) (P)circle(0.05) (Q)circle(0.05);
\node (ref) at (0,0){\phantom{$-$}};
\end{tikzpicture}
\quad$=$\quad
\begin{tikzpicture}[baseline=(ref.base)]
\tikzmath{\r1=0.65;\r2=0.3;}
\draw[fill=gray!20] (0,0)circle(1.5);
\draw[thick] (90:\r1) -- (-30:\r1) -- (-150:\r1) -- cycle;
\foreach \t in {90,-30,-150} {
\draw[fill=white] (-\t:1.5)circle(0.05) (\t:\r1)circle(\r2);
\foreach \s in {75,-75,180}
\draw[fill=white] (\t:\r1) +({\t+\s}:\r2)circle(0.05);
\draw[thick] (\t:{\r1+\r2}) -- (\t:1.5);
}
\node (ref) at (0,0){\phantom{$-$}};
\end{tikzpicture}
\caption{The standard lantern with triangular boundary.}\label{fig-puncL}
\end{figure}

Let $\phi^i:\Oq\to\skein(\lantern)$, $i=1,\dotsc,6$, denote the embedding along the
standard arcs $a^i$. They fit together to define an $R$-linear map
\begin{equation}
\label{phidef}
\phi:\Oq^{\otimes6}\to\skein(\lantern), \qquad
\phi(x^1\otimes\dotsb\otimes x^6)=\phi^1(x^1)\cup\dotsm\cup\phi^6(x^6) \,.
\end{equation}

\begin{lemma}
$\phi$ is an algebra embedding if we use the $\cup$-product $\skein(\lantern)$.
Hence, its $\skein^0$ is a subalgebra (for both products).
\end{lemma}

\begin{proof}
$\phi$ is injective using the basis from Theorem~\ref{thm-basis}. Each $\phi_i$ is an
algebra embedding, and their images commute with each other under $\cup$ since the
standard arcs all end on different boundary edges. Thus, $\phi$ is an algebra map if
we use the $\cup$-product on $\skein(\lantern)$. Then clearly, the image is a
subalgebra under $\cup$, but the two products differ by a scalar, so it is also a
subalgebra under $\cdot$.
\end{proof}

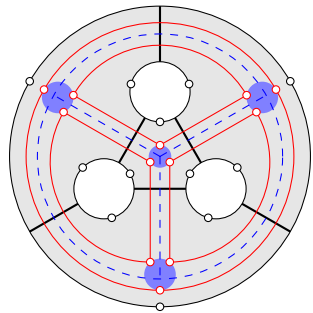
\begin{figure}[htpb!]
\centering
\begin{tikzpicture}[every node/.style={inner sep=2pt},baseline=(ref.base)]
\tikzmath{\r1=0.86;\r2=0.4;\r3=2-(2-\r1-\r2)/2;\r9=0.15;\r4=\r3-\r9*2;
\h=\r3-\r4+\r9/2;\d=sqrt(3)/2*\r9;\r5=(\h*\h+\d*\d)/(2*\h);}
\draw[fill=gray!20] (0,0)circle(2);
\draw[thick] (90:\r1) -- (-30:\r1) -- (-150:\r1) -- cycle;
\foreach \t in {90,-30,-150} {
\draw[fill=white] (-\t:2)circle(0.05) (\t:\r1)circle(\r2);
\foreach \s in {75,-75,180}
\draw[fill=white] (\t:\r1) +({\t+\s}:\r2)circle(0.05);
\draw[thick] (\t:{\r1+\r2}) -- (\t:2);
}
\fill[blue!50] (0,0)circle(\r9) foreach \t in {30,150,-90}
{(\t:{\r3+\r9-\r5})circle(\r5)};
\draw[blue,dashed] (0,0)circle(\r3) foreach \t in {30,150,-90} {(0,0) -- (\t:\r3)};
\draw[red] (0,0)circle({\r3+\r9}) foreach \t in {30,150,-90}
{({\t+60}:\r9) -- ++(\t:\r4) arc[radius=\r4,start angle=\t,delta angle=120]
  -- cycle};
\draw[red,fill=white] foreach \t in {30,150,-90} {(\t:{\r3+\r9})circle(0.05)
({\t+60}:\r9)circle(0.05) +(\t:\r4)circle(0.05) +({\t+120}:\r4)circle(0.05)};
\node (ref) at (0,0){\phantom{$-$}};
\end{tikzpicture}
\caption{Splitting the standard lantern.}\label{fig-lgraph}
\end{figure}

\begin{lemma}
\label{lemma-cr-span}
The restriction of the quotient $\skein(\lantern)\onto\skeincr(\lantern)$ to $\skein^0$
is surjective.
\end{lemma}

\begin{proof}
We start by drawing a graph dual to the standard arcs. See the dashed blue graph in
Figure~\ref{fig-lgraph}. Given any diagram on $\lantern$, we can assume that it is
disjoint from the vertices of the blue graph and transverse to the edges. Then we draw
4 curves isotopic to the boundary triangles and very close to the graph, and we put
punctures on the curves close to the vertices of the graph. See the red curves in the
same figure. Now we use the comodule property to split along the red curves and apply
the counit to the annuli. The result is the same element in $\skeincr(\lantern)$ but
given as a sum of new diagrams. If the red curves are close enough to the graph, then
each term is a diagram consisting of standard arcs. Such a diagram is in the image of
$\phi$.
\end{proof}

Let $\qglue(\lantern)$ denote the quotient of $\skeincr(\lantern)$ by the left ideal
generated by standard arcs with opposite states assigned to the endpoints. In other
words, the left ideal is generated by
\begin{equation}
\label{eq-qglue-def}
a^i_{+-}=a^i_{-+}=0,\qquad i=1,\dotsc,6
\end{equation}
using the notation of Section~\ref{sec-Oq}. The motivation for this relation is the
vanishing of the off-diagonal entries of the $\gamma$ matrix given in Equation
~\eqref{Mgamma}. 

Since $\skeincr(\lantern)$ is also a quotient of $\skein^0\subset\skein(\lantern)$, we
can obtain $\qglue(\lantern)$ using a different order of quotients. Let $\qtorus^0$ be
the quotient of $\skein^0$ by \eqref{eq-qglue-def}. We should consider the left ideal
in the quotient, but using the definition \eqref{eq-Oq-def} of $\Oq$, it is easy to
see that the left ideal is also two-sided. Also from the definition, since we set the
off-diagonal elements of $\Oq$ to $0$, we have
\begin{equation}
\label{eq-lantern-torus}
\qtorus^0=R[(a^1_{--})^{\pm1},\dotsc,(a^6_{--})^{\pm1}]
\end{equation}
when we use $\cup$ product, which implies that $\qtorus^0$ is a quantum torus
under $\cdot$ by the discussion in Section~\ref{sub.qtorus}. Here, inverses are given by
$(a^i_{--})^{-1}=a^i_{++}$. This shows $\qglue(\lantern)$ is a quotient of a quantum
torus.

In the next section, we show that the corner reductions for $\qtorus^0$ are equivalent
to the Lagrangian equation, which justify the notation $\qglue(\lantern)$. However,
this is not obvious a priori, so we choose to define $\qglue(\lantern)$ as a quotient of
$\skeincr(\lantern)$, which is better for the splitting homomorphism.

\begin{lemma}\label{lemma-qglue-2side}
The quotient \eqref{eq-qglue-def} is 2-sided in the sense that
\begin{equation}
a^i_{+-}\cdot\qglue(\lantern)=a^i_{-+}\cdot\qglue(\lantern)=0.
\end{equation}
\end{lemma}

\begin{proof}
This is true for $\qtorus^0$, so it holds for the quotient $\qglue(\lantern)$.
\end{proof}

\subsection{Presentation of $\qglue(\lantern)$}

We described $\qglue(\lantern)$ as some quotient of $\qtorus^0$. In this
section we give a presentation for this quotient. To make connection to the quantum
gluing module, we rename the standard arcs.

The standard arcs cut $\lantern$ into 4 components, dual to the ideal vertices. Choose
one of these components and one standard arc on it to label as $b$. Label the other
two arcs $b',b''$ as shown in Figure~\ref{fig-stdL-arcs}, and label the arc opposite to
$b^\square$ by $c^\square$.

Let 
\begin{equation}\label{eq-qz-def}
\qz^\square=b^\square_{--},\,\qy=c^\square_{--}\in\skein(\lantern).
\end{equation}
By the definition of the product \eqref{eq-new-prod},
\begin{equation}
\label{eq-zzp}
\qz''\qz=q\qz\qz'',\qquad \qy\qz=\qz\qy.
\end{equation}
Other $q$-commuting relations, e.g. $\qz\qy'=q\qy'\qz$, can be obtained by symmetry.
These are also used as the $q$-commuting relations of $\qtorus^0$.

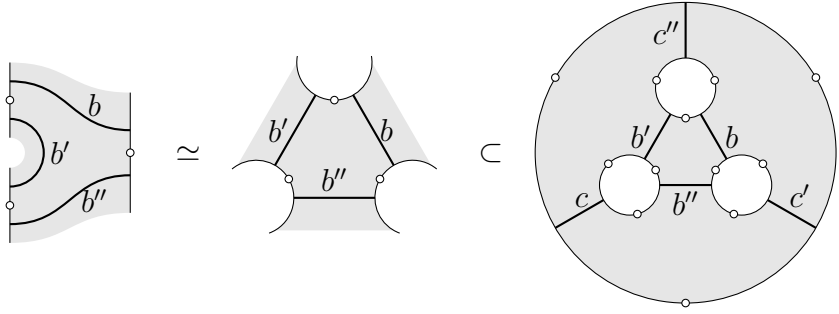
\begin{figure}[htpb!]
\centering
\begin{tikzpicture}[baseline=(ref.base)]
\tikzmath{\w=1.6;\h=\w/2;}
\fill[gray!20] (0,-0.2) arc[radius=0.2,start angle=-90,end angle=90]
-- (0,1.2)..controls +(\h,0) and +(-\h,0)..(\w,0.8)
-- (\w,-0.8)..controls +(-\h,0) and +(\h,0)..(0,-1.2) -- cycle;
\draw (0,0.2) -- (0,1.2) (0,-0.2) -- (0,-1.2) (\w,-0.8) -- (\w,0.8);
\draw[fill=white] (0,0.7)circle(0.05) (0,-0.7)circle(0.05) (\w,0)circle(0.05);
\draw[thick,inner sep=2pt] (0,-0.45) arc[radius=0.45,start angle=-90,end angle=90]
node[midway,right]{$b'$}
(0,0.95)..controls +(\h,0) and +(-\h,0)..node[near end,above]{$b$} (\w,0.3)
(0,-0.95)..controls +(\h,0) and +(-\h,0)..node[near end,below]{$b''$} (\w,-0.3);
\node (ref) at (\h,0){\phantom{$-$}};
\end{tikzpicture}
\quad$\simeq$\quad
\begin{tikzpicture}[baseline=(ref.base)]
\tikzmath{\r1=1.2;\r2=0.5;\s=\r1*sqrt(3)-\r2;}
\fill[gray!20] (90:\r1) ++(0:\r2) -- ++(180:{2*\r2}) -- ++(-120:\s)
-- ++(-60:{2*\r2}) -- ++(0:\s) -- ++(60:{2*\r2}) -- cycle;
\draw[thick,inner sep=2pt]
(90:\r1) --node[left]{$b'$} (-150:\r1) --node[above]{$b''$} (-30:\r1)
-- node[right]{$b$} cycle;
\foreach \t in {90,-30,-150} {
  \draw[fill=white] (\t:\r1)+({\t-80}:\r2)arc[radius=\r2,start
  angle={\t-80},delta angle=-200];
\draw[fill=white] (\t:{\r1-\r2})circle(0.05);
}
\node (ref) at (0,0){\phantom{$-$}};
\end{tikzpicture}
\quad$\subset$\quad
\begin{tikzpicture}[every node/.style={inner sep=2pt},baseline=(ref.base)]
\tikzmath{\r1=0.86;\r2=0.4;}
\draw[fill=gray!20] (0,0)circle(2);
\draw[thick] (90:\r1) --node[right,inner sep=4pt]{$b$} (-30:\r1)
--node[below]{$b''$} (-150:\r1) --node[left]{$b'$} cycle;
\foreach \t in {90,-30,-150} {
\draw[fill=white] (-\t:2)circle(0.05) (\t:\r1)circle(\r2);
\foreach \s in {75,-75,180}
\draw[fill=white] (\t:\r1) +({\t+\s}:\r2)circle(0.05);
\draw[thick] (\t:{\r1+\r2}) -- (\t:2);
}
\path (90:1.6)node[left]{$c''$} (-150:1.6)node[above]{$c$} (-30:1.6)
node[anchor=240]{$c'$};
\node (ref) at (0,0){\phantom{$-$}};
\end{tikzpicture}
\caption{Labeling arcs on the standard lantern.}\label{fig-stdL-arcs}
\end{figure}

\begin{lemma}
\label{lemma-qz-rel}
In $\qglue(\lantern)$, we have
\begin{equation}\label{eq-qz-rel}
\qz\qz'\qz''=\iunit q^{3/2},\qquad
\qz^{-2}+(\qz'')^2=1 \,.
\end{equation}
\end{lemma}

\newenvironment{stdLhex}{
\begin{linkdiag}[0]
\tikzmath{\w=1.6;\h=\w/2;}
\fill[gray!20] (0,-0.2) arc[radius=0.2,start angle=-90,end angle=90]
-- (0,1.2)..controls +(\h,0) and +(-\h,0)..(\w,0.8)
-- (\w,-0.8)..controls +(-\h,0) and +(\h,0)..(0,-1.2) -- cycle;
\draw (0,0.2) -- (0,1.2) (0,-0.2) -- (0,-1.2) (\w,-0.8) -- (\w,0.8);
\draw[fill=white] (0,0.7)circle(0.07) (0,-0.7)circle(0.07) (\w,0)circle(0.07);
}{\end{linkdiag}}

\begin{proof}
To prove the first identity, start with $\qz'=b'_{--}$ and apply Corollary~\ref{cor-tw}.
\begin{equation}
\qz'=
\begin{stdLhex}
\draw[thick] (0,-0.45)\stnodel{-}..controls +(1.6,0) and +(1.6,0)..(0,0.45) \stnodel{-};
\end{stdLhex}
=(-\iunit q^{-1})\cdot q\left(\iunit
\begin{stdLhex}
\draw[thick] (0,-0.95)\stnodel{+}..controls +(1.6,0.4) and +(1.6,-0.4)..(0,0.95)
\stnodel{+};
\end{stdLhex}
+
\begin{stdLhex}
\draw[thick] (0,-0.95)\stnodel{-}..controls +(1.6,0.4) and +(1.6,-0.4)..(0,0.95)
\stnodel{+};
\end{stdLhex}
\right)
\end{equation}
After applying \eqref{eq-cup-puncture}, the diagrams are of the form $b_{\mu\nu}\cup
b''_{\mu''\nu''}$.
This is zero if $\mu\ne\nu$ or $\mu''\ne\nu''$. Therefore, the only nonzero term after
\eqref{eq-cup-puncture} is
\begin{equation}
\qz'=\iunit q
\begin{stdLhex}
\draw[thick] (0,-0.95)\stnodel{+}..controls +(\h,0) and +(-\h,0)..(\w,-0.3)\stnode{+}
(0,0.95)\stnodel{+}..controls +(\h,0) and +(-\h,0)..(\w,0.3)\stnode{+};
\end{stdLhex}
=\iunit q\qz^{-1}\cup(\qz'')^{-1}.
\end{equation}
This is equivalent to $\qz\qz'\qz''=\iunit q^{3/2}$.

The second identity is similar, starting with $b'_{+-}=0$ where the $+$ state is the
top endpoint.
\begin{align*}
0&=
\begin{stdLhex}
  \draw[thick] (0,-0.45)\stnodel{-}..controls +(1.6,0) and +(1.6,0)..(0,0.45)
  \stnodel{+};
\end{stdLhex}
=\iunit
\begin{stdLhex}
  \draw[thick] (0,-0.95)\stnodel{+}..controls +(1.6,0.4) and +(1.6,-0.4)..(0,0.95)
  \stnodel{+};
\end{stdLhex}
+
\begin{stdLhex}
  \draw[thick] (0,-0.95)\stnodel{-}..controls +(1.6,0.4) and +(1.6,-0.4)..(0,0.95)
  \stnodel{+};
\end{stdLhex}
+
\begin{stdLhex}
  \draw[thick] (0,-0.95)\stnodel{+}..controls +(1.6,0.4) and +(1.6,-0.4)..(0,0.95)
  \stnodel{-};
\end{stdLhex}
-\iunit
\begin{stdLhex}
  \draw[thick] (0,-0.95)\stnodel{-}..controls +(1.6,0.4) and +(1.6,-0.4)..(0,0.95)
  \stnodel{-};
\end{stdLhex}\\
\intertext{The first two terms are calculated before, and the last two terms are
similar. We get}
&=\iunit(\iunit q)
\begin{stdLhex}
  \draw[thick] (0,-0.95)\stnodel{+}..controls +(\h,0) and +(-\h,0)..(\w,-0.3)
  \stnode{+} (0,0.95)
\stnodel{+}..controls +(\h,0) and +(-\h,0)..(\w,0.3)\stnode{+};
\end{stdLhex}
+0+
\begin{stdLhex}
  \draw[thick] (0,-0.95)\stnodel{+}..controls +(\h,0) and +(-\h,0)..(\w,-0.3)
  \stnode{+} (0,0.95)
\stnodel{-}..controls +(\h,0) and +(-\h,0)..(\w,0.3)\stnode{-};
\end{stdLhex}
-\iunit(-\iunit q^{-1})
\begin{stdLhex}
  \draw[thick] (0,-0.95)\stnodel{-}..controls +(\h,0) and +(-\h,0)..(\w,-0.3)
  \stnode{-} (0,0.95) \stnodel{-}..controls +(\h,0) and +(-\h,0)..(\w,0.3)\stnode{-};
\end{stdLhex}\\
&=-q\qz^{-1}\cup(\qz'')^{-1}+\qz\cup(\qz'')^{-1}-q^{-1}\qz\cup\qz''.
\end{align*}
After multiplying on the left by $\qz^{-1}\cup\qz''$, we get $\qz^{-2}+(\qz'')^2=1$.
\end{proof}

\begin{remark}
The same calculations can be done with $b'_{++}$ and $b'_{-+}$, but they do not give
additional relations.
\end{remark}

The next theorem gives a promised presentation for the $\qglue(\lantern)$ module.

\begin{theorem}
\label{thm-lantern}
We have
\begin{equation}\label{GG}
\qglue(\lantern)=
\qtorus\langle\qz,\qz'',\qy\rangle
\quotbyL{\qz^{-2}+(\qz'')^2-1, \qz^2-\qy^2} \,.
\end{equation}
\end{theorem}

\begin{proof}
Let $\qglue^0$ denote the quotient of $\qtorus^0$ by the left ideal generated by
\eqref{eq-qz-rel} as well as the relations obtained by symmetries of $\lantern$. This
includes the vertex equations
\begin{equation}
\qz\qz'\qz''=\qz\qy'\qy''=\qy\qz'\qy''=\qy\qy'\qz''=\iunit q^{3/2}.
\end{equation}
Similarly, there are 12 Lagrangian equations. We can eliminate $\qz',\qy',\qy''$ using
the vertex equations. Then we get $\qz^2=\qy^2$, and only one Lagrangian
$\qz^{-2}+(\qz'')^2=1$ is required. Thus, 
\begin{equation}
\qglue^0=\qtorus\langle\qz,\qz'',\qy\rangle
\quotbyL{\qz^{-2}+(\qz'')^2-1, \qz^2-\qy^2}.
\end{equation}
The discussion prior to the theorem shows that there is a surjective map
$f:\qglue^0\to\qglue(\lantern)$. Now we find the inverse.

Recall the setup of Lemma~\ref{lemma-cr-span}. For each vertex of the blue graph, draw
a small disk around it such that the punctures on the red curves are on the boundary
of the disk. See the light blue disks in Figure~\ref{fig-lgraph}. Let $\lantern'$ be
$\lantern$ minus the closure of the disks. Then the red curves break into 12 ideal arcs
on $\lantern'$.

If we split $\lantern'$ along the 12 red ideal arcs, then the surface becomes 4
standard annuli and 6 bigons. Thus, the splitting homomorphism has the form
$\skein(\lantern')\to(\skein(\annulus)^{\otimes4})\otimes(\skein(\bigon)^{\otimes6})$.
If we consider the corner reduction $\skeincr(\lantern')$, the defining relations
\eqref{eq-bad} and \eqref{eq-crdef} are contained in the annuli, so if we also reduce
the annuli, we get
\begin{equation}
\cutred:\skeincr(\lantern')\to(\skeincr(\annulus)^{\otimes4})
\otimes(\skein(\bigon)^{\otimes6}).
\end{equation}
Let $k:\skein(\bigon)^{\otimes6}\xrightarrow[\cong]{\phi}\skein^0\onto\qtorus^0$
be the quotient map. Now consider the composition
\begin{equation}
\label{eq-lantern-inv}
\tilde{g}:\skeincr(\lantern')
\xrightarrow{\cutred}
(\skeincr(\annulus)^{\otimes4})\otimes(\skein(\bigon)^{\otimes6})
\xrightarrow{(\epsilon^{\otimes4})\otimes k}
\qtorus^0\onto\qglue^0.
\end{equation}
The inclusion $\lantern'\embed\lantern$ induces a surjective map
$\skeincr(\lantern')\onto\skeincr(\lantern)\onto\qglue(\lantern)$. We claim that
$\tilde{g}$ induces a map $g:\qglue(\lantern)\to\qglue^1$. Then it is easy to check
that $g$ is inverse to $f$, which proves the theorem.

To finish the proof, we show that $g$ is well-defined. For this, we need to consider
the kernel $I=\ker(\skeincr(\lantern')\onto\qglue(\lantern))$, which is a left ideal. 
By definition, the kernel of $\skeincr(\lantern)\onto\qglue(\lantern)$ is generated
by \eqref{eq-qglue-def}, which can be lifted to $\skeincr(\lantern')$. On the other
hand, the kernel of $\skeincr(\lantern')\onto\skeincr(\lantern)$ is generated by the
handle slides.
\begin{equation}
\label{eq-lantern-handle}
\begin{linkdiag}
\fill[gray!20] (0,0) rectangle (1,1);
\fill[blue!50] (0.5,0.5)circle(0.15);
\draw[blue,dashed] (0.5,0.5) -- (0.5,0) (0.5,0.5) -- (0,1) (0.5,0.5) -- (1,1);
\draw[thick,rounded corners] (0.2,0) -- (0.2,0.5) -- (0.5,1);
\end{linkdiag}
=
\begin{linkdiag}
\fill[gray!20] (0,0) rectangle (1,1);
\fill[blue!50] (0.5,0.5)circle(0.15);
\draw[blue,dashed] (0.5,0.5) -- (0.5,0) (0.5,0.5) -- (0,1) (0.5,0.5) -- (1,1);
\draw[thick,rounded corners] (0.2,0) -- (0.8,0.4) -- (0.5,1);
\end{linkdiag}.
\end{equation}
Together, these relations generate $I$ as a left ideal in $\skeincr(\lantern')$. We
just need to show $\tilde{g}(I)=0$.

\begin{figure}[htpb!]
\centering
\begin{tikzpicture}[baseline=(ref.base)]
\fill[gray!20] (0,0)circle(1);
\fill[blue!50] (0,0)circle(0.2);
\draw[blue,dashed] foreach \t in {-90,30,150} {(0,0) -- (\t:1)};
\clip (0,0)circle(1);
\begin{scope}[red]
\draw foreach \t in {90,-30,-150} {(\t:0.2) -- +({\t+60}:1) (\t:0.2) -- +({\t-60}:1)};
\draw[fill=white] foreach \t in {90,-30,-150} {(\t:0.2)circle(0.05)};
\draw[dashed]
foreach \t in {90,-30,-150} {(\t:0.6) -- +({\t+60}:1) (\t:0.6) -- +({\t-60}:1)};
\draw[fill=white] foreach \t in {90,-30,-150} {(\t:0.6)circle(0.05)};
\end{scope}
\draw[thick,rounded corners,dashed] (-165:1) -- (150:0.4) -- (105:1);
\draw[thick,rounded corners]
(-160:1) -- (-165:0.5) -- (-90:0.4) -- (-30:0.4) -- (30:0.4) -- (105:0.5) -- (100:1);
\node (ref) at (0,0){\phantom{$-$}};
\end{tikzpicture}
\caption{Extra splitting curves.}\label{fig-extra-split}
\end{figure}
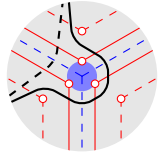

$\tilde{g}$ maps \eqref{eq-qglue-def} to zero via either $\epsilon$ or $k$. For the
handle slide, we first reduce to the case of a single standard arc.
In each annulus component, draw another splitting curve close enough to the red curve
but also away from the handle slide region. See the dashed red curve in Figure
~\ref{fig-extra-split}, where both sides of the handle slide are shown. Here, close
enough means the diagram in the region between them consists of small segments of
standard arcs and the segments of the handle slide. Now split along the dashed curves
and apply the counit. Since the handle slide is sandwiched between the red curves, the
counit for both sides of the handle slide are the same. We can also choose the height
order so that the arc changed by the handle slide is at the bottom. This means if
$\tilde{g}$ is invariant for the handle slide of a single standard arc, then by
multiplying the rest of the diagram on the left, we prove the invariance for general
diagrams. Up to symmetry, there is only one handle slide of standard arcs (with 4
combinations of states). By reinterpreting the calculations in Lemma
~\ref{lemma-qz-rel}, we obtain the invariance under $\tilde{g}$. Therefore,
$\tilde{g}$ sends both types of generators of $I$ to $0$.
\end{proof}


\section{A quantum trace map of triangulated 3-manifolds}
\label{sec.qtrace}

In this section we give the promised definition of the quantum trace map
~\eqref{defqtr} and prove its properties stated in Theorems~\ref{thm.1}
and~\ref{thm.frob}.

\subsection{Quantum gluing module}
\label{sec-qglue}

Let $M$ be an oriented 3-manifold with an ideal triangulation $\triang$. Let
$\heeg_\triang:=(\surface_\triang,\{A_f\},\{B_e\})$ be the dual surface from
Section~\ref{sec-dualS}. By construction, $\surface_\triang$ split along all $A$-curves
consists of a lantern $\lantern_j=\surface_\triang\cap T_j$ for each tetrahedron $T_j$.

The quantum gluing module $\qglue(\triang)$ was explained in the introduction. By
rearranging the order of quotients and tensor products, there is a presentation more
compatible with splitting. Note that the Lagrangian equations only involve variables
from the same tetrahedron. Thus, the quotient by Lagrangian can be taken before the
tensor product, so
\begin{equation}
\qglue(\triang)=\rideal{\text{edge}}\backslash
\Big(\bigotimes_{j=1}^N\qtorus
\langle\qz_j,\qz''_j\rangle\quotbyL{\text{Lagrangian}}\Big) \,.
\end{equation}
Using Theorem~\ref{thm-lantern}, for each tetrahedron $T_j$, we get an isomorphism
\begin{equation}\label{eq-qglueL-elim-y}
\qglue(\lantern)\quotbyL{\qz-\qy} \cong
\qtorus\langle\qz_j,\qz''_j\rangle\quotbyL{\text{Lagrangian}}
\end{equation}
sending $\qz\mapsto\qz_j$ and $\qz''\mapsto\qz''_j$.
Tensoring these together, we get another definition
\begin{equation}\label{eq-qglue-triang}
\qglue(\triang)=
\rideal{\text{edge}}\backslash
\Big(\bigotimes_{j=1}^N\qglue(\lantern_j)\quotbyL{\qz_j-\qy_j}\Big)
\end{equation}
where the generators of each copy $\qglue(\lantern_j)$ is now written with the index $j$.

\begin{remark}
\label{rem-qglue-ext}
We can extend the definition of $\qglue(\triang)$ to include the $\qy$ variables by
removing the quotient of $\lideal{\qz_j-\qy_j}$. The results below all work with this
extension.
\end{remark}

Like the skein module, if the choices of $R$ and $q$ need to be shown, we will use the
notation $\qglue_q(\triang;R)$. As an example, the coordinate ring $\cx[G_\triang]$ of
the gluing variety is $\qglue_1(\triang;\cx)/\sqrt{0}$.

We have a universal coefficient property similar to \eqref{eq-skein-ucoef}.

\begin{lemma}
\label{lemma-qglue-classical}
Suppose $r:R\to R'$ is a ring homomorphism and $r(q)=q'$. Then
\begin{equation}\label{eq-qglue-classical}
\qglue_q(\triang;R)\otimes_R R'\cong\qglue_{q'}(\triang;R').
\end{equation}
\end{lemma}

\begin{proof}
The proof is the same as \cite[Proposition~2.2(4)]{Przytycki:fund}. The key here is
that $\qglue$ is a quotient of a quantum torus, which is a free module. In particular,
$\qtorus_q(\triang;R)\otimes_R R'=\qtorus_{q'}(\triang;R')$.
Let $I_R=\ker(\qtorus_q(\triang;R)\onto\qglue_q(\triang;R))$, and define $I_{R'}$
similarly. Then we have exact sequences
\begin{equation}
\begin{tikzcd}
I_R\otimes_R R' \arrow[r] \arrow[d] &
\qtorus_q(\triang;R)\otimes_R R' \arrow[r] \arrow[d,equal] &
\qglue_q(\triang;R)\otimes_R R' \arrow[r] \arrow[d] & 0\\
I_{R'} \arrow[r] & \qtorus_{q'}(\triang;R') \arrow[r] &
\qglue_{q'}(\triang;R') \arrow[r] & 0
\end{tikzcd}
\end{equation}
Here, the top row is exact by the right exactness of the tensor product. A formal
argument shows that the left vertical map is surjective. Then the five lemma proves
\eqref{eq-qglue-classical}.
\end{proof}

\subsection{Quantum trace map}

Starting from the dual surface $\heeg_\triang:=(\surface_\triang,\{A_f\},\{B_e\})$,
instead of a cell structure, we can recover $M$ from $\heeg_\triang$ using a Heegaard-like
process. By attaching 2-handles to the closed thickening $\surface_\triang\times[-1,1]$
along all $A_f\times\{-1\}$ and capping off spherical boundary components, we obtain the
handlebody inside the surface $\surface_\triang$. Then $M$ is obtained by attaching
2-handles to all $B_e\times\{1\}$.

By Proposition~\ref{prop-handle}, there is a surjective map
$\skein(\surface_\triang)\onto\skein(M)$,
and the kernel is generated by handle slides. Since the $A$-handles are attached to the
bottom, the submodule $\lideal{A}$ of $A$-handle slides is a left ideal in
$\skein(\surface_\triang)$.
Similarly, the $B$-handle slides generate a right ideal $\rideal{B}$. Thus, we get the
isomorphism claimed in Proposition~\ref{prop.ST}
\begin{equation}
\rideal{B}\backslash\skein(\surface_\triang)/\lideal{A}
\xrightarrow{\cong}\skein(M).
\end{equation}

Define the quantum trace map $\qtr_\triang:\skein(M)\to\qglue(\triang)$ using the
following diagram.
\begin{equation}
\label{eq-qtr-def}
\begin{tikzcd}
\skein(\surface_\triang) \arrow[d,two heads] \arrow[r,"\cut_A"]
\arrow[drr,"\ttr_\triang"']&
\bigotimes_{j=1}^N\skeincr(\lantern_j) \arrow[r,two heads] &
\bigotimes_{j=1}^N\qglue(\lantern_j) \arrow[d,two heads] \\
\skein(M) \arrow[rr,dashed,"\qtr^q_\triang"']
&& \qglue(\triang)
\end{tikzcd}
\end{equation}
Here $\cut_A$ is the splitting map along all $A$-circles.
As before, $q$ can be included like $\qtr^q_\triang$ if it is important.

\begin{theorem}
The map $\qtr_\triang:\skein(M)\to\qglue(\triang)$ is well-defined.
\end{theorem}

\begin{proof}
Define the composition $\ttr_\triang$ using the diagram. We need to show that
$\ttr_\triang(\lideal{A})=0$ and $\ttr_\triang(\rideal{B})=0$.

First consider the handle slide along $A_f$ for some face $f$. To calculate the cut
$\cut_A$, we need to isotope the handle slide region slightly to be disjoint from
the curve $A_f$. Then after splitting, the handle slide becomes an identity in
$\skeincr(\lantern)$ by Lemma~\ref{lemma-crhandle}. Thus, $\cut_A(\lideal{A})=0$,
so $\ttr_\triang(\lideal{A})=0$ as well.

\begin{figure}[htpb!]
\centering
\begin{tikzpicture}[baseline=(ref.base),inner sep=2pt]
\fill[gray!20,even odd rule] (0,0) circle(0.5) circle(1.6);
\draw[blue] (0,0) circle(0.8);
\draw[red] foreach \t in {-30,-150,90,30} {(\t:0.5) -- (\t:1.6)} (60:0.6)
node[rotate=-30]{...};
\draw[thick,rounded corners,dashed] (-115:1.6)..controls (-115:1) and
(-65:1)..(-65:1.6) node[midway,below]{$b$};
\draw[thick,rounded corners] (-120:1.6) -- (-120:1)
arc[radius=1,start angle=-120,delta angle=-300]
-- (-60:1.6) (135:1)node[above left]{$a$};
\node (ref) at (0,0) {\phantom{$-$}};
\end{tikzpicture}
\caption{$B$-handle slide: $a=b$ in $\skein(M)$.}\label{fig-bslide}
\end{figure}
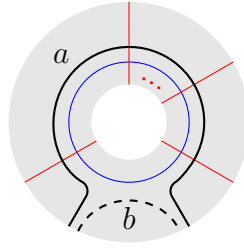

Now look at the handle slide along $B_e$ for some edge $e$. A neighborhood of $B_e$
is shown in Figure~\ref{fig-bslide}, where the blue circle is $B_e$, and the red
segments are parts of the $A$-circles. The black strands are part of link, and both
sides of the handle slide are shown. There may be additional strands of the link in
the region, but they are all below the strand in the figure. By definition, we split
along the red $A$-circles and identify the components with the standard lantern
$\lantern$ such that the blue $B$-arcs are the standard arcs. This means the punctures
of the split are away from the $B$-circles. We use the convention where we omit the
additional strands and consider the relations ``at the top''.

Start with the solid strand $a$ in Figure~\ref{fig-bslide}. Except for the
segments in the south, the arcs after splitting along the $A$-circles are standard
arcs, which we can make to be higher than all other strands. By Lemma
~\ref{lemma-qglue-2side}, if any of the standard arcs after splitting has opposite
states, then the diagram is zero in $\qglue$. Therefore, any nonzero term has all
$+$ or all $-$ states in the region in Figure~\ref{fig-bslide}. Define the following
tangles in the surface after splitting.
\begin{equation}
a_\mu=
\begin{tikzpicture}[baseline=(ref.base),inner sep=1pt]
\fill[gray!20,even odd rule] (0,0) circle(0.5) circle(1.6);
\draw[blue] (0,0) circle(0.8);
\draw[red] foreach \t in {-30,-150,90,30} {(\t:0.5) -- (\t:1.6)} (60:0.6)
node[rotate=-30]{...};
\draw[thick,rounded corners] (-120:1.6) -- (-120:1)
arc[radius=1,start angle=-120,delta angle=-30] node[above left]{\stsize$\mu$}
(-60:1.6) -- (-60:1) arc[radius=1,start angle=-60,delta angle=30]
node[above right]{\stsize$\mu$};
\node (ref) at (0,0) {\phantom{$-$}};
\end{tikzpicture}\,,\quad
c=
\begin{tikzpicture}[baseline=(ref.base),inner sep=1pt]
\fill[gray!20,even odd rule] (0,0) circle(0.5) circle(1.6);
\draw[blue] (0,0) circle(0.8);
\draw[red] foreach \t in {-30,-150,90,30} {(\t:0.5) -- (\t:1.6)} (60:0.6)
node[rotate=-30]{...};
\draw[thick] (-150:1)node[below]{\stsize$-$}
arc[radius=1,start angle=-150,delta angle=-240] node[below]{\stsize$-$} (90:1)
node[above left]{\stsize$-$} node[above right]{\stsize$-$} (30:1)
node[above]{\stsize$-$} node[right]{\stsize$-$};
\node (ref) at (0,0) {\phantom{$-$}};
\end{tikzpicture}\,,\quad
d=
\begin{tikzpicture}[baseline=(ref.base),inner sep=1pt]
\fill[gray!20,even odd rule] (0,0) circle(0.5) circle(1.6);
\draw[blue] (0,0) circle(0.8);
\draw[red] foreach \t in {-30,-150,90,30} {(\t:0.5) -- (\t:1.6)} (60:0.6)
node[rotate=-30]{...};
\draw[thick] (-150:1)node[above]{\stsize$-$}
arc[radius=1,start angle=-150,delta angle=120] node[above]{\stsize$-$};
\node (ref) at (0,0) {\phantom{$-$}};
\end{tikzpicture}\,.
\end{equation}
Here, $a_\mu$ is a sum over the states of all endpoints not shown in the figure.
Then by definition,
\begin{equation}
\label{eq-ttr-alpha}
\ttr_\triang(a)=c\cup a_-+c^{-1}\cup a_+.
\end{equation}
Here, the $\cup$-product is performed before reduction to $\qglue(\triang)$. On the
other hand, if we isotope $b$ and impose height order as below,
\begin{equation}
b=
\begin{tikzpicture}[baseline=(ref.base),inner sep=1pt]
\fill[gray!20,even odd rule] (0,0) circle(0.5) circle(1.6);
\draw[blue] (0,0) circle(0.8);
\draw[red] foreach \t in {-30,-150,90,30} {(\t:0.5) -- (\t:1.6)} (60:0.6)
node[rotate=-30]{...};
\path[red,tips,->] (-30:1.6) -- (-30:0.5);
\path[red,tips,->] (-150:1.6) -- (-150:0.5); 
\draw[thick] (-120:1.6) {[rounded corners] -- (-120:1.2)}
arc[radius=1.2,start angle=-120,delta angle=-45]
arc[radius=0.1,start angle=-165,delta angle=-180]
arc[radius=1,start angle=-165,delta angle=150]
arc[radius=0.1,start angle=165,delta angle=-180]
{[rounded corners] arc[radius=1.2,start angle=-15,delta angle=-45]} -- (-60:1.6);
\node (ref) at (0,0) {\phantom{$-$}};
\end{tikzpicture}\,,
\end{equation}
we get
\begin{equation}\label{eq-ttr-beta}
\ttr_\triang(b)
=\relarc[->]{-}{+}(d\cup a_+)\relarcl[->]{-}{+}
+\relarc[->]{+}{-}(d^{-1}\cup a_-)\relarcl[->]{+}{-}
=-q^{-2}d\cup a_+-q^2d^{-1}\cup a_-.
\end{equation}
We finish the proof by showing that the (quantized) edge equation imply
$\ttr_\triang(a)=\ttr_\triang(b)$.

To do so, we first identify the edge equation with $d\cup c=-q^2$. By
construction, different standard arcs do not end on the same boundary edge. Thus,
$d\cup c$ consists of arcs that commute under $\cup$, so it is the
Weyl-ordered product under $\cdot$, which is the monomial used in the edge
equation.

Now we start from \eqref{eq-ttr-alpha} and insert $1=-q^{-2}(d\cup c)$ and
its inverse,
\begin{equation}
\label{eq-ttr-alpha2}
\ttr_\triang(a)=-q^2(d^{-1}\cup c^{-1})\cdot(c\cup a_-)
-q^2(d\cup c)\cdot(c^{-1}\cup a_+).
\end{equation}
We would like to cancel $c$ with its inverse. This can be done by rewriting the
$\cdot$ using $\cup$. The correction factors in both terms are $1$. This is because
the terms in \eqref{eq-ttr-alpha} are results of splitting. This means the endpoints
and their states appear in pairs on matching boundary triangles. The same is true of
$d\cup c$ and its inverse. Because of the orientation, the matching boundary
triangles give opposite factors in \eqref{eq-new-prod}. Therefore, in
\eqref{eq-ttr-alpha2},
we can replace $\cdot$ with $\cup$ without additional factors. Then after cancelling
$c$ with its inverse, we get $\ttr_\triang(a)=\ttr_\triang(b)$. This shows
$\ttr_\triang(\rideal{B})=0$.
\end{proof}

\subsection{Classical limit}

Now we establish the classical limit, namely the commutativity of the
diagram \eqref{eqn.diagram}. This can be factored as follows.
\begin{equation}
\label{eq-classical-diag}
\begin{tikzcd}
\skein_q(M;\Runiv) \arrow[d,"\otimes\cx"] \arrow[r,"\qtr^q_\triang"] &
\qglue_q(\triang;\Runiv) \arrow[d,"\otimes\cx"] \\
\skein_1(M;\cx) \arrow[d,two heads,"/\sqrt{0}"] \arrow[r,"\qtr^1_\triang"] &
\qglue_1(\triang;\cx) \arrow[d,two heads,"/\sqrt{0}"]\\
\cx[\frameX] \arrow[r,"\tr_\triang"] & \cx[G_\triang]
\end{tikzcd}
\end{equation}
The vertical map $\skein_1(M;\cx)\to\cx[\frameX]$ comes from Lemma~\ref{lem-classical}.
The bottom map $\tr_\triang$ is the classical trace map induced by $G_\triang\to\frameX$
from Proposition~\ref{prop-classical-rep}.

\begin{theorem}
The diagram \eqref{eq-classical-diag} commutes.
\end{theorem}

\begin{proof}
The top square commutes essentially by definition, so we focus on the bottom square.

As observed in \cite{Le:decomp}, in the stated skein algebra with $q=1$, the element
defined by a tangle is unchanged by crossing changes as well as height reordering on the
boundary. This also means the stated skein algebra is commutative (for both $\cup$ and
$\cdot$ since the correction factor in \eqref{eq-new-prod} is $1$). This is also passed
on to the various quotients, and the diagram \eqref{eq-qtr-def} that defines
$\qtr^1_\triang$ is a commutative diagram of algebras and algebra maps.

This discussion shows that $\qtr^1_\triang$ is an algebra homomorphism, so we can check
the bottom square on algebra generators of $\skein_1(M;\cx)$. Since $\skein_1(M;\cx)$
is a quotient of $\skein_1(\surface_\triang;\cx)$, closed curves on $\surface_\triang$
with vertical framing generates $\skein_1(M;\cx)$. Such a curve is homotopic to a
curve $K$ on the smooth 1-skeleton $\dtrunc^{(1)}$ defined in
Section~\ref{sec-cell}. Thus, it is sufficient to consider curves $K$ of this form.

To calculate $\qtr^1_\triang(K)$, we need to calculate the splitting $\cut_A(K)$. By
definition, $K$ is transverse to all splitting curves $A_f$. Then \eqref{eq-le-split} can
be applied to $K$ since heights do not matter for $q=1$. Choose an arbitrary orientation
of $K$ and write $K=m^1m^2\dotsm m^k$ as a concatenation of edges of
$\dtrunc^{(1)}$. Then
\begin{equation}
\cut_A(K)=\sum_{s_1,\dotsc,s_k=\pm1}m^1_{s_1s_2}m^2_{s_2s_3}\dotsm m^k_{s_ks_1}
=\tr(\qvar{M}^1\qvar{M}^2\dotsm\qvar{M}^k).
\end{equation}
Here, we used the notation for arcs with states assigned from Section~\ref{sec-Oq}, and
\begin{equation}
\qvar{M}^i=\begin{pmatrix}m^i_{++}&m^i_{+-}\\m^i_{-+}&m^i_{--}\end{pmatrix}
\end{equation}
is a $2\times2$ matrix over $\skeincr_1(\lantern;\cx)$, which is mapped to
$\qglue_1(\lantern;\cx)$. If the arc $m^i$ is an $\alpha$ edge, then
\begin{equation}
\label{Mist}  
m^i_{st}=\relarc[]{s}{t},\qquad \qvar{M}^i=\begin{pmatrix}0&1\\-1&0\end{pmatrix} \,.
\end{equation}
If the arc $m^i$ is a $\beta$ edge, then
\begin{equation}
\label{Mbeta}  
m^i_{st}=\relcorner{t}{s},\qquad \qvar{M}^i=
\begin{pmatrix}\iunit&0\\1&-\iunit \end{pmatrix} \,.
\end{equation}
Finally, if the arc $m^i$ is a $\gamma$ edge, then it is a standard arc of the lantern.
By definitions \eqref{eq-qglue-def} and \eqref{eq-qz-def}, after reduced to
$\qglue_1(\lantern;\cx)$,
\begin{equation}
\label{Mgamma}  
\qvar{M}^i=\begin{pmatrix}\qz^{-1}&0\\0&\qz\end{pmatrix}.
\end{equation}
Let $M^i$ be $\qvar{M}^i\bmod\text{nilradical}$, which removes the hat on $\qz$.
Then $M^i$ are exactly the matrices in Proposition~\ref{prop-classical-rep}, so
\begin{equation}
\tr_\triang(t_K)=\tr(M^1M^2\dotsm M^k)=\qtr^1_\triang(K)\bmod\text{nilradical}.
\end{equation}
Thus, the bottom square commutes.
\end{proof}

\subsection{Chebyshev-Frobenius map}
\label{sub.2frob}

In this section we discuss the quantum trace map at roots of unity. To do so,
we need to recall the Chebyshev-Frobenius map on the skein module and on the quantum
torus. $R=\cx$ throughout this section and are omitted from the notations.

Given a framed knot $K\subset M$, let the \term{framed power} $K^{(n)}\subset M$ be the
link consisting of $n$ parallel copies of $K$, obtained by small translation in the
direction of the framing. For a polynomial $f(x)=\sum_{i=0}^n a_ix^i$, the
\term{threading}
of $K$ by $f$ is defined as
\begin{equation}
K^{(f)}=\sum_{i=0}^n a_i K^{(i)}\in\skein(M).
\end{equation}
More generally, the threading of a link $L=K_1\cup\dotsb\cup K_\ell$ by $f$ is defined
as a ``linear extension'' of applying $f(x)$ to every component of $L$
\begin{equation}
  L^{(f)}=\sum_{i_1,\dotsc,i_\ell=0}^n a_{i_1}\dotsm
  a_{i_\ell} K_1^{(i_1)}\cup\dotsb\cup K_\ell^{(i_\ell)}\in\skein(M).
\end{equation}

When $q$ is a root of unity, threading by Chebyshev polynomials has special properties.
The Chebyshev polynomial $T_n(x)\in\ints[x]$ is the family of polynomials given by
\begin{equation}
T_0(x)=2,\qquad T_1(x)=x,\qquad T_n(x)=xT_{n-1}(x)-T_{n-2}(x),\quad n\ge 2.
\end{equation}
Suppose $\zeta\in\cx$ such that $\zeta^4$ is a primitive $N$-th root of unity. Let
$\varepsilon=\zeta^{N^2}$. Then threading links by $T_N$ defines a map
\begin{equation}
\Phi_\zeta:\skein_\varepsilon(M)\to\skein_\zeta(M).
\end{equation}
This is the \term{Chebyshev homomorphism} introduced by \cite{BW:I}, and it can
be defined in the context of punctured bordered surfaces. Note the framed power
$a^{(N)}$ makes sense for a (framed, stated) arc $a$ over $\surface$.

\begin{theorem}[{\cite[Corollary~4.7]{BL}}]\label{thm-Cheb-surface}
Fix a punctured bordered surface $\surface$. Then there exists a unique algebra
homomorphism $\Phi_\zeta:\skein_\varepsilon(\surface)\to\skein_\zeta(\surface)$ (using
the $\cup$-product) such that
\begin{equation}
\Phi_\zeta(a)=a^{(N)}\quad\text{if $a$ is an arc},\qquad
\Phi_\zeta(K)=K^{(T_N)}\quad\text{if $K$ is a knot}.
\end{equation}

Let $e$ be an ideal arc on $\surface$ and $\surface_e$ be the splitting of $\surface$
along $e$. Then $\Phi_\zeta$ is compatible with splitting homomorphisms, that is,
the following diagram commutes.
\begin{equation}
\begin{tikzcd}
\skein_\varepsilon(\surface) \arrow[r,"\cut_e"] \arrow[d,"\Phi_\zeta"] &
\skein_\varepsilon(\surface_e) \arrow[d,"\Phi_\zeta"]\\
\skein_\zeta(\surface) \arrow[r,"\cut_e"] &
\skein_\zeta(\surface_e)
\end{tikzcd}
\end{equation}
\end{theorem}

We next recall the analogous map, the \term{Frobenius homomorphism} on a quantum torus.
It is given by
\begin{equation}
\Phi^\qtorus_\zeta:\qtorus_\varepsilon\langle x_1,\dotsc,x_k\rangle
\to\qtorus_\zeta\langle x_1,\dotsc,x_k\rangle,\qquad
\Phi^\qtorus_\zeta(x_i)=x_i^N.
\end{equation}
This map does not require $\zeta$ to be a root of unity, and $N$ could be arbitrary
(but still with $\varepsilon=\zeta^{N^2}$). However, it is only relevant to us when
$\zeta$ and $N$ are as before.

\begin{lemma}
\label{lemma-qglue-Frob}
The Frobenius homomorphism for $\qtorus\langle\qz,qz''\rangle$ maps the Lagrangian
to the Lagrangian.

The Frobenius homomorphism for $\qtorus(\triang)$ maps the edge equations to the
edge equations.

Thus, the Frobenius homomorphisms induce maps
\begin{equation}
\varphi_\zeta:\qglue_\varepsilon(\lantern)\to\qglue_\zeta(\lantern),\qquad
\varphi_\zeta:\qglue_\varepsilon(\triang)\to\qglue_\zeta(\triang).
\end{equation}
\end{lemma}

\begin{proof}
For the Lagrangian equations, we use the following form of the $q$-binomial theorem
at roots of
unity: if $XY=\omega YX$ for a root of unity $\omega$ with order $N$, then the
quantum binomial theorem and the vanishing of the quantum binomial at roots of unity
implies that
\begin{equation}
(X+Y)^N=X^N+Y^N \,.
\end{equation}
Apply this to $X=\qz^{-2}$, $Y=(\qz'')^2$, and $\omega=\zeta^4$, we get
\begin{equation}
\Phi^\qtorus_\zeta(\qz^{-2}+(\qz'')^2)=\qz^{-2N}+(\qz'')^{2N}
=(\qz^{-2}+(\qz'')^2)^N=1.
\end{equation}

For the edge equations, this is a standard algebraic calculation. First, a simple
check shows that
\begin{equation}
\label{eq-zeta-power}
(-\zeta^2)^N=-\varepsilon^2.
\end{equation}
Now, $\Phi^\qtorus_\zeta$ sends a Weyl-ordered monomial to its $N$-th power, so
$(\text{edge})=-\varepsilon^2$ is sent to $(\text{edge})^N=-\varepsilon^2$, which is
the correct equation by \eqref{eq-zeta-power}.

After verifying that all calculations are compatible with multiplication on the
appropriate sides, we get the map $\varphi_\zeta$.
\end{proof}

As mentioned in Remark~\ref{remark-ichoice}, here we actually need to be careful about
the choice of the 4-th root of unity $\iunit$. By \eqref{eq-zeta-power},
\begin{equation}
\label{eq-iprime}
\iunit'=(\iunit\zeta)^N/\varepsilon
\end{equation}
is a primitive 4-th root of unity, but it could be either $\pm\iunit$. We use $\iunit$
for $q=\zeta$ and $\iunit'$ for $q=\varepsilon$.

In the rest of the section we give a proof of Theorem~\ref{thm.frob}. 
Recall that to split a surface along a closed curve, we remove 3 points and split
along the three ideal arcs. Let $\surface'_\triang$ be $\surface_\triang$ with three
points removed for each $A$-curve, and let $\cut'_A$ be the splitting along ideal
$A$-arcs for $\surface'_\triang$. Then by definition \eqref{eq-splitcr}, the following
diagram commutes.
\begin{equation}
\begin{tikzcd}
\skein(\surface'_\triang) \arrow[r,"\cut'_A"] \arrow[d,two heads] &
\bigotimes_{j=1}^N\skein(\lantern_j) \arrow[d,two heads] \\
\skein(\surface_\triang) \arrow[r,"\cut_A"] &
\bigotimes_{j=1}^N\skeincr(\lantern_j)
\end{tikzcd}
\end{equation}

Expand the diagram \eqref{eq-Cheb-Frob} using the definition of the quantum trace map
and replace $\cut_A$ by $\cut'_A$ using the diagram above. Then we need to show that
the following diagram commutes.
\begin{equation}
\begin{tikzcd}
\skein_\varepsilon(M) \arrow[d,"\Phi_\zeta"] &
\skein_\varepsilon(\surface'_\triang) \arrow[l,two heads] \arrow[r,"\cut'_A"]
\arrow[d,"\Phi_\zeta"] &
\bigotimes_{j=1}^N\skein_\varepsilon(\lantern_j) \arrow[r,two heads]
\arrow[d,"\otimes\Phi_\zeta"] &
\bigotimes_{j=1}^N\qglue_\varepsilon(\lantern_j) \arrow[r,two heads]
\arrow[d,"\otimes\varphi_\zeta"] &
\qglue_\varepsilon(\triang) \arrow[d,"\varphi_\zeta"] \\
\skein_\zeta(M) &
\skein_\zeta(\surface'_\triang) \arrow[l,two heads] \arrow[r,"\cut'_A"] &
\bigotimes_{j=1}^N\skein_\zeta(\lantern_j) \arrow[r,two heads] &
\bigotimes_{j=1}^N\qglue_\zeta(\lantern_j) \arrow[r,two heads] &
\qglue_\zeta(\triang)
\end{tikzcd}
\end{equation}
The first and the last squares commute essentially by definition. The second square
commutes by Theorem~\ref{thm-Cheb-surface}.
For the remaining third square, we show that it commutes on each tensor factor. By
definition, if $a$ is a simple arc diagram whose endpoints are on different boundary
edges, then the framed power $a^{(N)}$ agrees with the algebraic power $a^N$. Then by the
choice \eqref{eq-iprime}, $\Phi_\zeta$ respects the corner reductions \eqref{eq-bad} and
\eqref{eq-crdef}, so we can replace $\skein_\bullet(\lantern_j)$ with
$\skeincr_\bullet(\lantern_j)$ in the diagram. Then by Lemma~\ref{lemma-cr-span},
$\skeincr_\bullet(\lantern)$ is spanned by products of arcs with endpoints on different
boundary edges, so $\Phi_\zeta$ is given by the algebraic power again. Moreover, each arc
is either a generator of $\qglue(\lantern_j)$ or $0$, so $\Phi_\zeta$ matches
$\varphi_\zeta$ on this spanning set. This proves the commutativity of the third square.

This completes the proof of Theorem~\ref{thm.frob}.
\qed


\section{Computational aspects}
\label{sec.compute}

In this section we discuss how to compute the quantum trace map ~\eqref{defqtr}
using the methods of \texttt{SnapPy}, following the pioneering ideas of Thurston.
This method uses as input triangulations of 3-manifolds~\cite{snappy}. The data of
a triangulation $\calT$ encoded in \texttt{SnapPy} allows one to define the 
lantern surface and to describe the peripheral curves on it. In a future addition
to these methods, one may add framed links in the complement of a knot and trace them
to the ideal triangulation of the knot complement, and then to the lantern surface
giving an effective computation of the quantum trace map. 

\subsection{Cusp diagram and dual surface}

In this section we discuss how to obtain the lantern surface from an ideal triangulation
encoded in \texttt{SnapPy}.

Given an ideal triangulation $\triang$ of $M$, the cell decomposition $\dtrunc$ using
double truncation restricts to a cell decomposition $\dtrunc_\partial$ of the boundary
$\partial M$. From Figure~\ref{fig-doubly}, we see that $\partial M$ is glued from
the vertex hexagons $(\gamma^{-1}\beta)^3$, the base polygons of edge prisms $\gamma^n$,
and the circles $\beta^2$.

In \texttt{SnapPy}, the tetrahedra in ideal triangulations are only truncated once at the
vertices but not the edges. In this case, the boundary $\partial M$ is given by a
triangulation $\lambda$. We can easily truncate $\lambda$ to obtain $\dtrunc_\partial$.
The triangles of $\lambda$ become the vertex hexagons of $\dtrunc_\partial$, and the
corners of the triangles correspond to $\gamma$ edges. The truncated edges of $\lambda$
can be doubled into $\beta$ edges if needed. See Figure~\ref{fig-cusp-trunc}.

\begin{figure}[htpb!]
\centering
\begin{tikzpicture}
\tikzmath{\r=2;\s=0.4;\t=\r/2-\s;}
\draw (0,0) -- (30:\r) -- (90:\r) -- (150:\r) -- cycle -- (90:\r);
\fill[blue,opacity=0.5] (0,0) -- (30:\s)
arc[radius=\s,start angle=30,end angle=150] (90:\r) -- ++(-30:\s)
arc[radius=\s,start angle=-30,end angle=-150] (30:\r) -- ++(150:\s)
arc[radius=\s,start angle=150,delta angle=60] (150:\r) -- ++(30:\s)
arc[radius=\s,start angle=30,delta angle=-60];
\draw[blue] (30:\s) arc[radius=\s,start angle=30,end angle=150] (90:\r) ++(-30:\s)
arc[radius=\s,start angle=-30,end angle=-150] (30:\r) ++(150:\s)
arc[radius=\s,start angle=150,delta angle=60] (150:\r) ++(30:\s)
arc[radius=\s,start angle=30,delta angle=-60];
\begin{scope}[red]
\draw (90:{\r/2}) circle[x radius=\t/3,y radius=\t];
\draw (30:\s) [rotate=-60]
arc[x radius=\t/3,y radius=\t,start angle=-90,delta angle=-180];
\draw (90:\r) ++ (-150:\s) [rotate=-60]
arc[x radius=\t/3,y radius=\t,start angle=90,delta angle=-180];
\draw (150:\s) [rotate=60]
arc[x radius=\t/3,y radius=\t,start angle=-90,delta angle=180];
\draw (90:\r) ++ (-30:\s) [rotate=60]
arc[x radius=\t/3,y radius=\t,start angle=90,delta angle=180];
\end{scope}
\end{tikzpicture}
\caption{The truncated triangulation of the cusp.}\label{fig-cusp-trunc}
\end{figure}
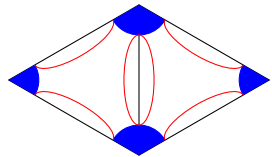

Now compare $\partial M$ with the dual surface $\surface_\triang$, shown in
Figure~\ref{fig-smooth-skel}. They share the vertex hexagons and the $\beta^2$ circles.
While $\partial M$ has the bases of edge prisms, $\surface_\triang$ has the sides of the
edge prisms, which are the annuli between parallel $B$-curves. This shows that
$\partial M$ can be obtained from $\surface_\triang$ by surgery along all $B$-curves.

Conversely, we can reverse the surgery. This also recovers the decorations on
$\surface_\triang$ with some additional bookkeeping which \texttt{SnapPy} does. By
definition, the reverse surgery truncates the vertices of $\lambda$ and pair the
boundary circles by gluing in annuli. Clearly, the $B$-curves are just the boundary
circles after truncation. The $A$-curves are homotopic to $(\alpha\beta)^3$ in the
1-skeleton $\dtrunc^{(1)}$. $\alpha$ edges are just transverse arcs of the inserted
annuli, and $\beta$ edges are homotopic to truncated edges of $\lambda$. Note the ends
of edges in $\lambda$ correspond to ends of edges in $\triang$. \texttt{SnapPy}
keeps track of this correspondence, which determines which triples of truncated edges
connect to form the $A$-circles.

It is convenient to describe $\surface_\triang$ by splitting into lanterns so that the
diagrams are planar and that the quantum trace can be calculated. This can be obtained
from $\lambda$ by
\begin{enumerate}
\item
  truncating the vertices to create the $\gamma$ edges,
\item
  splitting along the truncated edges, resulting in the vertex hexagons, and
\item
  gluing the hexagons along $\gamma$ edges, which become standard arcs on the lantern.
\end{enumerate}
In step (2), the pairings of the truncated edges after splitting needs to be recorded
so that the lanterns can glue back together.

\texttt{SnapPy} also calculates the meridian and longitude, given as cycles in the
dual 1-skeleton of $\lambda$. This means the curves intersect edges of $\lambda$
transversely and do not go through vertices of $\lambda$. Then the steps above also
draw the curves on the lanterns.

The second author has implemented the lantern surface of a \texttt{SnapPy} triangulation
as a python module. 

\subsection{Example: The $4_1$ knot}
\label{sub.41}

In this section we illustrate the discussion of the previous section, and the
computation of the quantum trace map with the default triangulation $\triang$ of
the $4_1$ knot with isometry signature \texttt{cPcbbbiht\_BaCB}. This is a triangulation
with two tetrahedra $T_0$ and $T_1$ shown in Figure~\ref{fig-41-triang} using the
convention from Figure~\ref{fig-label-tetra}. There are 4 face pairings labeled A,B,C,D.

\begin{figure}[htpb!]
\centering
\begin{tikzpicture}
\tikzmath{\r=1;\s=0.6;}
\draw (0,0) circle(\r) node[anchor=-150]{3};
\draw (0,0) -- (-150:\r)node[anchor=30]{1} (0,0) -- (-30:\r)
node[anchor=150]{0} (0,0) -- (90:\r)node[above]{2};
\path[-o-={0.5}{>}] (0,0) -- (90:\r);
\path[-o-={0.5}{>}] (-30:\r) arc[radius=\r,start angle=-30,delta angle=120];
\path[-o-={0.5}{>}] (-30:\r) arc[radius=\r,start angle=-30,delta angle=-120];
\path[-o-={0.5}{>>}] (-30:\r) -- (0,0);
\path[-o-={0.5}{>>}] (-150:\r) -- (0,0);
\path[-o-={0.5}{>>}] (-150:\r) arc[radius=\r,start angle=-150,delta angle=-120];
\path (0,-\r-0.5)node{$T_0$} (30:\s)node{B} (150:\s)node{A} (-90:\s)
node{C} (45:\r+0.4)node{D};
\begin{scope}[xshift={\r*4cm}]
\draw (0,0) circle(\r) node[anchor=-150]{3};
\draw (0,0) -- (-150:\r)node[anchor=30]{1} (0,0) -- (-30:\r)
node[anchor=150]{0} (0,0) -- (90:\r)node[above]{2};
\path[-o-={0.5}{>}] (90:\r) -- (0,0);
\path[-o-={0.5}{>}] (90:\r) arc[radius=\r,start angle=90,delta angle=-120];
\path[-o-={0.5}{>}] (-150:\r) arc[radius=\r,start angle=-150,delta angle=120];
\path[-o-={0.5}{>>}] (-30:\r) -- (0,0);
\path[-o-={0.5}{>>}] (-150:\r) -- (0,0);
\path[-o-={0.5}{>>}] (-150:\r) arc[radius=\r,start angle=-150,delta angle=-120];
\path (0,-\r-0.5)node{$T_1$} (30:\s)node{D} (150:\s)node{A} (-90:\s)node{C} (45:\r+0.4)
node{B};
\end{scope}
\end{tikzpicture}
\caption{\texttt{SnapPy} triangulation of the $4_1$ knot.}
\label{fig-41-triang}
\end{figure}
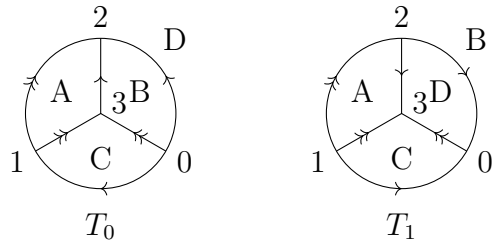

The triangulation $\lambda$ of the cusp calculated by \texttt{SnapPy} is shown in
Figure~\ref{fig-41-cusp} after truncation. The triangle near vertex $i$ in $T_j$ is
indexed by $i+4j$. The edge numbering is also included for the lantern diagram later.
The labels 1,2 at truncated vertices refer to the edges of $\triang$ corresponding to
single or double arrows in Figure~\ref{fig-41-triang}. The meridian $\mu$ and
longitude $\lambda$ (simplified by twisting twice around $\mu$) of the $4_1$ knot are
also included.

\begin{figure}[htpb!]
\centering
\begin{tikzpicture}
\tikzmath{\r=2;\s=0.4;\t=0.6;\e=0.8;\c=\r-\e;}
\path (0,0)coordinate(b-1) +(120:\r)coordinate(t-1)
++(\r,0)coordinate(b-2) +(120:\r)coordinate(t-2)
++(\r,0)coordinate(b1) +(120:\r)coordinate(t1)
++(\r,0)coordinate(b2) +(120:\r)coordinate(t2)
++(\r,0)coordinate(b-1r) +(120:\r)coordinate(t-1r);
\draw[red] (b-1) -- (b-1r) -- (t-1r) -- (t-1) -- cycle
(b-1) -- (t-2) -- (b-2) -- (t1) -- (b1) -- (t2) -- (b2) -- (t-1r);
\path[red,every node/.style={font=\scriptsize,inner sep=1pt}]
(b-1) +(120:\e)node[anchor=30]{9} +(60:\c)node[anchor=-30]{4}
(b-2) +(120:\e)node[anchor=30]{8} +(60:\c)node[anchor=-30]{10}
(b1) +(120:\e)node[anchor=30]{5} +(60:\c)node[anchor=-30]{11}
(b2) +(120:\e)node[anchor=30]{0} +(60:\c)node[anchor=-30]{3}
(b-1r) +(120:\e)node[anchor=30]{9}
foreach \p/\a in {b-1/below,t-1/above} {(\p) ++(\r/2,0)node[\a]{2} ++(\r,0)
  node[\a]{1} ++(\r,0)node[\a]{7} ++(\r,0)node[\a]{6}};
\draw[blue,every node/.style={circle,fill=white,font=\small,inner sep=2pt}]
(t-1)node{$1$} ++(-60:\s) -- ++(60:\s)
(b-1)node{$1$} ++(120:\s) -- ++(0:\s) -- ++(-60:\s)
(t-1r)node{$1$} ++(-60:\s) -- ++(180:\s) -- ++(120:\s)
(b-1r)node{$1$} ++(120:\s) -- ++(-120:\s)
(t-2)node{$2$} ++(180:\s) -- ++(-60:\s) -- ++(0:\s) -- ++(60:\s)
(b-2)node{$2$} ++(180:\s) -- ++(60:\s) -- ++(0:\s) -- ++(-60:\s)
(t1)node{$1$} ++(180:\s) -- ++(-60:\s) -- ++(0:\s) -- ++(60:\s)
(b1)node{$1$} ++(180:\s) -- ++(60:\s) -- ++(0:\s) -- ++(-60:\s)
(t2)node{$2$} ++(180:\s) -- ++(-60:\s) -- ++(0:\s) -- ++(60:\s)
(b2)node{$2$} ++(180:\s) -- ++(60:\s) -- ++(0:\s) -- ++(-60:\s);
\path (barycentric cs:t-1r=1,b2=1,b-1r=1)node{0}
(barycentric cs:t2=1,b1=1,b2=1)node{1}
(barycentric cs:t1=1,b-2=1,b1=1)node{2}
(barycentric cs:t-2=1,b-1=1,b-2=1)node{3}
(barycentric cs:t-1=1,b-1=1,t-2=1)node{6}
(barycentric cs:t-2=1,b-2=1,t1=1)node{7}
(barycentric cs:t1=1,b1=1,t2=1)node{4}
(barycentric cs:t2=1,b2=1,t-1r=1)node{5};
\draw[blue,-o-={0.6}{>}] (b-1r) ++(120:\s) -- ++(-120:\s);
\draw[blue,-o-={0.6}{>}] (b1) ++(60:\s) -- ++(-60:\s);
\draw[blue,-o-={0.7}{>>}] (b-2) ++(120:\s) -- ++(-120:\s);
\draw[blue,-o-={0.7}{>>}] (b2) ++(180:\s) -- ++(60:\s);
\draw[->] (b2) ++(-\t,-0.3) -- ++(0,0.3)
arc[radius={\r-\t},start angle=0,end angle=60]
arc[radius=\t,start angle=-120,end angle=-180] -- ++(0,0.3)
node[right]{$\mu$};
\draw[->] (b-1) ++(120:{\r/2}) ++(-0.3,0) -- ++({4*\r+0.6},0)
node[anchor=-135]{$\lambda-2\mu$};
\end{tikzpicture}
\caption{Cusp diagram of the $4_1$ knot.}\label{fig-41-cusp}
\end{figure}
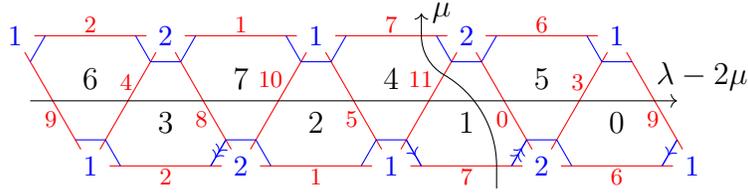

Now we can convert the cusp diagram to the dual surface $\surface_\triang$ by gluing
the truncated vertices of the cusp. To make sure the truncated edges of the cusp
match correctly to form the $A$-circles, we mark the arcs around the 01 and 13 edges
in $T_0$ with arrows in Figure~\ref{fig-41-triang}, which can be continued to obtain
all other pairings.

Next, we split $\surface_\triang$ along $A$-circles to obtain lantern diagrams. We
can determine all pairings of the blue edges using the markings in
Figure~\ref{fig-41-cusp}
or using additional \texttt{SnapPy} data not shown in the figure. After splitting the
red edges and gluing the blue edges, we obtain Figure~\ref{fig-41-lantern}. Here, the
orientations of the blue standard arcs are induced from Figure~\ref{fig-41-cusp}.

\begin{figure}[htpb!]
\centering
\begin{tikzpicture}[baseline=(ref.base)]
\tikzmath{\r=1.5;\s=0.3;\d=0.45;\e=\r*0.6;}
\fill[gray!20] (0,0)circle(\r+0.8);
\begin{scope}[thick,blue,radius=\r]
\draw[-o-={0.5}{>}] (90:\r) arc[start angle=90,delta angle=120];
\draw[-o-={0.5}{>}] (-150:\r) arc[start angle=-150,delta angle=120];
\draw[-o-={0.5}{>>}] (90:\r) arc[start angle=90,delta angle=-120];
\draw[-o-={0.5}{>}] (0,0) -- (90:\r);
\draw[-o-={0.5}{>>}] (0,0) -- (-150:\r);
\draw[-o-={0.5}{>>}] (-30:\r) -- (0,0);
\end{scope}
\draw[thick,red,fill=white,radius=\s] (90:\r)circle node{1} (-30:\r)
circle node{2} (-150:\r)circle node{3} (0,0)circle node{0};
\path[red,font=\scriptsize] (90:\r) +(90:\d)node{3} +(-45:\d)node{4} +(-135:\d)node{5}
(-30:\r) +(-30:\d)node{6} +(-165:\d)node{7} +(105:\d)node{8}
(-150:\r) +(-150:\d)node{9} +(75:\d)node{10} +(-15:\d)node{11}
(0,0) +(-90:\d)node{0} +(150:\d)node{1} +(30:\d)node{2};
\path (30:\e)node{3} (150:\e)node{2} (-90:\e)node{1} (45:\r+0.5)node{0};
\path (-90:\r+1.2)node{$\lantern_0$};
\begin{scope}[xshift=\r*4cm]
\fill[gray!20] (0,0)circle(\r+0.8);
\begin{scope}[thick,blue,radius=\r]
\draw[-o-={0.5}{>}] (-30:\r) arc[start angle=-30,delta angle=-120];
\draw[-o-={0.5}{>}] (-150:\r) arc[start angle=-150,delta angle=-120];
\draw[-o-={0.5}{>>}] (90:\r) arc[start angle=90,delta angle=-120];
\draw[-o-={0.5}{>}] (90:\r) -- (0,0);
\draw[-o-={0.5}{>>}] (0,0) -- (-150:\r);
\draw[-o-={0.5}{>>}] (-30:\r) -- (0,0);
\end{scope}
\draw[thick,red,fill=white,radius=\s] (90:\r)circle node{5} (-30:\r)
circle node{6} (-150:\r)circle node{7} (0,0)circle node{4};
\path[red,font=\scriptsize] (90:\r) +(90:\d)node{11} +(-45:\d)node{10} +(-135:\d)
node{9} (-30:\r) +(-30:\d)node{7} +(-165:\d)node{6} +(105:\d)node{8}
(-150:\r) +(-150:\d)node{5} +(75:\d)node{4} +(-15:\d)node{3}
(0,0) +(-90:\d)node{0} +(150:\d)node{2} +(30:\d)node{1};
\path (30:\e)node{7} (150:\e)node{6} (-90:\e)node{5} (45:\r+0.5)node{4};
\path (-90:\r+1.2)node{$\lantern_1$};
\end{scope}
\node (ref) at (0,0){\phantom{$-$}};
\end{tikzpicture}
\caption{Lantern diagram for the $4_1$ knot.}\label{fig-41-lantern}
\end{figure}
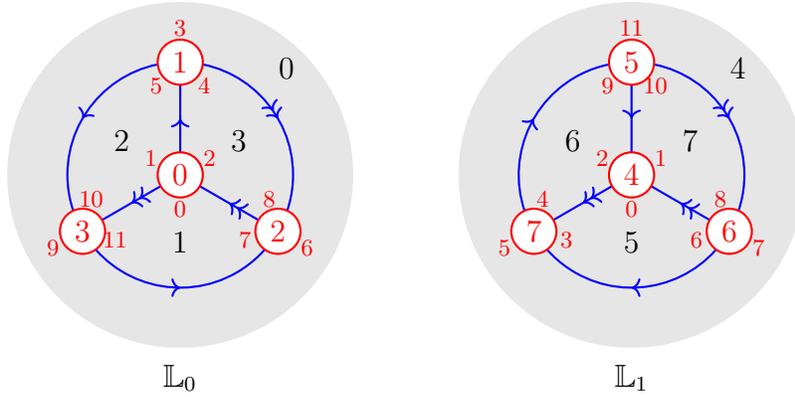

The meridian $\mu$ and the longitude $\lambda$ are shown in the subsequent
Figures~\ref{fig-41-meri} and \ref{fig-41-long}, where the labels are omitted to
avoid clutter.

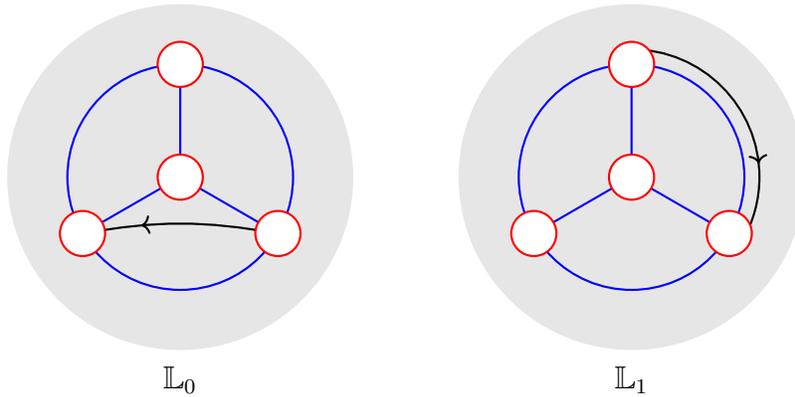
\begin{figure}[htpb!]
\centering
\begin{tikzpicture}[baseline=(ref.base),thick]
\tikzmath{\r=1.5;\s=0.3;\d=0.45;\e=\r*0.6;}
\fill[gray!20] (0,0)circle(\r+0.8);
\draw[blue] (0,0)circle(\r) foreach \t in {90,-30,-150} {(0,0) -- (\t:\r)};
\draw[-o-={0.7}{>}] (-30:\r) to[out=170,in=10] (-150:\r);
\draw[red,fill=white] (90:\r)circle(\s) (-30:\r)circle(\s) (-150:\r)
circle(\s) (0,0) circle(\s);
\path (-90:\r+1.2)node{$\lantern_0$};
\begin{scope}[xshift=\r*4cm]
\fill[gray!20] (0,0)circle(\r+0.8);
\draw[blue] (0,0)circle(\r) foreach \t in {90,-30,-150} {(0,0) -- (\t:\r)};
\draw[-o-={0.7}{>}] (90:\r+0.2) arc[radius=\r+0.2,start angle=90,end angle=-30];
\draw[red,fill=white] (90:\r)circle(\s) (-30:\r)circle(\s) (-150:\r)
circle(\s) (0,0)circle(\s); \path (-90:\r+1.2)node{$\lantern_1$};
\end{scope}
\node (ref) at (0,0){\phantom{$-$}};
\end{tikzpicture}
\caption{Meridian of the $4_1$ knot.}\label{fig-41-meri}
\end{figure}

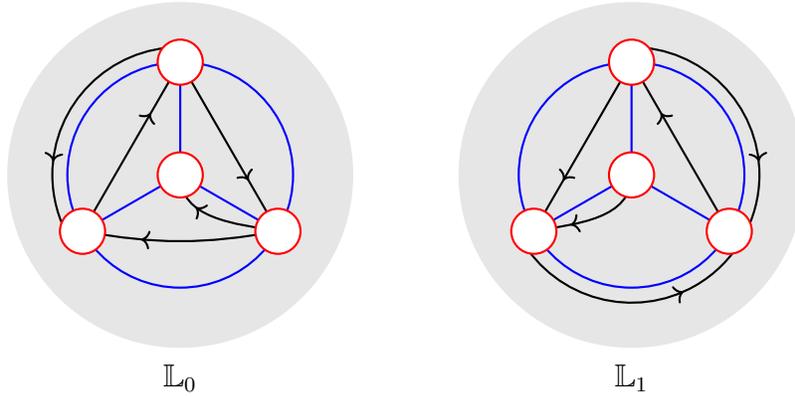
\begin{figure}[htpb!]
\centering
\begin{tikzpicture}[baseline=(ref.base),thick]
\tikzmath{\r=1.5;\s=0.3;\d=0.45;\e=\r*0.6;}
\fill[gray!20] (0,0)circle(\r+0.8);
\draw[blue] (0,0)circle(\r) foreach \t in {90,-30,-150} {(0,0) -- (\t:\r)};
\draw[-o-={0.7}{>}] (-30:\r) to[out=-170,in=-10] (-150:\r);
\draw[-o-={0.7}{>}] (-30:\r) to[out=170,in=-90] (0,0);
\draw[-o-={0.7}{>}] (90:\r) -- (-30:\r);
\draw[-o-={0.7}{>}] (-150:\r) -- (90:\r);
\draw[-o-={0.7}{>}] (90:\r+0.2) arc[radius=\r+0.2,start angle=90,delta angle=120];
\draw[red,fill=white] (90:\r)circle(\s) (-30:\r)circle(\s) (-150:\r)
circle(\s) (0,0)circle(\s);
\path (-90:\r+1.2)node{$\lantern_0$};
\begin{scope}[xshift=\r*4cm]
\fill[gray!20] (0,0)circle(\r+0.8);
\draw[blue] (0,0)circle(\r) foreach \t in {90,-30,-150} {(0,0) -- (\t:\r)};
\draw[-o-={0.7}{>}] (0,0) to[out=-90,in=10] (-150:\r);
\draw[-o-={0.7}{>}] (-30:\r) -- (90:\r);
\draw[-o-={0.7}{>}] (90:\r) -- (-150:\r);
\draw[-o-={0.7}{>}] (90:\r+0.2) arc[radius=\r+0.2,start angle=90,delta angle=-120];
\draw[-o-={0.7}{>}] (-150:\r+0.2) arc[radius=\r+0.2,start angle=-150,delta angle=120];
\draw[red,fill=white] (90:\r)circle(\s) (-30:\r)circle(\s) (-150:\r)
circle(\s) (0,0)circle(\s);
\path (-90:\r+1.2)node{$\lantern_1$};
\end{scope}
\node (ref) at (0,0){\phantom{$-$}};
\end{tikzpicture}
\caption{Longitude of the $4_1$ knot.}\label{fig-41-long}
\end{figure}

Finally, we calculate $\qtr_\triang(\mu)$ as an example. To interpret
Figure~\ref{fig-41-meri} as the diagram of $\mu\in\skein(\surface_\triang)$ after
splitting, we need to add punctures to the boundaries of $\lantern_j$. The exact
positions of the punctures do not matter, as long as they are consistent between
lanterns. We make the choice that the punctures on edges 7 and 11 are close to edges
6 and 9 respectively. This makes the arc $\mu^1=\mu\cap\lantern_1$
standard, but the arc $\mu^0=\mu\cap\lantern_0$ is not. In the state sum of the
splitting, if the endpoints of $\mu^1$ are assigned opposite states, then the diagram
is zero in $\qglue(\lantern_1)$. Thus,
\begin{equation}
\qtr_\triang(\mu)=\sum_{s\in\{\pm\}}\mu^0_{ss}\otimes\mu^1_{ss}.
\end{equation}
To express this in terms of quantized shape variables, we adopt the convention that
$\qz_j$ is associated to the $01$ edge in $T_j$. In lantern diagrams, it corresponds to
the standard arc connecting boundary circles 2,3 (shifted by $4j$). In this convention,
$\mu^1_{ss}=(\qz''_1)^{-s}$. For $\mu^0_{ss}$, we can use Corollary~\ref{cor-tw} to twist
it into standard arcs, giving $\mu^0_{ss}=\qz_0^s$. Therefore,
\begin{equation}
\qtr_\triang(\mu)=\qz_0(\qz''_1)^{-1}+\qz_0^{-1}\qz''_1.
\end{equation}

\subsection*{Acknowledgements}

The authors wish to thank Francis Bonahon, Tudor Dimofte and Thang L\^e for
enlightening conversations.


\appendix

\section{Proof of Theorem~\ref{thm-stdA}}
\label{sec-stdA-iso}

Throughout this section, $\surface$ is a surface with triangular boundary. All
diagrams are positively ordered in this section, so the height order will be omitted.

The only nontrivial part is that $\phi_a$ is bijective. We do so by identifying a
basis of $\skeincr(\annulus)$. Draw the standard annulus as in Figure~\ref{fig-ann-cut},
where the top and the bottom are identified. The dashed (ideal) arc is determined by
the choice of $a$, and it will be used to describe the basis.

\begin{figure}[htpb!]
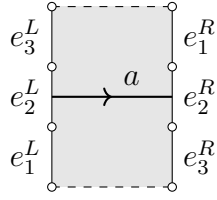

\centering
$\begin{anndiag}{2}{3}
\draw (0,0) -- (0,3);
\foreach \y in {0,1,2,3}
\draw[fill=white] (0,\y)circle(0.07);
\foreach \i in {1,2,3}
\path (2,{3.5-\i}) node[right]{$e^R_\i$} (0,{\i-0.5}) node[left]{$e^L_\i$};
\draw[thick] (0,1.5) -- (2,1.5);
\path[tips,thick,->] (0,1.5) -- (1,1.5) node[above right]{$a$};
\end{anndiag}$
\caption{Labelling the standard annulus.}\label{fig-ann-cut}
\end{figure}

We say a tangle diagram on $\annulus$ is \term{normal} if it is in general position
with the dashed arc. A normal isotopy is an isotopy within the class of normal diagrams.
Then two normal isotopy classes of tangle diagrams represent isotopic tangles if and
only if they are related by the Reidemeister moves and the following moves
\begin{equation}
\label{eq-stdA-dashcup}
\text{(RII'):}\quad
\begin{linkdiag}[1]
\fill[gray!20] (0,0) rectangle (1,2);
\draw[dashed] (0,0) -- (1,0) (0,2) -- (1,2);
\draw[thick] (0,0.3)..controls +(0.4,0)..(0.4,0)
(0,0.5)..controls +(0.6,0)..(0.6,0)
(0.4,2)..controls +(0,-0.4) and +(0,-0.4)..(0.6,2)
(0,0.75) -- +(1,0) (0,1.25) -- +(1,0) (0.8,1)node[rotate=90]{...};
\end{linkdiag}
=
\begin{linkdiag}[1]
\fill[gray!20] (0,0) rectangle (1,2);
\draw[dashed] (0,0) -- (1,0) (0,2) -- (1,2);
\draw[thick] (0,0.3)..controls +(0.6,0) and +(0.6,0)..(0,0.5)
(0,0.75) -- +(1,0) (0,1.25) -- +(1,0) (0.8,1)node[rotate=90]{...};
\end{linkdiag}.\qquad
\text{(RIII'):}\quad
\begin{linkdiag}[1]
\fill[gray!20] (0,0) rectangle (1,2);
\draw[dashed] (0,0) -- (1,0) (0,2) -- (1,2);
\draw[thick] (0,0.75) -- +(1,0) (0,1.25) -- +(1,0) (0.8,1)node[rotate=90]{...};
\begin{knot}[end tolerance=2pt]
\strand[thick] (0,0.3)..controls +(0.4,0)..(0.4,0)
(0,0.5)..controls +(0.6,0)..(0.6,0);
\strand[thick] (0,1.4)..controls +(1,0) and +(0,-0.5)..(0.4,2);
\strand[thick](0,1.6)..controls +(0.4,0)..(0.6,2);
\end{knot}
\end{linkdiag}
=
\begin{linkdiag}[1]
\fill[gray!20] (0,0) rectangle (1,2);
\draw[dashed] (0,0) -- (1,0) (0,2) -- (1,2);
\draw[thick] (0,0.75) -- +(1,0) (0,1.25) -- +(1,0) (0.8,1)node[rotate=90]{...};
\begin{knot}[end tolerance=2pt]
\strand[thick] (0,1.7)..controls +(0.4,0)..(0.4,2)
(0,1.5)..controls +(0.6,0)..(0.6,2);
\strand[thick] (0,0.4)..controls +(0.4,0)..(0.6,0);
\strand[thick] (0,0.6)..controls +(1,0) and +(0,0.5)..(0.4,0);
\end{knot}
\end{linkdiag}.
\end{equation}
They are related by isotopies of $\annulus$, but not as normal diagrams. Let $A$ be
the set of normal isotopy classes of tangle diagrams on $\annulus$. Then the set of
isotopy class of tangles is equivalent to the quotient of $A$ by (framed) Reidemeister
moves and RII', RIII'.

Let $RA$ denote the free $R$-module spanned by $A$. $RA$ has a product structure by
stacking with coefficients just like \eqref{eq-new-prod}. Of course, we only need the
left action induced by the product.

Recall the linear diamond lemma from \cite{SW}. Let $S\subset A\times RA$, and write
$a\to b$ for $(a,b)\in S$. Elements of $S$ are called reduction rules on $A$. Define
$(S)$ as the $R$-submodule spanned by $\{a-b\in RA\mid a\to b\}$.

For $x,y\in RA$ with $x=\sum_i r_ia_i$ where $r_i\in R$ and $a_i\in A$, write
$x\rightsquigarrow y$ if there exists a $j$ such that $r_j\ne0$ and
$y=r_jb+\sum_{i\ne j}r_ia_i$ where $a_j\to b$. Finally, let $x\succeq y$ if $x=y$ or if
there exists a sequence $x_0\rightsquigarrow x_1\dots\rightsquigarrow x_n$ with
$x_0=x$, $x_n=y$. Clearly, $(S)$ contains all $x-y$ with $x\succeq y$.

A set of reduction rules on $A$ is called \term{locally confluent} if for any $a\in A$
and $b,c\in RA$ with $a\to b, a\to c$, there exists $v\in RA$ such that $b\succeq v$
and $c\succeq v$. An element $a\in A$ is \term{irreducible} if the only $b\in RA$
with $a\succeq b$ is $b=a$. Let $A_{\textnormal{irr}}\subset A$ be the subset of
irreducible elements.

The following result is a combination of Theorems~2.2 and 2.3 in \cite{SW}.

\begin{theorem}
\label{thm-diamond}
Let $\deg:A\to J$ be a map to a well-ordered set $J$, and let
$A_j=\{a\in A:\deg(a)<j\}$. Suppose $S$ is a locally confluent set of reduction rules
on $A$ such that for any $a\to b$ with $\deg(a)=j$, $b\in RA_j$. Then the map
\begin{equation}
RA_{\textnormal{irr}}\embed RA\onto RA/(S)
\end{equation}
is an isomorphism.
\end{theorem}

Define the following set $S$ of reduction rules on $A$.
\begin{subequations}
\begin{align}
\label{eq-stdA-skein}
\begin{linkdiag}
\fill[gray!20] (-0.1,0)rectangle(1.1,1);
\begin{knot}
\strand[thick] (1,1)--(0,0); \strand[thick] (0,1)--(1,0);
\end{knot}
\end{linkdiag}
&\to q
\begin{linkdiag}
\fill[gray!20] (-0.1,0)rectangle(1.1,1);
\draw[thick] (0,0)..controls (0.5,0.5)..(0,1);
\draw[thick] (1,0)..controls (0.5,0.5)..(1,1);
\end{linkdiag}
+q^{-1}
\begin{linkdiag}
\fill[gray!20] (-0.1,0)rectangle(1.1,1);
\draw[thick] (0,0)..controls (0.5,0.5)..(1,0);
\draw[thick] (0,1)..controls (0.5,0.5)..(1,1);
\end{linkdiag},\qquad
\begin{linkdiag}
\fill[gray!20] (0,0)rectangle(1,1); \draw[thick] (0.5,0.5)circle(0.3);
\end{linkdiag}
\to(-q^2-q^{-2})
\begin{linkdiag}
\fill[gray!20] (0,0)rectangle(1,1);
\end{linkdiag}.\\
\label{eq-stdA-return}
\relarc[]{+}{+}&\to0,\qquad \relarc[]{-}{-}\to0,\qquad
\relarc[]{+}{-}\to q^{-1/2}\relempty.\\
\label{eq-stdA-stex}
\relacross[]{-}{+}
&\to q^2\relacross[]{+}{-}+q^{-1/2}
\begin{linkdiag}
\fill[gray!20] (0,0)rectangle(1,1);\draw (1,0)--(1,1);
\draw [thick] (0,0.7)..controls(0.8,0.7) and (0.8,0.3)..(0,0.3);
\end{linkdiag}.\\
\label{eq-stdA-handle}
\begin{linkdiag}[1]
\fill[gray!20] (0,0) rectangle (1,2);
\draw[dashed] (0,0) -- (1,0) (0,2) -- (1,2);
\draw[thick] (0,0.4)..controls +(0.5,0)..(0.5,0)
(0,1.6)..controls +(0.5,0)..(0.5,2)
(0,0.75) -- +(1,0) (0,1.25) -- +(1,0) (0.8,1)node[rotate=90]{...};
\end{linkdiag}
&\to
\begin{linkdiag}[1]
\fill[gray!20] (0,0) rectangle (1,2);
\draw[dashed] (0,0) -- (1,0) (0,2) -- (1,2);
\begin{knot}
\strand[thick] (0,0.75) -- +(1,0) (0,1.25) -- +(1,0);
\strand[thick] (0,0.4)..controls +(0.8,0) and +(0.8,0)..(0,1.6);
\end{knot}
\path (0.8,1)node[rotate=90]{...};
\end{linkdiag}.\\
\label{eq-stdA-twsimple}
\begin{linkdiag}[0.9]
\fill[gray!20] (0,0) rectangle (1,1.8); \draw (1,0) -- (1,1.8);
\draw[fill=white] (1,0)circle(0.07) (1,1.3)circle(0.07);
\draw[thick] (0,0.5) -- +(1,0) (0,1) -- +(1,0);
\draw[thick] (0,0.2) -- +(1,0) \stnode{+};
\path (0.8,0.75)node[rotate=90]{...} (1,0.65)node[right]{$e^\ast_3$};
\end{linkdiag}
&\to \iunit q^{\frac{1}{2}(d^\ast_2-d^\ast_3+3)}
\begin{linkdiag}[0.9]
\fill[gray!20] (0,0) rectangle (1,1.8); \draw (1,0) -- (1,1.8);
\draw[fill=white] (1,0)circle(0.07) (1,1.3)circle(0.07);
\begin{knot}
\strand[thick] (0,0.5) -- +(1,0) (0,1) -- +(1,0);
\strand[thick] (0,0.2)..controls+(0.5,0) and +(-0.5,0)..(1,1.6)\stnode{-};
\end{knot}
\path (0.8,0.75)node[rotate=90]{...};
\end{linkdiag},\qquad
\begin{linkdiag}[0.9]
\fill[gray!20] (0,0) rectangle (1,1.8); \draw (1,0) -- (1,1.8);
\draw[fill=white] (1,0)circle(0.07) (1,1.3)circle(0.07);
\draw[thick] (0,0.5) -- +(1,0) (0,1) -- +(1,0);
\draw[thick] (0,1.6) -- +(1,0) \stnode{-};
\path (0.8,0.75)node[rotate=90]{...} (1,0.65)node[right]{$e^\ast_2$};
\end{linkdiag}
\to -\iunit q^{-\frac{1}{2}(d^\ast_1-d^\ast_2+3)}
\begin{linkdiag}[0.9]
\fill[gray!20] (0,0) rectangle (1,1.8); \draw (1,0) -- (1,1.8);
\draw[fill=white] (1,0)circle(0.07) (1,1.3)circle(0.07);
\begin{knot}
\strand[thick] (0,0.5) -- +(1,0) (0,1) -- +(1,0);
\strand[thick] (0,1.6)..controls+(0.5,0) and +(-0.5,0)..(1,0.2)\stnode{+};
\end{knot}
\path (0.8,0.75)node[rotate=90]{...};
\end{linkdiag}.\\
\label{eq-stdA-twupm}
\begin{linkdiag}[0.9]
\fill[gray!20] (0,0) rectangle (1,1.8); \draw (1,0) -- (1,1.8);
\draw[fill=white] (1,0)circle(0.07) (1,1.3)circle(0.07);
\draw[thick] (0,0.5) -- +(1,0) (0,1) -- +(1,0);
\draw[thick] (0,0.2) -- +(1,0) \stnode{-};
\path (0.8,0.75)node[rotate=90]{...} (1,0.65)node[right]{$e^\ast_3$};
\end{linkdiag}
&\to \iunit q^{\frac{1}{2}(d^\ast_3-d^\ast_2+3)}
\begin{linkdiag}[0.9]
\fill[gray!20] (0,0) rectangle (1,1.8); \draw (1,0) -- (1,1.8);
\draw[fill=white] (1,0)circle(0.07) (1,1.3)circle(0.07);
\begin{knot}
\strand[thick] (0,0.5) -- +(1,0) (0,1) -- +(1,0);
\strand[thick] (0,0.2)..controls +(0.5,0) and +(-0.5,0)..(1,1.6) \stnode{+};
\end{knot}
\path (0.8,0.75)node[rotate=90]{...};
\end{linkdiag}
+q^{\frac{1}{2}(2d^\ast_1-d^\ast_2-d^\ast_3+1)}
\begin{linkdiag}[0.9]
\fill[gray!20] (0,0) rectangle (1,1.8); \draw (1,0) -- (1,1.8);
\draw[fill=white] (1,0)circle(0.07) (1,1.3)circle(0.07);
\begin{knot}
\strand[thick] (0,0.5) -- +(1,0) (0,1) -- +(1,0);
\strand[thick] (0,0.2)..controls+(0.5,0) and +(-0.5,0)..(1,1.6)\stnode{-};
\end{knot}
\path (0.8,0.75)node[rotate=90]{...};
\end{linkdiag}.\\
\label{eq-stdA-twdownp}
\begin{linkdiag}[0.9]
\fill[gray!20] (0,0) rectangle (1,1.8); \draw (1,0) -- (1,1.8);
\draw[fill=white] (1,0)circle(0.07) (1,1.3)circle(0.07);
\draw[thick] (0,0.5) -- +(1,0) (0,1) -- +(1,0);
\draw[thick] (0,1.6) -- +(1,0) \stnode{+};
\path (0.8,0.75)node[rotate=90]{...} (1,0.65)node[right]{$e^\ast_2$};
\end{linkdiag}
&\to q^{\frac{1}{2}(2d^\ast_3-d^\ast_1-d^\ast_2-1)}
\begin{linkdiag}[0.9]
\fill[gray!20] (0,0) rectangle (1,1.8); \draw (1,0) -- (1,1.8);
\draw[fill=white] (1,0)circle(0.07) (1,1.3)circle(0.07);
\begin{knot}
\strand[thick] (0,0.5) -- +(1,0) (0,1) -- +(1,0);
\strand[thick] (0,1.6)..controls +(0.5,0) and +(-0.5,0)..(1,0.2) \stnode{+};
\end{knot}
\path (0.8,0.75)node[rotate=90]{...};
\end{linkdiag}
-\iunit q^{\frac{1}{2}(d^\ast_1-d^\ast_2-3)}
\begin{linkdiag}[0.9]
\fill[gray!20] (0,0) rectangle (1,1.8); \draw (1,0) -- (1,1.8);
\draw[fill=white] (1,0)circle(0.07) (1,1.3)circle(0.07);
\begin{knot}
\strand[thick] (0,0.5) -- +(1,0) (0,1) -- +(1,0);
\strand[thick] (0,1.6)..controls+(0.5,0) and +(-0.5,0)..(1,0.2)\stnode{-};
\end{knot}
\path (0.8,0.75)node[rotate=90]{...};
\end{linkdiag}.
\end{align}
\end{subequations}
Here, $d^\ast_i$ is the grading on the edge $e^\ast_i$ for $\ast=L,R$ and $i=1,2,3$.
Note \eqref{eq-stdA-skein}--\eqref{eq-stdA-stex} are taken from the proof of
\cite[Theorem~2.11]{Le:decomp} (cited here as Theorem~\ref{thm-basis}).
\eqref{eq-stdA-handle}--\eqref{eq-stdA-twdownp} are from Corollary~\ref{cor-tw} and
Lemma~\ref{lemma-crhandle}.

\begin{lemma}
\label{lemma-RA-quot}
The natural map $RA\to\skeincr(\annulus)$ descends to an isomorphism
$RA/(S)\to\skeincr(\annulus)$.
\end{lemma}

\begin{proof}
The set of isotopy classes of tangles on $\annulus$ is the quotient of $A$ by
Reidemeister moves and RII', RIII'. Therefore, $\skeincr(\annulus)=RA/I$ where $I$ is
the left ideal generated by Reidemeister moves, RII', RIII',
\eqref{eq-skein}--\eqref{eq-stex}, \eqref{eq-bad}, and \eqref{eq-crdef}. Thus, we
need to prove that $I=(S)$.

First we check that $(S)$ is a left ideal in $RA$. This would be completely trivial
if the product is simply stacking. The coefficients introduced by \eqref{eq-new-prod}
requires some work for \eqref{eq-stdA-twsimple}--\eqref{eq-stdA-twdownp}, but it is
still straightforward.

By construction, all elements of $(S)$ are $0$ in $\skeincr(\annulus)$. Thus,
$(S)\subset I$. To show the reverse inclusion, we need to reduce the generators of
$I$ to $0$ using $S$.
For Reidemeister moves, this is well-known; see e.g., \cite{Kau}. For RII', apply
\eqref{eq-stdA-handle} and Reidemeister moves. For RIII', after resolving the crossing,
the two sides are related by an RII' move. \eqref{eq-skein}, \eqref{eq-loop},
\eqref{eq-arcs}, and \eqref{eq-stex} correspond to
\eqref{eq-stdA-skein}--\eqref{eq-stdA-stex}.
For \eqref{eq-bad} and \eqref{eq-crdef}, they
each come in 6 flavors, one for each puncture of $\annulus$. The opposite sides have the
same reduction rules, so we only need to consider 3. Direct calculation shows they indeed
have the correct reductions.
\end{proof}

\begin{lemma}
\label{lemma-confl}
The set $S$ of reduction rules is locally confluent.
\end{lemma}

\begin{proof}
Define the support of the reduction moves to be the following closed subsets of
$\annulus$:
\begin{itemize}
\item
  the crossing in the first move of \eqref{eq-stdA-skein},
\item
  the disk bounded by the loop in the second move of \eqref{eq-stdA-skein},
\item
  the disk bounded by the returning arc and part of $\partial\annulus$ in each move
  of \eqref{eq-stdA-return},
\item
  the interval between the endpoints in \eqref{eq-stdA-stex},
\item
  the nontrivial loop through the intersection with the dashed arc in
  \eqref{eq-stdA-handle}, and
\item
  the boundary triangle in each move of \eqref{eq-stdA-twsimple}--\eqref{eq-stdA-twdownp}.
\end{itemize}
Let $a\in A$ be a diagram where two reduction rules $s_1,s_2\in S$ apply. If the
support of the reductions are disjoint, then it is easy to see that $s_2$ applies to
the $s_1$-reduction and vice versa. Moreover, when both reductions are applied, the
result is independent of the order. Therefore, we only need to consider the case when
the supports intersect. In addition, \cite[Lemma~2.10]{Le:decomp} already proved that
\eqref{eq-stdA-skein}--\eqref{eq-stdA-stex} are confluent. Thus, the remaining cases are
\begin{enumerate}
\item
  \eqref{eq-stdA-return} or \eqref{eq-stdA-stex} with
  \eqref{eq-stdA-twsimple}--\eqref{eq-stdA-twdownp}.
\item
  Among \eqref{eq-stdA-twsimple}--\eqref{eq-stdA-twdownp}.
\end{enumerate}

In (1), if the lower endpoint of \eqref{eq-stdA-return} or \eqref{eq-stdA-stex} is
not the one modified by \eqref{eq-stdA-twsimple}--\eqref{eq-stdA-twdownp}, then the
confluence is the same as the disjoint support case. There are 8 cases left to check:
when the components of \eqref{eq-stdA-return} or \eqref{eq-stdA-stex} are the lowest
on $e^\ast_1$ or $e^\ast_3$. Here we check one case. The rest are similar.
\begin{align*}
\begin{linkdiag}[1]
\fill[gray!20] (0,0) rectangle (1,2); \draw (1,0) -- (1,2);
\draw[fill=white] (1,0)circle(0.07) (1,1.5)circle(0.07);
\draw[thick] (0,0.8) -- +(1,0) (0,1.2) -- +(1,0);
\draw[thick] (1,0.2)\stnode{-}..controls +(-0.5,0) and +(-0.5,0)..(1,0.5)\stnode{+};
\path (0.25,1)node{...} (1,1)node[right]{$e^\ast_3$};
\end{linkdiag}
&\xrightarrow{\eqref{eq-stdA-twupm}}
\iunit q^{\frac{1}{2}(d^\ast_3-d^\ast_2+3)}
\begin{linkdiag}[1]
\fill[gray!20] (0,0) rectangle (1,2); \draw (1,0) -- (1,2);
\draw[fill=white] (1,0)circle(0.07) (1,1.5)circle(0.07);
{\begin{knot}[consider self intersections=no splits]
\strand[thick] (0,0.8) -- +(1,0) (0,1.2) -- +(1,0);
\strand[thick] (1,1.8)\stnode{+}..controls +(-0.4,0) and +(0.4,0)..
(0.4,0.2)..controls +(-0.4,0) and +(-1,0)..(1,0.5)\stnode{+};
\end{knot}}
\path (0.25,1)node{...};
\end{linkdiag}
+q^{\frac{1}{2}(2d^\ast_1-d^\ast_2-d^\ast_3+1)}
\begin{linkdiag}[1]
\fill[gray!20] (0,0) rectangle (1,2); \draw (1,0) -- (1,2);
\draw[fill=white] (1,0)circle(0.07) (1,1.5)circle(0.07);
{\begin{knot}[consider self intersections=no splits,flip crossing=3]
\strand[thick] (0,0.8) -- +(1,0) (0,1.2) -- +(1,0);
\strand[thick] (1,1.8)\stnode{-}..controls +(-0.4,0) and +(0.4,0)..
(0.4,0.2)..controls +(-0.4,0) and +(-1,0)..(1,0.5)\stnode{+};
\end{knot}}
\path (0.25,1)node{...};
\end{linkdiag}\\
&\xrightarrow{\text{RI}}(-q^{-3})
(\iunit q^{\frac{1}{2}(d^\ast_3-d^\ast_2+3)})
\begin{linkdiag}[1]
\fill[gray!20] (0,0) rectangle (1,2); \draw (1,0) -- (1,2);
\draw[fill=white] (1,0)circle(0.07) (1,1.5)circle(0.07);
{\begin{knot}[consider self intersections=no splits]
\strand[thick] (0,0.8) -- +(1,0) (0,1.2) -- +(1,0);
\strand[thick] (1,1.8)\stnode{+}..controls +(-0.6,0) and +(-0.6,0)..(1,0.5)\stnode{+};
\end{knot}}
\path (0.25,1)node{...};
\end{linkdiag}
+(\text{...})
\begin{linkdiag}[1]
\fill[gray!20] (0,0) rectangle (1,2); \draw (1,0) -- (1,2);
\draw[fill=white] (1,0)circle(0.07) (1,1.5)circle(0.07);
{\begin{knot}[consider self intersections=no splits]
\strand[thick] (0,0.8) -- +(1,0) (0,1.2) -- +(1,0);
\strand[thick] (1,1.8)\stnode{-}..controls +(-0.6,0) and +(-0.6,0)..(1,0.5)\stnode{+};
\end{knot}}
\path (0.25,1)node{...};
\end{linkdiag}\\
&\xrightarrow{\eqref{eq-stdA-twsimple}}
(-\iunit q^{\frac{1}{2}(d^\ast_3-d^\ast_2-3)})
(\iunit q^{\frac{1}{2}((d^\ast_2+1)-(d^\ast_3+1)+3)})
\begin{linkdiag}[1]
\fill[gray!20] (0,0) rectangle (1,2); \draw (1,0) -- (1,2);
\draw[fill=white] (1,0)circle(0.07) (1,1.3)circle(0.07);
\begin{knot}
\strand[thick] (0,0.5) -- +(1,0) (0,1) -- +(1,0);
\strand[thick] (1,1.8)\stnode{+}..controls +(-0.8,0) and
+(-0.3,0)..(0.5,0.2)..controls +(0.4,0) and +(-0.4,0)..(1,1.5)\stnode{-};
\end{knot}
\path (0.1,0.75)node[rotate=90]{...};
\end{linkdiag}
+(\text{...})
\begin{linkdiag}[1]
\fill[gray!20] (0,0) rectangle (1,2); \draw (1,0) -- (1,2);
\draw[fill=white] (1,0)circle(0.07) (1,1.3)circle(0.07);
\begin{knot}
\strand[thick] (0,0.5) -- +(1,0) (0,1) -- +(1,0);
\strand[thick] (1,1.8)\stnode{-}..controls +(-0.8,0) and
+(-0.3,0)..(0.5,0.2)..controls +(0.4,0) and +(-0.4,0)..(1,1.5)\stnode{-};
\end{knot}
\path (0.1,0.75)node[rotate=90]{...};
\end{linkdiag}.
\end{align*}
After RII and \eqref{eq-stdA-return}, this agrees with the alternative reduction by
applying \eqref{eq-stdA-return} first.

In (2), the overlaps are of the following form
\begin{equation}
\begin{anndiag}{1}{2.7}
\begin{scope}[thick]
\foreach \y in {0.4,0.7,1.3,1.6,2.2,2.5}
\draw[thick] (0,\y) -- +(1,0);
\draw[thick] (0,2) -- +(1,0)\stnode{\mu} (0,0.2) -- +(1,0)\stnode{\nu};
\end{scope}
\foreach \y in {0.55,1.45,2.35}
\path (0.2,\y)node{...};
\end{anndiag}.
\end{equation}
Similar to the disjoint overlap case, the two reductions can be applied in any order,
but the coefficients and the diagrams are different depending on the order. After some
careful calculations, the reduction agree after some applications of
\eqref{eq-stdA-skein}--\eqref{eq-stdA-stex}
on $e^\ast_2$.
\end{proof}

\begin{proof}[Proof that $\phi_a$ is bijective]
Define the degree $\deg=(\deg_1,\deg_2):A\to\nats^2$. $\deg_1$ is the total number of
intersections with $e^\ast_{1,3}$ and the dashed arc. $\deg_2$ is the degree defined in
the proof of \cite[Lemma~2.10]{Le:decomp}. We do not need the exact definition. All we
need are the following results:
\begin{itemize}
\item
  $\deg_2$ is decreased by \eqref{eq-stdA-skein}--\eqref{eq-stdA-stex}.
\item
  $\deg_2(a)=0$ $\Leftrightarrow$ $a$ is irreducible with respect to
  \eqref{eq-stdA-skein}--\eqref{eq-stdA-stex} $\Leftrightarrow$ $a$ represents a
  basis element of $\skein(\annulus)$ in Theorem~\ref{thm-basis} with $\ori$ positive.
\end{itemize}
We use the lexicographical order on $\nats^2$. Since
\eqref{eq-stdA-skein}--\eqref{eq-stdA-stex} do not increase $\deg_1$, the first fact
above shows they decrease $\deg$. For the rest, they all decrease $\deg_1$ by $1$, so
they decrease $\deg$ as well. Thus, $\deg$ satisfies the condition in Theorem
~\ref{thm-diamond}. The local confluence is proved in Lemma~\ref{lemma-confl}. By
Theorem~\ref{thm-diamond}, $RA/(S)$, naturally identified with $\skeincr(\annulus)$
by Lemma~\ref{lemma-RA-quot}, is a free $R$-module with a basis given by the
irreducible diagrams. The irreducible diagrams are basis elements of $\skein(\annulus)$
with no intersection with the dashed arc or endpoints on $e^\ast_{1,3}$. Such diagrams
are contained in a neighborhood of $a$ after normal isotopy, and they match the basis
of $\skein(\bigon)\cong\Oq$ via $\phi_a$. Therefore, $\phi_a$ is bijective.
\end{proof}


\bibliographystyle{hamsalpha}
\bibliography{biblio}
\end{document}